\documentclass[a4paper,11pt,twoside]{amsart}
\usepackage{amssymb,amsthm,amsmath,amstext}
\usepackage{mathrsfs}  
\usepackage{bm}        
\usepackage{multirow}
\usepackage[top=1in,bottom=1in,left=1in,right=1in]{geometry}

\usepackage[all]{xy}
\usepackage{rotating}

\numberwithin{equation}{section}
\theoremstyle{plain}
\newtheorem{theorem}{Theorem}[subsection]
\newtheorem{prop}[theorem]{Proposition}
\newtheorem{lemma}[theorem]{Lemma}
\newtheorem{corollary}[theorem]{Corollary}
\newtheorem{conjecture}[theorem]{Conjecture}
\newtheorem{question}[theorem]{Question}
\newtheorem{problem}[theorem]{Problem}

\theoremstyle{definition}
\newtheorem{definition}[theorem]{Definition}

\theoremstyle{remark}
\newtheorem{remark}[theorem]{Remark}
\newtheorem{example}[theorem]{Example}

\numberwithin{equation}{section}

\newcommand{\sheaf}[1]{\mathscr{#1}}

\newcommand{\cat}[1]{\mathsf{#1}}

\newcommand{\isom}{\cong}

\newcommand{\equi}{\simeq}

\newcommand{\im}{\operatorname{im}}

\newcommand{\OO}{\sheaf{O}}
\newcommand{\D}{\cat{D}}
\newcommand{\Db}{\D^{\mathrm{b}}}

\newcommand{\Br}{\mathrm{Br}}
\newcommand{\NS}{\mathrm{NS}}
\newcommand{\Pic}{\mathrm{Pic}}

\newcommand{\rk}{{\rm rk}}

\newcommand{\End}{\mathrm{End}}
\newcommand{\Hom}{\mathrm{Hom}}
\newcommand{\Spec}{\mathrm{Spec}}

\newcommand{\Hilb}{{\rm Hilb}}
\renewcommand{\dim}{{\rm dim}\,}

\newcommand{\tensor}{\otimes}

\newcommand{\Alb}{\mathrm{Alb}}
\newcommand{\alb}{\mathrm{alb}}
\newcommand{\homological}{\mathrm{hom}}

\newcommand{\CH}{\mathrm{CH}}

\newcommand{\Gm}{\mathbb{G}_\mathrm{m}}
\newcommand{\mapto}[1]{\xrightarrow{#1}}

\newcommand{\sod}[1]{\langle #1 \rangle}

\DeclareMathOperator{\Rep}{\mathrm{rdim}}
\DeclareMathOperator{\Repcat}{\mathrm{rdim}}
\DeclareMathOperator{\coRepcat}{\mathrm{rcodim}}

\newcommand{\Ext}{{\rm Ext}}

\newcommand{\ur}{\mathrm{ur}}
\newcommand{\dR}{\mathrm{dR}}
\newcommand{\Hur}{H_{\ur}}
\newcommand{\Het}{H_{\et}}
\newcommand{\HdR}{H_{\dR}}
\newcommand{\HB}{H_{\mathrm{B}}}
\newcommand{\Local}{\mathsf{Local}}
\newcommand{\Ab}{\mathsf{Ab}}

\newcommand{\Sch}{\mathsf{Sch}}
\newcommand{\Frac}{\mathrm{Frac}}

\newcommand{\pr}{0}

\newcommand{\cal}{\mathscr}

\newcommand{\kc}{{\cal C}}

\newcommand{\ko}{{\cal O}}

\newcommand{\kt}{{\cal T}}

\newcommand{\NN}{\mathbb{N}}
\newcommand{\ZZ}{\mathbb{Z}}
\newcommand{\QQ}{\mathbb{Q}}

\newcommand{\CC}{\mathbb{C}}

\newcommand{\FF}{\mathbb{F}}
\newcommand{\LL}{\mathbb{L}}
\newcommand{\PP}{\mathbb{P}}
\renewcommand{\AA}{\mathbb{A}}

\newcommand{\RHom}{\mathbf{R}\Hom}

\newcommand{\cC}{\mathcal{C}}

\newcommand{\et}{\mathrm{\acute{e}t}}

\newcommand{\linedef}[1]{\textit{#1}}

\usepackage[backref=page]{hyperref}
\hypersetup{pdftitle={Cycles, derived categories, and rationality},pdfauthor={Asher Auel and Marcello Bernardara}} 
\hypersetup{colorlinks=true,linkcolor=blue,anchorcolor=blue,citecolor=blue}

\begin{document}

\title[Cycles, derived categories, and rationality]{Cycles, derived categories, and rationality}

\begin{abstract} 
Our main goal is to give a sense of recent developments in the
(stable) rationality problem from the point of view of unramified
cohomology and $0$-cycles as well as derived categories and
semiorthogonal decompositions, and how these perspectives intertwine
and reflect each other.  In particular, in the case of algebraic
surfaces, we explain the relationship between Bloch's conjecture,
Chow-theoretic decompositions of the diagonal, categorical
representability, and the existence of phantom subcategories of the
derived category.
\end{abstract}

\author{Asher Auel}
\address{Department of Mathematics\\ %
Yale University \\ %
10 Hillhouse Avenue \\ %
New Haven, CT 06511}
\email{asher.auel@yale.edu}

\author{Marcello Bernardara}
\address{Institut de Math\'ematiques de Toulouse \\ %
Universit\'e Paul Sabatier \\ %
118 route de Narbonne \\ %
31062 Toulouse Cedex 9\\ %
France}
\email{marcello.bernardara@math.univ-toulouse.fr}

\maketitle

\setcounter{tocdepth}{1}
\tableofcontents

In this text, we explore two potential measures of rationality.  The
first is the universal triviality of the Chow group of $0$-cycles,
which is related to the Chow-theoretic decomposition of the diagonal
of Bloch and Srinivas.  A powerful degeneration method for obstructing
the universal triviality of the Chow group of $0$-cycles, initiated by
Voisin, and developed by Colliot-Th\'el\`ene and Pirutka, combines
techniques from singularity theory and unramified cohomology and has
led to a recent breakthrough in the stable rationality problem.

The second is categorical representability, which is defined by the
existence of semiorthogonal decompositions of the derived category
into components whose dimensions can be bounded. We will give a
precise definition of this notion and present many examples, as well
as motivate why one should expect categorical representability in
codimension $2$ for rational varieties.

Furthermore, we would like to explore how the Chow-theoretic and
derived categorical measures of rationality can contrast and reflect
each other.  One of the motivating topics in this circle of ideas is
the relationship, for complex surfaces, between Bloch's conjecture,
the universal triviality of the Chow group of $0$-cycles, and the
existence of phantoms in the derived categories.  Another motivating
topic is the rationality problem for cubic fourfolds and its
connections between derived categories, Hodge theory, as well as
Voisin's recent results on the universal triviality of the Chow group
of $0$-cycles for certain loci of special cubic fourfolds.

{\small
\subsection*{Acknowledgements.}
This text started as notes that were collected together in preparation
for lectures given at the Spring 2016 FRG mini-workshop
\textit{Unramified cohomology, derived categories, and rationality} at
the University of Utah in Salt Lake City, February 26--28, 2015.  We
are very grateful to the organizers for the invitation to lecture
there and for their great hospitality. These notes were greatly
improved on the basis of preliminary talks (which represent the first
half of the document) given by M.~Musta\c{t}a, I.~Shipman, T.~de~Fernex, and Ch.~Hacon.  Discussing with them and other participants
during the meeting made for a very fruitful experience.  We thank
J.-L.~Colliot-Th\'el\`ene and the anonymous referees for numerous helpful comments on the text.
}

\section{Preliminaries on Chow groups}

Let $k$ be an arbitrary field.  By scheme, we will mean a separated
$k$-scheme of finite type.
In this section, we give a quick introduction to the Chow group of
algebraic cycles on a scheme up to rational equivalence. See Fulton's book
\cite{fulton:intersection_theory} for more details.

Denote by $Z_i(X)$ the free $\ZZ$-module generated by all
$i$-dimensional closed integral subschemes of $X$. The elements of
$Z_i$ are called \linedef{algebraic $i$-cycles}. We will
also employ the codimension notation $Z^i(X)=Z_{n-i}(X)$ when $X$ is smooth of pure dimension $n$.  The
\linedef{support} of an $i$-cycle $\sum_n a_n [V_n]$ is the union of
the closed subschemes $V_i$ in $X$; it is \linedef{effective} if $a_n
>0$ for all $n$.

Given an $(i+1)$-dimensional closed integral subscheme $W$ of $X$, and
a closed integral subscheme $V \subset W$ of codimension 1, we denote
by $\ko_{W,V}$ the local ring of $W$ at the generic point of $V$; it
is a local domain of dimension 1 whose field of fractions is the
function field $k(W)$.  For a nonzero function $f \in \OO_{W,V}$, we
define the \linedef{order} of vanishing $\mathrm{ord}_V(f)$ of $f$
along $V$ to be the length of the $\OO_{W,V}$-module $\OO_{W,V}/(f)$.
The order extends uniquely to a homomorphism $\mathrm{ord}_V :
k(W)^\times \to \ZZ$.  If $W$ is normal, when $\mathrm{ord}_V$
coincides with the usual discrete valuation on $k(W)$ associated to
$V$.  We also define the \linedef{divisor} of a rational function $f
\in k(W)^\times$ as an $i$-cycle on $X$ given by
$$
[\mathrm{div}(f)] = \sum_{V \subset W} \mathrm{ord}_V(f) [V],
$$
where the sum is taken over all closed integral subscheme $V \subset
W$ of codimension 1.  An $i$-cycle $z$ on $X$ is \linedef{rationally
equivalent} to 0 if there exists a finite number of closed integral
$(i+1)$-dimensional subschemes $W_j \subset X$ and rational functions
$f_j \in K(W_j)^\times$ such that $z = \sum_j [\mathrm{div}(f_j)]$ in
$Z_i(X)$. Note that the set of $i$-cycles rationally equivalent to $0$
forms a subgroup of $Z_i(X)$ since
$[\mathrm{div}(f^{-1})]=-[\mathrm{div}(f)]$ for any rational function
$f \in K(W)^\times$.  Denote the associated equivalence relation on
$Z_i(X)$ by $\sim_{\text{rat}}$.  The \linedef{Chow group} of
$i$-cycles on $X$ is the quotient group $\CH_i(X) =
Z_i(X)/\!\sim_{\text{rat}}$ of algebraic $i$-cycles modulo rational
equivalence.

\subsection{Morphisms}

Let $f: X \to Y$ be a proper morphism of schemes. Define a
\linedef{push-forward} map on cycles $f_* : Z_i (X) \to Z_i(Y)$
additively as follows.  For a closed integral subscheme $V \subset X$
define
$$
f_*([V]) = \left\lbrace \begin{array}{ll}
                    0 & \text{ if } \mathrm{dim}(f(V)) < \mathrm{dim}(V) \\
                    \deg(V / f(V))\, [f(V)] & \text{ if } \mathrm{dim}(f(V)) = \mathrm{dim}(V)
                   \end{array}\right.
$$
where $\deg(V/f(V))$ is the degree of the finite extension of function
fields $k(V)/k(f(V))$ determined by $f$.  This map respects rational
equivalence and hence induces a push-forward map on Chow groups  $f_* : \CH_i(X)
\to \CH_i(Y)$.

Let $f: X \to Y$ be a flat morphism of relative dimension $r$. Define
a \linedef{pull-back} map $f^*: Z_i(Y) \to Z_{i+r}(X)$ additively as
follows.  For a closed integral subscheme $V \subset Y$ define
$$
f^* ([V]) = [f^{-1}(V)].
$$
This map respects rational equivalence and hence induces a
pull-back map on Chow groups $f^* : \CH_i(Y) \to \CH_{i+r}(X)$.

A special case of proper push forward is given by considering a closed
immersion $\iota: Z \to X$.  Letting $j: U \to X$ be the open
complement of $Z$, we note that $j$ is flat. There is an exact
excision sequence
$$
\CH_i(Z) \mapto{\,\iota_*\,} \CH_i(X)
\mapto{\,j^*\,} \CH_i(U) \longrightarrow 0
$$
which comes from an analogous exact sequence on the level of cycles.

Moreover, we have the following compatibility between the proper
push-forward and flat pull-back.  Given a cartesian diagram
  $$\xymatrix{
 X' \ar[r]^{g'} \ar[d]_{f'} & X \ar[d]^f \\
 Y' \ar[r]^{g} & Y,}$$
where $g$ is flat of relative dimension $r$ and $f$ is proper,
then $g^* f_* = f'_* g'^*$ as maps $\CH_i(X) \to \CH_{i+r}(Y')$.
 
A third natural map between Chow groups is the \linedef{Gysin
map}. Given a regular closed embedding $\iota: X \to Y$ of codimension
$r$, one can define a map $\iota^! : \CH_i(Y) \to \CH_{i-r}(X)$.  The
precise definition is much more involved, see~\cite[\S 5.2;
6.2]{fulton:intersection_theory}, and in particular, is not induced
from a map on cycles.  In particular, this map factors through
$\CH_i(N)$, where $s: N \to Y$ is the normal cone of $X$ in $Y$, and
can be described by the composition of a Gysin map $s^!$ for vector
bundles and the inverse of the pull back $f^*$ by $f: N \to X$.

The Gysin map allows one to write the \linedef{excess intersection
formula} for a regular closed embedding $\iota: X \to Y$ of
codimension $r$
$$
\iota^! \iota_* (\alpha) = c_r(h^* N_{X/Y}) \cap \alpha,
$$
for any cycle $\alpha$ in $Z_*(Y)$, where $h: N_{X/Y} \to X$ is the
normal bundle of $X$ in $Y$ and $c_r$ denotes the $r$-th Chern class.
In particular, this shows that $\iota^! \iota_* = 0$ whenever
$N_{X/Y}$ is trivial.

We now define the Gysin map for any local complete intersection (lci)
morphism $f: X \to Y$. Consider the factorization of $f$ as
$$
X \mapto{~\iota~} P \mapto{~h~} Y,
$$
where $\iota$ is a regular embedding of codimension $r$ and $h$ is
smooth of relative dimension $m$.  Then we can define the Gysin map
$$
f^! = \iota^! \circ h^* : \CH_i(Y) \mapto{~h^*\,} \CH_{i-m}(P) \mapto{~\iota^!\,} \CH_{i-m-r}(X).
$$
Such a map is independent on the chosen factorization of $f$ and
coincides with the flat pull-back $f^*$ whenever $f$ is flat.  As a
relevant example, any morphism $f: X \to Y$ between smooth $k$-schemes
is lci (indeed, $f$ factors $X \to X \times Y \to Y$ into the regular
graph embedding followed by the smooth projection morphism), hence
induces a Gysin map $f^! : \CH^i(Y) \to \CH^i(X)$.  Finally, even when
$f$ is not flat, we often denote $f^*=f^!$, so that $f^*$ is defined
for any lci morphism.

\subsection{Intersections}
\label{subsec:intersections}

Let $X$ be a smooth $k$-scheme of pure dimension $n$.  Then the Chow
group admits an \linedef{intersection product} as follows.  For closed
integral subschemes $V \subset X$ and $W \subset X$ of codimension $i$
and $j$, respectively, define
$$
[V].[W] = \Delta^![V \times W] \in \CH^{i+j}(X)
$$
where $\Delta : X \to X \times X$ is the diagonal morphism, a regular
embedding of codimension $n$.  This induces a bilinear map $\CH^i(X)
\times \CH^j(X) \to \CH^{i+j}(X)$, which makes $\CH(X) =
\oplus_{i\geq 0} \CH^i(X)$ into a commutative graded ring with
identity $[X] \in \CH^0(X)$.  Gysin maps between smooth $k$-schemes
are then ring homomorphisms for the intersection product.

One can understand the intersection product in terms of literal
intersections of subschemes, via moving lemmas, see
\cite[\S11.4]{fulton:intersection_theory} for a discussion of the
technicalities involved.  Assuming that $X$ is smooth and
quasi-projective, given closed integral subschemes $V \subset X$ and
$W \subset X$ of codimension $i$ and $j$, respectively, the moving
lemma says that one can replace $V$ by a rationally equivalent cycle
$V' = \sum_l a_l [V_l]$ so that $V'$ and $W$ meet properly, i.e., $V_l
\cap W$ has codimension $i+j$ for all $l$.  Then one can define
$[V].[W] = \sum_l a_l[V_l \cap W]$.  A more refined moving lemma is
then required to show that the rational equivalence class of this
product is independent of the choice of cycle $V'$.

One easy moving lemma that we will need to use often is the moving
lemma for 0-cycles: given a smooth quasi-projective $k$-scheme $X$, an
open dense subscheme $U \subset X$, and a 0-cycles $z$ on $X$, there
exists a 0-cycle $z'$ on $X$ rationally equivalent to $z$ such that
the support of $z'$ is contained in $U$.  See, e.g.,
\cite[\S2.3]{fulton:singular}, \cite[p.~599]{colliot:finitude} for a
reference to the classical moving lemma, which implies this.

\subsection{Correspondences}
\label{subsec:correspondences}

Let $X$ and $Y$ be smooth $k$-schemes of pure dimension $n$ and $m$
respectively.  We recall some notions from
\cite[\S16.1]{fulton:intersection_theory}.

\begin{definition}
A \linedef{correspondence} $\alpha$ from $X$ to $Y$ is an element
$\alpha \in \CH(X\times Y)$. The same element $\alpha$, seen in $\CH(Y
\times X)$ is called the transpose correspondence $\alpha'$ from $Y$
to $X$.
\end{definition}

Now assume that $Y$ is proper over $k$ and let $Z$ be a smooth
equidimensional $k$-scheme.  If $\alpha \in \CH(X \times Y)$ and
$\beta \in \CH(Y \times Z)$ are correspondences, we define the
composed correspondence
$$
\beta \circ \alpha = p_{X \times Z *} (p_{X \times Y}^*(\alpha) . p_{Y
\times Z}^* (\beta))  \in \CH(X \times Z),
$$
where $p_\bullet$ denotes the projection from $X \times Y \times Z$ to
$\bullet,$ and where we use the intersection product on $\CH(X \times
Y \times Z)$.  Taking $X=Y=Z$ smooth and proper, the operation of
composition of correspondences makes $\CH(X\times X)$ into an
associative ring with unit $[\Delta_X]$.

Correspondences between $X$ and $Y$ naturally give rise to morphisms
between their Chow groups as follows. If $\alpha \in \CH^{m+i}(X \times
Y)$ is a correspondence from $X$ to $Y$, then we define
\begin{align*}
\alpha_* : \CH_j(X) {}& \longrightarrow \CH_{j-i}(Y) 
\qquad &
\alpha^* : \CH^j(Y) {}& \longrightarrow \CH^{j+i}(X)\\
z {}& \longmapsto q_*(p^*(z).\alpha)
\qquad &
z {}& \longmapsto p_*(q^*(z).\alpha)
  \end{align*}
where $p$ and $q$ denote the projections from $X \times Y$ to $X$ and
$Y$, respectively.  If $\beta$ is a correspondence from $Y$ to $Z$,
then we have  $(\beta \circ \alpha)_* = \beta_* \circ \alpha_*$ and $(\beta
\circ \alpha)^* = \alpha^* \circ \beta^*$.
An important special case are correspondences $\alpha \in \CH^n(X
\times X)$, which define maps $\alpha_* : \CH_i(X) \to \CH_i(X)$ and
$\alpha^* : \CH^i(X) \to \CH^i(X)$.  In particular, the map
$$
\CH^n(X \times X) \to \End_\ZZ(\CH_i(X))
$$
is a ring homomorphism
(see~\cite[Cor.~16.1.2]{fulton:intersection_theory}).
  
As an example, letting $f: X \to Y$ be a morphism with graph $\Gamma_f
\subset X \times Y$, we can consider $\alpha = [\Gamma_f] \in
\CH^m(X\times Y)$ as a correspondence from $X$ to $Y$, and then
$\alpha_*=f_*$ and $\alpha^* = f^*$.

\subsection{Specialization}
\label{subsec:specialization}

Most of the previous intersection theoretic considerations carry over
to a more general relative setting, replacing the base field $k$ with
a regular base scheme $S$.

Let $X$ be a scheme that is separated and finite type over $S$.  For a
closed integral subscheme $V \subset X$, we define the
\linedef{relative dimension} $\dim_{\!S}(V) =
\mathrm{tr.deg}(K(V)/K(W)) - \mathrm{codim}_S(W)$, where $W$ is the
closure of the image of $V$ in $S$. A relative $i$-cycle on $X/S$ is
an integer linear combination of integral subschemes of $X$ of
relative dimension $i$.  The notion of rational equivalance of
relative $i$-cycle is as before and we denote by $\CH_i(X/S)$ the
group of relative $i$-cycles on $X/S$ up to rational equivalence.  As
before, there are push-forwards for proper $S$-morphisms, pull-backs
for flat $S$-morphisms, and Gysin maps for lci $S$-morphisms.

Now suppose that $\iota:\overline{S} \to S$ is a regular embedding of
codimension $r$, and let $j:S^0 \to S$ be the complement of
$\overline{S}$. Consider the following diagram of cartesian squares:
$$
\xymatrix{
\overline{X} \ar[r]\ar[d] & X \ar[d] & X^0 \ar[l] \ar[d]\\
\overline{S} \ar[r]^{\iota} & S & S^0 \ar[l]_{j}
}
$$
Noting that $\CH_i(\overline{X}/\overline{S}) =
\CH_{i-r}(\overline{X}/S)$ and $\CH_i(X^0/S^0)=\CH_i(X^0/S)$, then the
Gysin map $\iota^!: \CH_i(X/S) \to
\CH_{i-r}(\overline{X}/S)=\CH_{i}(\overline{X}/\overline{S})$ gives
rise to a diagram
$$
\xymatrix{
\CH_{i}(\overline{X}/{S}) \ar[r]^{\iota_*} \ar@{=}[d] & \CH_i(X/S)
\ar[d]^(.41){\iota^!} \ar[r]^{j^*} & \CH_i(X^0/S) \ar@{=}[d] \ar[r] & 0\\
\CH_{i+r}(\overline{X}/\overline{S}) &
\CH_{i}(\overline{X}/\overline{S}) & \CH_i(X^0/S^0) \ar@{-->}[l]_\sigma& 
}
$$
where the top row is the relative short exact excision sequence.  We
see that the obstruction to defining a well-defined map
\linedef{specialization map} $\sigma : \CH_i(X^0/S^0) \to
\CH_i(\overline{X}/\overline{S})$ fitting into the diagram is
precisely the image of $\iota^! \iota_*$.  By the excess intersection
formula, if $N_{\overline{S}/S}$ is trivial, then $\iota^!  \iota_*
=0$, in which case we arrive at a well-defined specialization map.
When it exists, the specialization map is compatible with push-forwards
and pull-backs.  

An important special case is when $S=\Spec(R)$ for a discrete
valuation ring $R$, so that $S^0=\Spec(K)$ and
$\overline{S}=\Spec(k)$, where $k$ and $K$ denote the residue and the
fraction field of $R$, respectively. Given a separated $R$-scheme $X$
of finite type, the $k$-scheme $\overline{X}=X_k$ is the special
fiber, the $K$-scheme $X^0=X_K$ is the generic fiber, and we arrive at
specialization maps $\sigma : \CH_i(X_K) \to \CH_i(X_k)$. For more
details, we refer to~\cite[\S20.3]{fulton:intersection_theory}.

\section{Preliminaries on semiorthogonal decompositions}\label{sect:prel-on-dercat}

Let $k$ be an arbitrary field. We present here the basic notions of semiorthogonal decompositions
and exceptional objects for $k$-linear triangulated categories, bearing in mind
our main application, the derived categories of a projective $k$-variety.
We refer to~\cite[Ch. 1, 2, 3]{Huybook} for an introduction to derived
categories aimed at algebraic geometers. In particular, we will assume the
reader to be familiar with the notions of triangulated and derived categories, 
and basic homological algebra as well as complexes of coherent sheaves on
schemes.

However, a disclaimer here is necessary. The appropriate structure to consider
to study derived categories of smooth projective varieties is the structure
of $k$-linear differential graded (dg) category; that is, a category enriched over
dg complexes of $k$-vector spaces (see~\cite{keller:ICM} for definitions and
main properties). In this perspective, morphisms between two objects in the triangulated structure
can be seen as the zeroth cohomology of the complex of morphisms between the
same objects in the dg structure. 
Considering the dg structure is natural under many point of views: above all, all categories we will consider
can be endowed with a canonical dg structure (see~\cite{Lunts-Orlov}),
in such a way that dg functors will correspond
to Fourier--Mukai functors (see~\cite{toen:derived-morita}). Moreover,
the dg structure allows to define noncommutative motives, which give a motivic framework to
semiorthogonal decompositions. Even if related to some of our considerations,
we will not treat noncommutative motives in this report. The interested reader can consult~\cite{tabuada-book}.

\subsection{Semiorthogonal decompositions and their mutations}
Let $\cat{T}$ be a $k$-linear triangulated category. 
A full triangulated subcategory
$\cat{A}$ of $\cat{T}$ is called \it admissible \rm if the embedding
functor admits a left and a right adjoint.

\begin{definition}[\cite{bondal_kapranov:reconstructions}]
\label{def-semiortho}
A {\em semiorthogonal decomposition} of $\cat{T}$ is a sequence of
admissible subcategories $\cat{A}_1, \ldots, \cat{A}_n$ of $\cat{T}$
such that

\begin{itemize}

\item $\Hom_{\cat{T}}(A_i,A_j) = 0$ for all $i>j$ and any $A_i$ in $\cat{A}_i$
and $A_j$ in $\cat{A}_j$; 

\item for every object $T$ of $\cat{T}$, there is a chain
of morphisms $0=T_n \to T_{n-1} \to \ldots \to T_1 \to T_0 = T$ such that the cone
of $T_k \to T_{k-1}$ is an object of
$\cat{A}_k$ for all $k=1,\ldots,n$. 

\end{itemize}

Such a decomposition will be written
$$\cat{T} = \langle \cat{A}_1, \ldots, \cat{A}_n \rangle.$$
\end{definition}

If $\cat{A} \subset \cat{T}$ is admissible, we have
two semiorthogonal decompositions
$$
\cat{T}=\langle \cat{A}^{\perp}, \cat{A} \rangle = \langle \cat{A},
^{\perp}\cat{A} \rangle,
$$
where $\cat{A}^{\perp}$ and $^\perp\cat{A}$ are, respectively, the
left and right orthogonal of $\cat{A}$ in $\cat{T}$ (see~\cite[\S
3]{bondal_kapranov:reconstructions}).

Given a semiorthogonal decomposition $\cat{T} = \langle \cat{A},
\cat{B} \rangle$, Bondal~\cite[\S3]{bondal:representations} defines
left and right mutations $L_{\cat{A}}(\cat{B})$ and
$R_{\cat{B}}(\cat{A})$ of this pair. In particular, there are
equivalences $L_{\cat{A}}(\cat{B}) \equi \cat{B}$ and
$R_{\cat{B}}(\cat{A}) \equi \cat{A}$, and semiorthogonal
decompositions
$$
\begin{array}{ccc}
\cat{T}=\langle L_{\cat{A}}(\cat{B}), \cat{A} \rangle, & \,\,\,\,\, &  
\cat{T} = \langle \cat{B}, R_{\cat{B}}(\cat{A}) \rangle.
\end{array}
$$
We refrain from giving an explicit definition for the mutation
functors in general, which can be found in~\cite[\S
3]{bondal:representations}.  In \S\ref{subs:exc-objects} we will give
an explicit formula in the case where $\cat{A}$ and $\cat{B}$ are
generated by exceptional objects.

\subsection{Exceptional objects}
\label{subs:exc-objects}

Very special examples of admissible subcategories, semiorthogonal
decompositions, and their mutations are provided by the theory of
exceptional objects and collections. The theory of
exceptional objects and semiorthogonal decompositions in the case
where $k$ is algebraically closed and of characteristic zero was
studied in the Rudakov seminar at the end of the 80s, and developed by
Rudakov, Gorodentsev, Bondal, Kapranov, Kuleshov, and Orlov among others, see
\cite{gorodentsev-rudakov}, \cite{bondal:representations},
\cite{bondal_kapranov:reconstructions},
\cite{bondal_orlov:semiorthogonal}, and~\cite{helices-book}. As noted
in~\cite{auel-berna-bolo}, most fundamental properties persist over
any base field $k$.

Let $\cat{T}$ be a $k$-linear triangulated category. The triangulated
category $\sod{\{E_i\}_{i\in I}}$ \linedef{generated} by a class of
objects $\{E_i\}_{i\in I}$ of $\cat{T}$ is the smallest thick (that
is, closed under direct summands) full triangulated subcategory of
$\cat{T}$ containing the class.  We will write $\Ext_\cat{T}^r(E,F) = \Hom_\cat{T}(E,F[r])$.

\begin{definition}
\label{def-except}
Let $A$ be a division (not necessarily central) $k$-algebra (e.g., $A$
could be a field extension of $k$).  An object $E$ of $\cat{T}$ is
called \linedef{$A$-exceptional} if
$$
\Hom_{\cat{T}}(E,E) = A \quad \text{and} \quad
\Ext^r_{\cat{T}}(E,E)=0 \quad \text{for} \quad r \neq 0.
$$
An exceptional object in the classical sense
\cite[Def.~3.2]{gorodentsev-moving} of the term is a $k$-exceptional
object.  By \linedef{exceptional} object, we mean $A$-exceptional for
some division $k$-algebra $A$.

A totally ordered set $\{E_1,\ldots,E_n\}$ of exceptional objects is
called an \linedef{exceptional collection} if
$\Ext^r_{\cat{T}}(E_j,E_i)=0$ for all integers $r$ whenever $j>i$.  An
exceptional collection is \linedef{full} if it generates $\cat{T}$,
equivalently, if for an object $W$ of $\cat{T}$, the vanishing
$\Ext^r_{\cat{T}}(E_i,W)=0$ for all $i=1,\ldots,n$ and all integers $r$
implies $W=0$.  An exceptional collection is \linedef{strong} if
$\Ext^r_{\cat{T}}(E_i,E_j)=0$ whenever $r\neq 0$.
\end{definition}

Exceptional collections provide examples of semiorthogonal
decompositions when $\cat{T}$ is the bounded derived category of a
smooth projective scheme.

\begin{prop}[{\cite[Thm.~3.2]{bondal:representations}}]
Let $\{ E_1, \ldots, E_n \}$ be an exceptional collection on the
bounded derived category $\Db(X)$ of a smooth projective $k$-scheme
$X$.  Then there is a semiorthogonal decomposition
$$
\Db(X) = \langle \cat{A}, E_1, \ldots, E_n \rangle,
$$
where $\cat{A}=\sod{E_1,\ldots,E_n}^\perp$ is the full subcategory of
objects $W$ such that $\Ext^r_{\cat{T}}(E_i,W)=0$ for all $i=1,\ldots,
n$ and all integers $r$.  In particular, the sequence if full if and
only if $\cat{A}=0$.
\end{prop}

Given an exceptional pair $\{ E_1, E_2 \}$ with $E_i$ being
$A_i$-exceptional, consider the admissible subcategories $\langle E_i
\rangle$, forming a semiorthogonal pair. We can hence perform right
and left mutations, which provide equivalent admissible
subcategories. 

Recall that mutations provide equivalent admissible subcategories and
flip the semiorthogonality condition.
It easily follows that the object $R_{E_2}(E_1)$ is $A_1$-exceptional, the
object $L_{E_1}(E_2)$ is $A_2$-exceptional, and the pairs $\{
L_{E_1}(E_2), E_1 \}$ and $\{E_2, R_{E_2}(E_1)\}$ are exceptional.
We call $R_{E_2}(E_1)$ the \linedef{right mutation} of $E_1$ through $E_2$ and
$L_{E_1}(E_2)$ the \linedef{left mutation} of $E_2$ through $E_1$.

In the case of $k$-exceptional objects, mutations can be explicitly
computed.

\begin{definition}[{\cite[\S3.4]{gorodentsev-moving}}]
Given a $k$-exceptional pair $\{ E_1, E_2 \}$ in $\cat{T}$, the
\emph{left mutation} of $E_2$ with respect to $E_1$ is the object
$L_{E_1}(E_2)$ defined by the distinguished triangle:
\begin{equation}
\label{eq:left-mutation}
\Hom_\cat{T}(E_1,E_2) \otimes E_1 \mapto{\,ev\,} E_2
\longrightarrow L_{E_1}(E_2),
\end{equation}
where $ev$ is the canonical evaluation morphism. The \linedef{right
mutation} of $E_1$ with respect to $E_2$ is the object $R_{E_2}(E_1)$
defined by the distinguished triangle:
$$
R_{E_2}(E_1) \longrightarrow E_1 \mapto{coev}
\Hom_\cat{T}(E_1,E_2) \otimes E_2,
$$
where $coev$ is the canonical coevaluation morphism.
\end{definition}

Given an exceptional collection $\{ E_1, \ldots, E_n \}$, one
can consider any exceptional pair $\{ E_i, E_{i+1} \}$ and perform
either right or left mutation to get a new exceptional collection.

Exceptional collections provide an algebraic description of admissible
subcategories of $\cat{T}$. Indeed, if $E$ is an $A$-exceptional
object in $\cat{T}$, the triangulated subcategory $\langle E \rangle
\subset \cat{T}$ is equivalent to $\Db(k,A)$.  The equivalence
$\Db(k,A) \to \langle E \rangle$ is obtained by sending the complex
$A$ concentrated in degree $0$ to $E$. The right adjoint functor
is the morphism functor $\RHom(-,E)$.

We conclude this section by considering a weaker notion of exceptionality, which depends
only on the numerical class and on the bilinear form $\chi$.

\begin{definition}
Let $X$ be a smooth projective variety. A \linedef{numerically exceptional collection} is
a collection $E_1,\ldots,E_n$ of exceptional objects in the derived category $\Db(X)$ such that 
$\chi(E_i,E_j) = 0$ for $i > j$ and $\chi(E_i,E_i)=1$ for all $i=1,\ldots,n$.
\end{definition}

\begin{remark}
It is clear that any exceptional collection is a numerically exceptional
collection, while the converse need not to be true.
\end{remark}

\subsection{How to construct semiorthogonal decompositions? Examples and subtleties}
Given a variety $X$ is quite difficult to describe semiorthogonal decompositions
of $X$. Moreover, the geometry of $X$ plays a very important r\^ole in understanding
whether the category $\Db(X)$ has semiorthogonal decompositions and in describing
semiorthogonal sets of admissible subcategories. In general, the most difficult task is to show that 
such sets form a generating system for the whole category\footnote{Notice that $\Db(X)$ admits a fully orthogonal decomposition if and only if 
$X$ is not connected. We will only consider connected varieties.}.

The main motivation for the study of birational geometry via semiorthogonal decompositions
is the following famous theorem by Orlov~\cite{orlovprojbund}.

\begin{theorem}[Orlov]\label{thm:blow-ups}
Let $X$ be a smooth projective variety, $Z \subset X$ a smooth subvariety of codimension
$c \geq 2$, and $\sigma:Y \to X$ the blow-up of $Z$. Then the functor $L\sigma^*: \Db(X) \to \Db(Y)$
is fully faithful, and, for $i=1, \ldots, c-1$ there are fully faithful functors $\Phi_i : \Db(Z) \to \Db(Y)$,
and a semiorthogonal decomposition
$$\Db(Y)=\sod{L\sigma^* \Db(X), \Phi_1 \Db(Z), \ldots, \Phi_{c-1} \Db(Z)}$$
\end{theorem}

Notice that Orlov's argument of the fully faithfulness of $L\sigma^*$
extends to the cases of a surjective morphism with rationally connected
fibers between smooth and projective varieties, though the
description of the orthogonal complement is in general unknown. The
fact that $L\sigma^*$ is fully faithful in Theorem~\ref{thm:blow-ups}
can be seen as a special case of the following Lemma, since a blow up
gives a surjective map with the required properties.

\begin{lemma}\label{lem:birat-map-vs-ffemb}
Let $X$ and $Y$ be smooth and projective $k$-schemes and $\sigma: Y
\to X$ a surjective morphism such that $\sigma_* \ko_Y = \ko_X$ and
$R^i \sigma_* \ko_Y=0$ for $i\neq 0$.  Then $L\sigma^*: \Db(X) \to
\Db(Y)$ is fully faithful.
\end{lemma}
\begin{proof}
For any $A$ and $B$ objects in $\Db(X)$, we have
$$\Hom_{Y}(L\sigma^*A,L\sigma^*B)=\Hom_X(A, R\sigma_* L\sigma^* B) = \Hom_X(A,B\otimes R\sigma_* \ko_Y) = \Hom_X(A,B),$$
by adjunction, projection formula and by our assumption respectively.
\end{proof}

The canonical bundle and its associated invariants, like the geometric genus and the irregularity, play a central r\^ole
in this theory. First of all it is easy to remark, using Serre duality, that if $X$ has a trivial canonical bundle,
then there is no non-trivial semiorthogonal decomposition of $\Db(X)$. The results obtained
by Okawa~\cite{okawa-curves} and Kawatani--Okawa~\cite{kawatani-okawa}
for low dimensional varieties are also strongly related to the canonical bundle.

\begin{theorem}[Okawa]\label{thm:okawacurves}
Let $C$ be a smooth projective $k$-curve of positive genus. Then $\Db(C)$ has no non-trivial semiorthogonal decompositions.
\end{theorem}

\begin{theorem}[Kawatani--Okawa]\label{thm:okawasurfaces}
Let $k$ be algebraically closed and $S$ a smooth connected projective minimal surface. Suppose that
\begin{itemize}
 \item either $\kappa(S)=0$ and $S$ is not a classical Enriques surface, or
 \item $\kappa(S)=1$ and $p_g(S) >0$, or
 \item $\kappa(S)=2$, that $\dim H^1(S,\omega_S) > 1$, and for any one-dimensional connected component
 of the base locus of $\omega_S$, its intersection matrix is negative definite.
\end{itemize}
Then there is no nontrivial semiorthogonal decomposition of $\Db(S)$. 
\end{theorem}

Roughly speaking, one could say that varieties admitting semiorthogonal decompositions
should have cohomological properties which are very close to the ones of a Fano (relatively over some base) variety.

If $X$ is a Fano variety, that is if the canonical bundle $\omega_X$
is antiample, any line bundle $L$ on $X$ is a $k$-exceptional object
in $\Db(X)$, and hence gives a semiorthogonal decomposition $\Db(X) =
\sod{\cat{A},L}$, where $\cat{A}$ consists of objects right orthogonal
to $L$.  In the simpler case where $X$ has index Picard rank $1$ and
index $i$ (that is, $\omega_X=\ko(-i)$), and $\mathrm{char}(k)=0$, one
can use Kodaira vanishing theorems to construct a natural
$k$-exceptional sequence, as remarked by Kuznetsov~\cite[Corollary
3.5]{kuznet:fano3folds}.  

\begin{prop}[Kuznetsov]\label{prop:decoFano}
Let $X$ be a smooth Fano variety of Picard rank 1 with ample generator $\ko(1)$, and index $i$. Then there is a semiorthogonal
decomposition
$$\Db(X) = \sod{\cat{A}_X, \ko_X, \ldots \ko_X(i-1)},$$
where $\cat{A}_X=\sod{\ko_X, \ldots \ko_X(i-1)}^\perp$ is the category
of objects $W$ such that $\Ext^r(\ko(j),W)=0$ for all $0 \leq j < i$
and for all integers $r$.
\end{prop}
The previous result is easily generalized to the relative case of Mori fiber spaces as in~\cite[Proposition 2.2.2]{auel-berna-bolo}.

\begin{prop}\label{prop:decoMFS}
Let $\pi:X \to Y$ be a flat surjective fibration between smooth
varieties, such that $\Pic(X/Y) \simeq \ZZ$ with ample generator
$\ko_{X/Y}(1)$ and $\omega_{X/Y}= \ko_{X/Y}(-i)$. Set
$\Db(Y)(j):=\pi^* \Db(Y) \otimes \ko_{X/Y}(j)$.  For any $j$, this
gives a fully faithful embedding of $\Db(Y)$ into $\Db(X)$.  Moreover,
over a field $k$ of characteristic $0$, there is a semiorthogonal
decomposition
$$\Db(X) = \sod{\cat{A}_{X/Y}, \Db(Y)(0), \ldots, \Db(Y)(i-1)},$$
where
$$
\cat{A}_{X/Y} =\sod{\Db(Y)(0), \ldots, \Db(Y)(i-1)}^\perp
$$
is the category of objects $W$ such that $\Ext^r(\pi^* A \otimes
\ko(j),W)=0$ for all $0 \leq j < i$, for all integers $r$, and for all
objects $A$ in $\Db(Y)$.
\end{prop}
\begin{proof}
Notice that $\Db(Y)(j)$ being admissible in $\Db(X)$ is a consequence
of Lemma~\ref{lem:birat-map-vs-ffemb}. The semiorthogonality is given
by a relative Kodaira vanishing. Finally, define $\cat{A}_{X/Y}$ to be
the complement.
\end{proof}

\begin{remark}
The assumption on $k$ having characteristic zero is needed to ensure
that the Kodaira vanishing theorem holds on $X$, but can be
weakened. Indeed, Kodaira vanishing theorems hold in characteristic
$p$ for varieties that lift to a smooth variety in characteristic 0,
see Deligne--Illusie \cite{deligne-illusie}. For example, we could
consider any complete intersection in projective space of Fano type
over a field of characteristic $p$.
\end{remark}

One should consider the decompositions above as the most related to
the geometric structure of $X$, and the category $\cat{A}_{X/Y}$ as
the best witness of the birational behavior of $X$\footnote{Notice,
that one can replace the sequence $\Db(Y)(0),\ldots,\Db(Y)(i-1)$ by
$\Db(Y)(j),\ldots,\Db(Y)(j+i-1)$ for any integer $j$. However, this
would give orthogonal complements which are not only equivalent as
triangulated categories, but also as dg categories.}. This idea is
supported by the following results of Beilinson~\cite{beilinson} (for
the case of $\PP^n$) and Orlov~\cite{orlovprojbund}.

\begin{prop}[Beilinson, Orlov]\label{prop:deco-projective}
Let $\pi:X \to Y$ be a projective bundle of relative dimension $r$, that is, $X = \PP_Y(E)$ for 
some rank $r+1$ vector bundle $E$ on $Y$. Then 
$$
\Db(X) = \sod{\Db(Y)(0),\ldots,\Db(Y)(r)}.
$$
In other words, $\cat{A}_{X/Y}=0$.
\end{prop}

The most difficult task in proving the Proposition~\ref{prop:deco-projective}, already for projective spaces, is to
show that a given sequence of categories generates the whole category. This is done in~\cite{beilinson}
using a complex resolving the structure sheaf of the diagonal of $\PP^n \times \PP^n$. Let us list other known
descriptions of $\cat{A}_{X/Y}$.

\begin{example}\label{ex:semiortho-for-MFS}
Let $Y$ be a smooth projective $k$-variety and $\pi: X \to Y$ as in Proposition~\ref{prop:decoMFS}. Then $\cat{A}_{X/Y}$
is known in the following cases:

\smallskip

\noindent{\bf Projective bundles.} If $\pi: X \to Y$ is a projective bundle, then $\cat{A}_{X/Y}=0$,~\cite{orlovprojbund}.

\smallskip

\noindent{\bf Projective fibrations.} Let $\pi: X \to Y$ be a relative Brauer--Severi variety (that is, the geometric fibers of $\pi$
 are projective spaces, but $X$ is not isomorphic to $\PP(E)$ for any vector bundle $E$ on $Y$), and $\alpha$ in $\Br(Y)$ the class
 of $X$ and $r$ the relative dimension. If $\omega_{X/Y}$ generates $\Pic(X/Y)$, then 
 $$\cat{A}_{X/Y}= \sod{\Db(Y,\alpha), \ldots, \Db(Y,\alpha^r)}.$$
If $\omega_{X/Y}$ is not primitive, a similar description is possible~\cite{bernardara:brauer_severi}.

\smallskip

\noindent{\bf Quadric fibrations.} Let $\pi: X \to Y$ be a quadric fibration of relative dimension $r$ and let $\kc_0$ be the
sheaf of even Clifford algebras associated to the quadratic form defining $X$. Then $\cat{A}_{X/Y}=\Db(Y,\kc_0)$,~\cite{kuznetquadrics}.

\smallskip

\noindent{\bf Fibrations in intersections of quadrics.} Let $\{Q_i \to Y\}_{i=0}^s$ be quadric fibrations of relative dimension $r$
and $\pi: X \to Y$ be their intersection (see~\cite{auel-berna-bolo} for details), and suppose that $\omega_{X/Y}$ is relatively
antiample (that is, $r < 2s$). Then there is a $\PP^s$-bundle $Z \to Y$ and a sheaf of Clifford algebras $\kc_0$ on $Z$, and
$\cat{A}_{X/Y} = \Db(Z,\kc_0)$,~\cite{auel-berna-bolo}.

\smallskip

\noindent This list is far from being exhaustive, since many specific cases are also known (see, e.g., Table~\ref{table:fano-sods}
for 3 and 4 dimensional cases).
\end{example}

In the case where $k$ is not algebraically closed, then one can look for semiorthogonal decompositions
of $\Db(X_{\overline{k}})$ and understand whether they can give informations on $\cat{A}_{X/Y}$ or, more in
general on $\Db(X)$. This rather challenging problem can be tackled in the simplest case, where $\Db(X_{\overline{k}})$ is
generated by vector bundles, using Galois descent of such vector bundles (see~\cite{auel-berna-surf}).
With this in mind one can describe $\cat{A}_{X/Y}$ when $X$ is a minimal del Pezzo, $Y$ is a point, and $k$
is any field~\cite{auel-berna-surf}. Other cases of (generalized) Brauer--Severi varieties~\cite{bernardara:brauer_severi,blunk-gen-bs}
can be treated this way.

On the other hand, even if a geometric description of $\cat{A}_{X/Y}$ is not possible, one can calculate
its Serre functor.

\begin{definition}
Let $\cat{A}$ be a triangulated $k$-linear category with finite dimensional morphism spaces. A functor
$S: \cat{A}\to \cat{A}$ is a \linedef{Serre functor}
if it is a $k$-linear equivalence inducing a functorial isomorphism
$$\Hom_{\cat{A}}(X,Y) \simeq \Hom_{\cat{A}}(Y,S(X))^\vee$$
of $k$-vector spaces, for any object $X$ and $Y$ of $\cat{A}$.

A category $\cat{A}$ with a Serre functor $S_{\cat{A}}$ is a {\it Calabi--Yau
category} (or a {\it non commutative Calabi--Yau}) of dimension $n$ if $S_{\cat{A}}=[n]$.
It is a {\it fractional Calabi--Yau category} of dimension $n/c$ if $c$ is the
smallest integer such that the iterate Serre functor is a shift functor and $S_{\cat{A}}^c = [n]$.
Note that the fractional dimension
of $\cat{A}$ is not a rational number, but a pair of two integer numbers.
\end{definition}

The Serre functor generalizes the notion of Serre duality to a more general setting. Indeed, if $X$ is
a smooth projective $k$-scheme, then $S_{\Db(X)}(-) = - \otimes \omega_X [\dim(X)]$ by Serre duality.

As a consequence of the work of Bondal and Kapranov~\cite{bondal_kapranov:reconstructions}, if $X$ is a smooth and
projective $k$-scheme and $\cat{A}$ is an admissible subcategory of $\Db(X)$, then $\cat{A}$ has a Serre functor
which can be explicitly calculated from the Serre functor of $X$ using adjunctions to the embedding $\cat{A} \to \Db(X)$.
Kuznetsov performed explicitly these calculations for Fano hypersurfaces in projective spaces, see~\cite[Cor. 4.3]{kuznetsov:v14}.

\begin{prop}[Kuznetsov]
Let $X \subset \PP^{n+1}$ be a hypersurface of degree $d < n+2$, and set $c$ the greatest common
divisor of $d$ and $n+2$. Then $\cat{A}_X$ is a (fractional) Calabi--Yau category, that is
$S^{d/c}_{\cat{A}_X} = [\frac{(d-2)(n+2)}{c}]$.
\end{prop}

\begin{remark}
Notice that both $d/c$ and $\frac{(d-2)(n+2)}{c}$ are
integers. However, the fractional dimension of $\cat{A}_X$ is not a
simplification of the fraction $\frac{(d-2)(n+2)}{c}$ unless
$c=1$. For example, for a quartic fourfold we obtain $6/2$. However,
in the case where $d$ divides $n+2$, $\cat{A}_X$ is a Calabi--Yau
category.
\end{remark}
\begin{corollary}\label{cor:cubic-4folds=nck3}
If $X \subset \PP^5$ is a smooth cubic fourfold, then $\cat{A}_X$ is a
2-Calabi--Yau category (or a noncommutative K3 surface).
\end{corollary}

The categories $\cat{A}_{X/Y}$ also admit algebraic descriptions, that is, one can find an equivalence
with a triangulated category which arises from purely algebraic constructions. The main examples are
Orlov's description via matrix factorizations for Fano complete intersections in projective spaces
(see~\cite{orlov-matrix}) and a rather complicated description based on Homological Projective
Duality for fibrations in complete intersections of type $(d,\ldots,d)$ (see~\cite{BDFIK-higher-veronese}).

To tackle geometrical problems, we would like a description of $\cat{A}_{X/Y}$ by explicit geometric constructions.
A first case, which include a lot of Fano varieties, is the case of homogeneous varieties. These
are conjectured to always carry a full exceptional sequence, and one can construct a candidate sequence
using vanishing theorems and representation of parabolic subgroups, see~\cite{kuznet-polishchuk}. The hardest
part is to prove that such a sequence is full, for which spectral sequences are needed.

The most powerful tool to construct semiorthogonal decompositions is by far Kuznetsov's Homological Projective Duality (HPD). We refrain here to give any
definition, for which we refer to the very dense Kuznetsov's original paper~\cite{kuznetsov:hpd}. 
In practice, HPD allows to compare semiorthogonal decompositions of dual linear sections of fixed projectively
dual varieties. It is in general a hard task to show that two given varieties are HP-dual, and one of the
most challenging steps is to deal with singular varieties and their noncommutative resolutions. However,
HPD allows one to describe a great amount of semiorthogonal
decompositions for Fano varieties or Mori fiber spaces,
see~\cite{kuznetsov:hpd},
\cite{kuz:4fold},~\cite{kuznetsov:hyp-sections}, \cite{auel-berna-bolo} just to name a few.

\smallskip

On the other hand, (relatively) Fano varieties are not the only class of varieties whose derived category admits
a semiorthogonal decompositions. The first natural examples one should consider are surfaces with $p_g=q=0$, in 
which case any line bundle is a $k$-exceptional object. Hence the derived category of such surfaces always admits
nontrivial semiorthogonal decompositions. On the other hand, one can argue that, if $S$ is a such a surface,
then there is no fully faithful functor $\Db(C) \to \Db(S)$ for $C$ a curve of positive genus. Indeed,
such a functor would give a nontrivial Albanese variety (or, equivalently, a nontrivial $\Pic^0$), see~\cite{berna-tabuada-jacobians}.
Another way to present this argument is by noticing that $H^{p,q}(S)=0$ if $p-q \neq 0$. It follows that
the Hochschild homology $HH_i(S)=0$ for $i \neq 0$. This last fact obstructs the existence of the functor,
since $HH_{\pm 1}(C) \neq 0$ for a positive genus curve $C$.

It is then natural to look for semiorthogonal decompositions of the form:
$$\Db(S) = \sod{ \cat{A}_S, E_1, \ldots, E_n},$$
with $E_i$ $k$-exceptional objects and wonder about the maximal possible value of $n$ and the structure of $\cat{A}_S$.
Describing $\cat{A}_S$ is a very challenging question and we will
treat examples and their conjectural relation with rationality
questions in \S\ref{sect:Db-surfaces}. 

Let us conclude by remarking that, studying such surfaces, B\"ohning, Graf von Bothmer, and Sosna have been able to show 
that semiorthogonal decompositions do not enjoy, in general, a Jordan--H\"older type property~\cite{boeh-graf-sos-JH}. Notice that 
a further example is explained by Kuznetsov~\cite{kuznet:JH}. 

\begin{prop}[B\"ohning--Graf von Bothmer--Sosna]\label{prop:no-JH}
Let $X$ be the classical Godeaux complex surface. The bounded derived category $\Db(X)$ has two maximal exceptional sequences of different lengths:
one of length 11 and one of length 9 which cannot be extended further.
\end{prop}

\section{Unramified cohomology and decomposition of the diagonal}

Unramified cohomology has emerged in the last four decades as a
powerful tool for obstructing (stable) rationality in algebraic
geometry.  Much of its utility comes from the fact that the theory
rests on a combination of tools from scheme theory, birational
geometry, and algebraic $K$-theory.  Used notably in the context of
Noether's problem in the work of Saltman and Bogomolov, unramified
cohomology can be computed purely at the level of the function field,
without reference to a specific good model.

\subsection{Flavors of rationality}

A variety $X$ over a field $k$ is \linedef{rational} over $k$ if $X$
is $k$-birational to projective space $\PP^n$, it is
\linedef{unirational} over $k$ if there is a dominant rational $\PP^N
\dashrightarrow X$ for some $N$, it is \linedef{retract rational} over
$k$ if there is a dominant rational $\PP^N \dashrightarrow X$ with a
rational section, it is \linedef{stably rational} over $k$ if $X
\times \PP^N$ is rational for some $N$.  The notion of retract
rationality was introduced by Saltman in the context of Noether's
problem.

We have the following implications:
$$
\text{rational} \Rightarrow \text{stably rational} \Rightarrow
\text{retract rational} \Rightarrow \text{unirational} \Rightarrow
\text{rationally connected}.
$$
Several important motivating problems in the study of rationality in
algebraic geometry can be summarized as asking whether these implications
are strict.

\begin{problem}[L\"uroth problem]\label{luroth-problem}
Determine whether a given unirational variety $X$ is rational.
\end{problem}

\begin{problem}[Birational Zariski problem]\label{zariski-problem}
Determine whether a given stably rational variety $X$ is rational.
\end{problem}

\begin{problem}
Does there exist a rationally connected variety $X$ with $X(k) \neq
\varnothing$ that is not unirational?
\end{problem}

\begin{problem}
Does there exist a retract rational variety $X$ that is not stably rational?
\end{problem}

The L\"uroth question has a positive answer in dimension 1 (proved by
L\"uroth) over an arbitrary field and in dimension 2 over an
algebraically closed field of characteristic zero (proved by
Castelnuovo~\cite{castelnuovo}).  There exist counterexamples, i.e.,
unirational but nonrational surfaces, over the real numbers (as
remarked by Segre~\cite{segre:real}) and over an algebraically closed
field of characteristic $p>0$ (discovered by Zariski
\cite{zariski:surface}).  The first known counterexamples over $\CC$
were in dimension 3, discovered independently by
Clemens--Griffiths~\cite{clemens_griffiths}, Iskovskih--Manin
\cite{iskovskih_manin}, and Artin--Mumford~\cite{artin_mumford}.  We
point out that the example of Artin--Mumford also provided the first
example of a unirational variety that is not stably rational over
$\CC$.  The method of the intermediate Jacobian due to
Clemens--Griffiths and the method of birational rigidity due to
Iskovskih--Manin do not obstruct stable rationality. We will treat the
former in more details in \S\ref{sect:cubics}.

The first known counterexamples to the birational Zariski problem were
discovered by
Beauville--Colliot-Th\'el\`ene--Sansuc--Swinnerton-Dyer~\cite{bctssd}
in dimension 2 over nonalgebraically closed fields and in dimension 3
over $\CC$ using the method of the intermediate Jacobian, which we
will recall in \S\ref{sect:intjacobians}.

There exist retract rational tori that are not stably rational over
$\QQ$ discovered in the context of work by Swan and Voskresenski\v\i\ on
Noether's problem, see \cite[\S8.B,~p.~223]{colliot-thelene_sansuc:tores}. 

The last two questions are still open over an algebraically closed field!

\subsection{Unramified elements}
\label{subsec:unramified}

Various notions of unramified cohomology emerged in the late 1970s and
1980s \cite{colliot:quadratiques78}, \cite{colliot:quadratiques80},
\cite{colliot-thelene_sansuc:type_multiplicatif},
\cite{colliot-thelene_sansuc:reeles},
\cite{colliot-thelene_ojanguren}, mostly motivated by earlier
investigations of the Brauer group~\cite{auslander_goldman}, \cite{Br}
and the Gersten conjecture~\cite{bloch_ogus} in algebraic $K$-theory.
The general notion of ``unramified element'' of a functor is developed
in \cite[\S2]{colliot:santa_barbara}.  Rost~\cite[Rem.\
5.2]{rost:cycle_modules} gives a different perspective in terms of
cycle modules, also see Morel~\cite[\S2]{morel:proof_milnor}. Let $k$
be a field and denote by $\Local_k$ the category of local $k$-algebras
together with local $k$-algebra homomorphisms.  Let $\Ab$ be the
category of abelian groups and let $M : \Local_k \to \Ab$ be a
functor.  For any field $K/k$ the group of \linedef{unramified
elements} of $M$ in $K/k$ is the intersection
$$ 
M_{\ur}(K/k) = \bigcap_{k \subset \OO \subset K} \im \bigl( M(\OO) \to
M(K) \bigr)
$$
over all rank 1 discrete valuations rings $k \subset \OO \subset K$
with $\Frac(\OO)=K$.

There is a natural map $M(k) \to M_{\ur}(K/k)$ and we say that the
group of unramified elements $M_{\ur}(K/k)$ is \linedef{trivial} if
this map is surjective.

For $X$ an integral scheme of finite type over a field $k$, write
$M_{\ur}(X/k) = M_{\ur}(k(X)/k)$.  By definition, the group
$M_\ur(X/k)$ is a $k$-birational invariant of integral schemes of
finite type over $k$.  

We will be mostly concerned with the functor $M=\Het^i(-,\mu)$ with
coefficients $\mu$ either $\mu_n^{\tensor (i-1)}$ (under the
assumption ${\rm char}(k) \neq n$) or
$$
\QQ/\ZZ(i-1) = \varinjlim \mu_n^{\tensor (i-1)},
$$ 
the direct limit being taken over all integers $n$ coprime to the
characteristic of $k$.  In this case, $M_{\ur}(X/k)$ is called the
\linedef{unramified cohomology} group $\Hur^i(X,\mu)$ of $X$ with
coefficients in $\mu$.  

The reason why we only consider cohomology of degree $i$ with
coefficients that are twisted to degree $(i-1)$ is the following
well-known consequence of the norm residue isomorphism theorem proved
by Voevodsky, Rost, and Weibel (previously known as the Bloch--Kato
conjecture).

\begin{theorem}
Let $K$ be a field and $n$ a nonnegative integer prime to the
characteristic.  Then the natural map
$$
H^i(K,\mu_n^{\tensor (i-1)}) \to H^i(K,\QQ/\ZZ(i-1))
$$
is injective and the natural map $\varinjlim H^i(K,\mu_n^{\tensor
(i-1)}) \to H^i(K,\QQ/\ZZ(i-1))$ is an isomorphism, where the limit is
taken over all $n$ prime to the characteristic.
\end{theorem}

\begin{remark}
If $k$ is an algebraically closed field whose
characteristic is invertible in $\mu$, then $\Hur^i(X,\mu)=0$ for all
$i > \dim(X)$, since in this case the function field $k(X)$ has
cohomological dimension $\dim(X)$.
\end{remark}

Another important functor is the Milnor $K$-theory functor $M=K_i^M(-)$.

Let $\mathcal{H}_{\et}^i(\mu)$ be the Zariski sheaf on the category of
$k$-schemes $\Sch_k$ associated to the functor $\Het^{i}(-,\mu)$.  The
Gersten conjecture, proved by Bloch and Ogus~\cite{bloch_ogus}, allows
for the calculation of the cohomology groups of the sheaves
$\mathcal{H}_{\et}^i(\mu)$ on a smooth proper variety $X$ as the
cohomology groups of the Gersten complex (also known as the
``arithmetic resolution'') for \'etale cohomology:
$$
\xymatrix@C=22pt{
0 \ar[r] & H^{i}(F(X)) \ar[r]
& \smash{\underset{{x \in X^{(1)}}}{\textstyle\bigoplus}} H^{i-1}(F(x))
  \ar[r]
& \smash{\underset{{y \in X^{(2)}}}{\textstyle\bigoplus}}
H^{i-2}(F(y)) \ar[r] & \dotsm 
}
$$
where $H^i(-)$ denotes the Galois cohomology group in degree $i$ with
coefficients $\mu$ either $\mu_n^{\tensor (i-1)}$ or $\QQ/\ZZ(i-1)$,
where $X^{(i)}$ is the set of codimension $i$ points $x$ of $X$ with
residue field $F(x)$, and where the ``residue'' morphisms are Gysin
boundary maps induced from the spectral sequence associated to the
coniveau filtration, see Bloch--Ogus \cite[Thm.\ 4.2, Ex.\ 2.1, Rem.\
4.7]{bloch_ogus}.  In particular, we have that
$$
H^0(X,\mathcal{H}_{\et}^{i}(\mu)) = \Hur^{i}(X,\mu).
$$
This circle of ideas is generally called ``Bloch--Ogus theory.''

Over $\CC$, this leads to the following ``geometric interpretation'' of
unramified cohomology, as the direct limit over all Zariski open
coverings $\mathscr{U} = \{U_i\}$ of the set
$$
\bigl\{ \{\alpha_i\} \in H_B^i(\mathscr{U},\mu) \; : \; \alpha_i|_{U_{ij}}
= \alpha_j|_{U_{ij}} \bigr\}
$$
where $H_B^i(\mathscr{U},\mu) = \prod_i H_B^i(U_i,\mu)$ is Betti
cohomology, using the comparison with \'etale cohomology.

\subsection{Purity in low degree}
\label{subsec:purity}

There is a canonical map $\Het^i(X,\mu) \to \Hur^i(X,\mu)$.  If this
map is injective, surjective, or bijective we say that the
\linedef{injectivity}, \linedef{weak purity}, or \linedef{purity}
property hold for \'etale cohomology in degree $i$, respectively, see
Colliot-Th\'el\`ene \cite[\S2.2]{colliot:santa_barbara}.

For $X$ smooth over a field $k$, a general cohomological purity
theorem for \'etale cohomology is established by Artin in
\cite[XVI~3.9,~XIX~3.2]{SGA4}.

\begin{theorem}
Let $X$ be a smooth variety over a field $k$ and $V \subset X$ a
closed subvariety of pure codimension $\geq c$.  Then the restriction
maps
$$
\Het^i(X,\mu_n^{\tensor j}) \to \Het^i(X \smallsetminus V,
\mu_n^{\tensor j})
$$
are injective for $i < 2c$ and are isomorphisms for $i < 2c-1$.
\end{theorem}

An immediate consequence (taking $c=1$) is that purity holds for
\'etale cohomology in degree $\leq 1$, i.e.\
$$
\Het^0(X,\mu) = \Hur^0(X,\mu) = \mu, \quad \text{and} \quad
\Het^1(X,\mu) = \Hur^1(X,\mu).
$$ 
See also Colliot-Th\'el\`ene--Sansuc
\cite[Cor.~3.2,~Prop.~4.1]{colliot-thelene_sansuc:type_multiplicatif}
for an extension to any geometrically locally factorial and integral scheme. 

Combining this (for $c=2$) with a cohomological purity result for
discrete valuation rings and a Mayer--Vietoris sequence, one can
deduce that for $X$ smooth over a field, weak purity holds for \'etale
cohomology in degree 2.  Moreover, there's a canonical identification
$\Br(X)' = \Hur^2(X,\QQ/\ZZ(1))$ by Bloch--Ogus~\cite{bloch_ogus} such
that the canonical map $\Het^2(X,\QQ/\ZZ(1)) \to \Hur^2(X,\QQ/\ZZ(1))
= \Br(X)'$ arises from the Kummer exact sequence.  Here, $\Br(X)'$
denotes the prime-to-characteristic torsion subgroup of the
(cohomological) Brauer group $\Br(X) = \Het^2(X,\Gm)$.  For $X$ a
smooth variety over $\CC$ (or in fact $X$ any complex analytic space),
there is a split exact sequence
$$
0 \to \bigl(H^2(X,\ZZ)/\im(\Pic(X) \to H^2(X,\ZZ)\bigr)\tensor \QQ/\ZZ \to
\Br(X) \to H^3(X,\ZZ)_\text{tors} \to 0
$$
arising from the exponential sequence.  In particular, there is a
(noncanonical) isomorphism $\Br(X) \isom (\QQ/\ZZ)^{b_2-\rho}\oplus
H^3(X,\ZZ)_{\text{tors}}$, where $b_2$ is the second Betti number and
$\rho$ the Picard rank of $X$.  If $X$ satisfies $H^2(X,\ko_X)=0$
(e.g., $X$ is rationally connected), then $\Pic(X) \to H^2(X,\ZZ)$ is
an isomorphism, hence $\Br(X) = H^3(X,\ZZ)_{\text{tors}}$. 

There is a beautiful interpretation of unramified cohomology in degree
3 in terms of cycles of codimension 2, going back to
Barbieri-Viale~\cite{barbieri-viale:CH2}.  Let $X$ be a smooth projective
variety over $k$.  We say that $\CH_0(X)$ is supported in dimension
$r$ if there exists a smooth projective variety $Y$ over $k$ of
dimension $r$ and a morphism $f : Y \to X$ such that the pushforward
$f_* : \CH_0(Y) \to \CH_0(X)$ is surjective.  For example, if
$\CH_0(X)=\ZZ$ (e.g., $X$ is rationally connected) then $X$ is
supported in dimension $0$.

\begin{theorem}[{Colliot-Th\'el\`ene--Voisin~\cite[Thm.~1.1]{colliot_voisin:unramified_cohomology}}]
Let $X$ be a smooth projective variety over $\CC$.  Assume that
$\CH_0(X)$ is supported in dimension 2.  Then there is an isomorphism
$$
\Hur^3(X,\QQ/\ZZ(2)) \isom \frac{H^{2,2}(X)\cap
H^4(X,\ZZ)}{\im\bigl(\CH^2(X) \to H^4(X,\ZZ)\bigr)}.
$$
Equivalently, the unramified cohomology of $X$ in degree $3$ is the
obstruction to the validity of the integral Hodge conjecture for
cycles of codimension $2$.
\end{theorem}

More generally, without the assumption that $\CH_0(X)$ is supported in
dimension 2, the torsion subgroup of $(H^{2,2}(X)\cap
H^4(X,\ZZ))/\im\bigl(\CH^2(X) \to H^4(X,\ZZ)\bigr)$ is a quotient of
$\Hur^3(X,\QQ/\ZZ(2))$ by a divisible subgroup.

There is also a version of this result valid over more
general fields, in particular over finite fields, due to
Colliot-Th\'el\`ene and Kahn~\cite{colliot-kahn:H3unramified} and extended
to higher codimension cycles by Pirutka~\cite{pirutka:CH2}.  Finally,
there is a description of $\Hur^4(X,\QQ/\ZZ(3))$ in terms of torsion in
$\CH^3(X)$, due to Voisin~\cite{voisin:unramified_cohomology_degree_4}.

\subsection{Triviality}

If $F/k$ is a field extension, we write $X_F = X \times_k F$.  If $X$
is geometrically integral over $k$, we say that $M_{\ur}(X/k)$ is
\linedef{universally trivial} if $M_{\ur}(X_{F}/F)$ is trivial for
every field extension $F/k$. Let $N$ be a positive integer. We say
that $M_\ur(X/k)$ is \linedef{universally $N$-torsion} if the cokernel
of the natural map $M(F) \to M_\ur(X_{F}/F)$ is killed by $N$ for
every field extension $F/k$.

\begin{prop}[{\cite[\S2~and~Thm.~4.1.5]{colliot:santa_barbara}}]
\label{prop:Pn}
Let $M : \Local_k \to \Ab$ be a functor satisfying the following
conditions:
\begin{itemize}
\item If $\OO$ is a discrete valuation ring containing $k$, with
fraction field $K$ and residue field $\kappa$, then $\ker\bigl(M(\OO)
\to M(K)\bigr) \subset \ker\bigl(M(\OO) \to M(\kappa)\bigr)$.

\item If $A$ is a regular local ring of dimension 2 containing $k$,
with fraction field $K$, then $\im\bigl(M(A) \to M(K)\bigr) = 
\bigcap_{\mathrm{ht}(\mathfrak{p})=1} \im\bigl(M(A_{\mathfrak{p}}) \to M(K)\bigr)$.

\item The group $M_\ur(\AA^1_{k}/k)$ is universally trivial.
\end{itemize}
Then $M_\ur(\PP^n_{k}/k)$ is universally trivial.  In particular, if
$X$ is a rational variety over $k$, then $M_\ur(X/k)$ is universally trivial.
\end{prop}

The functor $\Het^i(-,\mu)$ satisfies the conditions of
Proposition~\ref{prop:Pn} (cf.\
\cite[Thm.~4.1.5]{colliot:santa_barbara}), hence if $X$ is a
$k$-rational variety, then $\Hur^i(X,\mu)$ is universally trivial.
More generally, $\Hur^i(X,\mu)$ is universally trivial if $X$ is
stably rational, see~\cite[Prop.~1.2]{colliot-thelene_ojanguren}, or
even retract $k$-rational, which can be proved using
\cite[Cor.~RC.12--13]{KMApp}, see
\cite[Prop.~2.15]{merkurjev:unramified_cycle_modules}.

\subsection{Applications: Noether's problem and Artin--Mumford}

Here we describe two important examples where unramified cohomology
has been used in the rationality problem. 

\begin{example}
Let $G$ be a finite group, $V$ a finite dimensional linear
representation over $k$, and $k(V)$ the field of rational functions on
the affine space associated to $V$.  Then Noether's question asks if
the field of invariants $k(V)^G$ is purely transcendental over $k$,
equivalently, if the variety $V/G$ is rational.  This question was
posed by Emmy Noether in 1913, and has endured as one of the most
challenging rationality problems in algebraic geometry.  

Over the rational numbers, the problem takes on a very arithmetic
flavor.  Indeed, Noether's original motivation was the inverse
Galois problem, see \cite{swan:survey} for a survey in this
direction.  So we will focus on the case when $k$ is algebraically
closed of characteristic zero.

In this case, the question has a positive answer when $G$ is any
abelian group but is still open for $G=A_n$ for $n\geq 6$.
Saltman~\cite{saltman:noethers_problem} gave the first examples of
$p$-groups having a negative answer to Noether's question when $k$ is
algebraically closed.  While $V/G$ often has terrible singularities
and its smooth projective models are not easy to compute, nor feasible
to work with, the insight of Saltman was that one could still compute
unramified cohomology, in particular, $\Hur^2(k(V)^G/k,\QQ/\ZZ(1))$.
By the above purity results, if $X$ were a smooth proper model of
$k(V)^G$, then $\Hur^2(k(V)^G/k,\QQ/\ZZ(1)) = \Br(X)$.
Bogomolov~\cite{bogomolov:noethers_problem} gave a simple group
theoretic formula to compute $\Hur^2(k(V)^G/k,\QQ/\ZZ(1))$ purely in
terms of $G$, when $k$ is algebraically closed of characteristic not
dividing the order of $G$.  Because of this, this group is often
called the ``Bogomolov multiplier'' in the literature.

We point out that examples of groups $G$ where
$\Hur^i(k(V)^G/k,\QQ/\ZZ(1))$ is trivial for $i=2$ yet nontrivial for
$i \geq 3$ were first constructed by
Peyre~\cite{peyre:noethers_problem}.
\end{example}  

\begin{example}\label{ex:artin_mumford}
Artin and Mumford~\cite{artin_mumford} constructed a unirational
threefold $X$ over an algebraically closed field of characteristic
$\neq 2$ having nontrivial 2-torsion in $\Br(X)
=H^3(X,\ZZ)_{\mathrm{tors}}$, which by purity (see
\S\ref{subsec:purity}) coincides with $\Hur^2(X,\QQ/\ZZ(1))$, hence
such $X$ not retract rational (hence not stably rational).  The
``Artin--Mumford solid'' $X$ is constructed as the desingularization
of a double cover of $\PP^3$ branched over a certain quartic
hypersurface with 10 nodes, and is unirational by construction.  The
solid $X$ can also be presented as a conic bundle $X \to S$ over a
rational surface $S$.  The unramified cohomology perspective on the
examples of Artin and Mumford was further investigated by
Colliot-Th\'el\`ene and Ojanguren~\cite{colliot-thelene_ojanguren}.
\end{example}

Denote by $\Ab^{\bullet}$ the category of graded abelian groups.  An
important class of functors $M : \Local_k \to \Ab^{\bullet}$ arise
from the theory of \linedef{cycle modules} due to
Rost~\cite[Rem.~5.2]{rost:cycle_modules}.  In particular, unramified
cohomology arises from the \'etale cohomology cycle module, and to
some extent, the theory of cycle modules is a generalization of the
theory of unramified cohomology.  Rost's key observation is that
classical Chow groups appear as the unramified elements of the Milnor
$K$-theory cycle module.  The definition of cohomology groups arising
from cycle modules is very parallel to the definition of homology of a
CW complex from the singular chain complex.  
A cycle module $M$ comes
equipped with residue maps of graded degree $-1$
$$
M^i(k(X)) \mapto{\partial} \bigoplus_{x \in X^{(1)}} M^{i-1}(k(x)) 
$$
for any integral $k$-variety $X$.  If $X$ is smooth and proper, then
the group of unramified elements $M^i_{\ur}(X/k)$ is defined to be the kernel.

\subsection{Decomposition of the diagonal}

We say that a smooth proper variety $X$ of dimension $n$ over a field
$k$ has an \linedef{(integral) decomposition of the diagonal} if we
can write
\begin{equation}
\label{eq:decomp_diagonal}
\Delta_X = P \times X + Z
\end{equation}
in $\CH^n(X \times X)$, where $P$ is a $0$-cycle of degree $1$ and
$Z$ is a cycle with support in $X \times V$ for some closed
subvariety $V \subsetneq X$.  We say that $X$ has a \linedef{rational
decomposition of the diagonal} if there exists $N \geq 1$ such that
\begin{equation}
\label{eq:rat_decomp_diagonal}
N \Delta_X = P \times X + Z
\end{equation}
in $\CH^n(X\times X)$, where $P$ is a $0$-cycle of degree $N$ and
$Z$ is as before.  This notion was studied by Bloch and
Srinivas~\cite{bloch_srinivas}, where the idea goes back to Bloch
\cite[Lecture~1,~Appendix]{bloch:lectures}, in his proof of Mumford's
result on 2-forms on surfaces.

\begin{example}
The class of the diagonal $\Delta_{\PP^n} \in \CH^n(\PP^n \times
\PP^n)$ can be expressed in terms of the pull backs $\alpha,
\beta \in \CH^1(\PP^n\times \PP^n)$ of hyperplane classes from the two
projections.  The Chow ring is generated in terms of these
$\CH(\PP^n\times\PP^n) =
\ZZ[\alpha,\beta]/(\alpha^{n+1},\beta^{n+1})$ and one can compute that
$$
\Delta_{\PP^n} = \alpha^n + \alpha^{n-1}\beta + \dotsm + \alpha
\beta^{n-1} + \beta^n
$$
in $\CH^n(\PP^n \times \PP^n)$, see
\cite[Ex.~8.4.2]{fulton:intersection_theory}.  The class $\alpha^n$ is
the same as the class $P \times \PP^n$, for $P \in \PP^n$ a rational point,
while the classes $\alpha^i \beta^{n-i}$ for $i>0$ all have support on
$\PP^n \times H$, where $H \subset \PP^n$ is the hyperplane defining
$\beta$.  So any projective space has an integral decomposition of the
diagonal.
\end{example}

\begin{example}
Let $f : Y \to X$ be a proper surjective generically finite morphism
of degree $N$ between smooth quasi-projective varieties.
Then $(f\times f)_* \Delta_Y = N \Delta_X$.  Assume that $Y$ has a
decomposition of the diagonal $\Delta_Y = P \times Y + Z$, where $P$
is a $0$-cycle of degree 1 on $Y$ and $Z$ is a cycle with support on
$Y \times V$.  Then
$$
N\Delta_X = (f\times f)_* \Delta_Y = (f\times f)_* (P \times Y + Z) =
f_*(P)\times X + Z'
$$
where $f_*(P)$ is a $0$-cycle of degree $N$ and $Z'$ is a cycle on $X$
with support on $X \times f(V)$.
Hence $X$ has a rational decomposition of the diagonal.

Let $f : Y \to X$ be a surjective birational morphism between smooth
quasi-projective varieties.  Given a decomposition of the diagonal
$\Delta_X = P \times X + Z$ on $X$, by the moving lemma for
$0$-cycles, we can move $P$, up to rational equivalence, outside of
the image of the exceptional locus of $f$.  Then $(f\times
f)^*\Delta_X - \Delta_Y$ is a sum of cycles whose projections to $Y$
are contained in the exceptional locus of $f$.  But $(f\times
f)^*\Delta_X = (f \times f)^*(P \times X + Z) = f^{-1}(P) \times Y +
(f\times f)^*(Z)$, and $(f \times f)^*(Z)$ is a cycle with support on
$Y \times f^*(V)$.  In total, $Y$ has a decomposition of the diagonal.

We can use this to show that if $\PP^n \dashrightarrow X$ is a
unirational parameterization of degree $N$ over a field of
characteristic zero, then $X$ has a rational decomposition of the
diagonal $N \Delta_X = P \times X + Z$.  Indeed, by resolution of
singularities, we can resolve the rational map to a proper surjective
generically finite morphism $Y \to X$ of degree $N$, where $Y \to
\PP^n$ is a sequence of blow up maps along smooth centers.  By the
above considerations, the decomposition of the diagonal on $\PP^n$
induces one on $Y$, which in turn induces the desired rational
decomposition of the diagonal on $X$.

We remark that one can argue without the use of resolution of
singularities, but this is slightly more delicate.
\end{example}

\subsection{Decomposition of the diagonal acting on cohomology}
\label{subsec:decomp_acting}

A rational decomposition of the diagonal puts strong restrictions on
the variety $X$.  For example, the following result is well known.

\begin{prop}
\label{prop:bloch}
Let $X$ be a smooth proper geometrically irreducible variety over a
field $k$ of characteristic zero.  If $X$ has a rational decomposition
of the diagonal then $H^0(X,\Omega^{i}_{X})=0$ and $H^i(X,\OO_X)=0$
for all $i > 0$.
\end{prop}
Over a complex surface, this result goes back to Bloch's
proof~\cite[App.~Lec.~1]{bloch:lectures} of
Mumford's~\cite{mumform:2-forms} result on 2-forms on surfaces,
exploiting a decomposition of the diagonal and the action of cycles on
various cohomology theories (de Rham and \'etale). This argument was
further developed in~\cite{bloch_srinivas}. A proof over the complex
numbers can be found in~\cite[Cor.~10.18,~\S10.2.2]{voisin:Hodge}. A
variant of the argument for rigid cohomology in characteristic $p$ is
developed by Esnault~\cite[p.~187]{esnault:finite_field_trivial_Chow},
in her proof that rationally connected varieties over a finite field
have a rational point.  A variant of the argument using logarithmic de
Rham cohomology over any field is developed by Totaro
\cite[Lem.~2.2]{totaro:hypersurfaces} using the cycle class map of Gros.

Let $H^i(-)$ be a cohomology theory with a cycle class map
$$
\CH^i(X) \to H^{2i}(X)
$$
and a theory of correspondences (basically a Weil cohomology theory),
so that for any $\alpha \in H^{2n}(X \times X)$, where $n=\dim(X)$,
there is a map
$$
\alpha_* = q_*(\alpha . p^*) : H^i(X) \to H^i(X).
$$
for any $i \geq 0$.  Here $p$ and $q$ are the left and right
projections $X \times X \to X$, respectively.  When $\alpha=
[\Delta_X] \in H^{2n}(X\times X)$, then $\alpha_*$ is the identity
map.  When $\alpha = [P \times X] \in H^{2n}(X\times X)$, then
$\alpha_*$ factors through $H^i(P)$, i.e., there is a commutative
diagram
$$
\xymatrix{
H^i(X) \ar[d] \ar[r]^{\alpha_*} & H^i(X) \ar@{=}[d] \\
H^i(P)        \ar[r]^{\alpha_*} & H^i(X)
}
$$
where the left hand vertical map is the pullback by the inclusion of
the zero-dimensional subscheme $P \subset X$ (we have in mind $N$
times a point).  Assuming that $H^i(P)=0$ for $i > 0$, we get that
$N[\Delta_X]_* = [Z]_*$ on $H^i(X)$, assuming a rational decomposition
of the diagonal as in \eqref{eq:rat_decomp_diagonal}.

On the other hand, since $Z$ is a cycle supported on $X \times V$ for
$V \subset X$ a proper closed subvariety, the restriction of $[Z]$ to
$H^{2n}(X \times X \smallsetminus V)$ is zero.  Consider $\alpha = [Z]
\in H^{2n}(X\times X)$ and the commutative diagram
$$
\xymatrix{
H^i(X) \ar@{=}[d] \ar[r]^{\alpha_*} & H^i(X) \ar[d] \\
H^i(X)            \ar[r]^(.45){0_*} & H^i(X \smallsetminus V)
}
$$
where the right hand vertical arrow is the pullback by the inclusion
$X \smallsetminus V \subset X$ and the bottom horizontal arrow is the
pushforward associated to the restriction of $[Z]$ to $X \times (X
\smallsetminus V)$, which is zero.  Hence we have that $\alpha_*H^i(X)
\subset \ker(H^i(X) \to H^i(X \smallsetminus V))$.  If we additionally
assume that the cohomology theory has a localization sequence
$$
\dotsm \to H^i_V(X) \to H^i(X) \to H^i(X \smallsetminus V) \to \dotsm
$$
involving cohomology with supports, then we can also conclude that
$\alpha_*H^i(X) \subset \im(H^i_V(X) \to H^i(X))$.

Applying this to (algebraic) de Rham cohomology $\HdR^i(-)$, we have
that (by the degeneration of the Hodge-to-de Rham spectral sequence)
any cycle class $\alpha \in \HdR^{2n}(X \times X)$ lands in
$$
H^n(X\times X, \Omega_{X\times X}^n) =
H^n(X \times X, \bigoplus_{j} \Omega_X^j \boxtimes \Omega_X^{n-j}) =
\bigoplus_{i,j} H^i(X,\Omega_X^j) \otimes H^{n-i}(X, \Omega_X^{n-j})
$$ 
so $\alpha$ has a component in $H^0(X,\Omega_X^i)\otimes H^n(X,
\Omega_X^{n-i})$, which is isomorphic (by Serre duality) to
$\End(H^0(X,\Omega_X^i))$.  Thus this component of the pushforward
$\alpha_*$ defines a map $H^0(X,\Omega_X^i) \to H^0(X,\Omega_X^i)$,
whose image lands in $\ker(H^0(X,\Omega_X^i) \to H^0(X\smallsetminus
V,\Omega_{X\smallsetminus V}^i))$, which is trivial since the
restriction map to a Zariski open is injective on global differential
forms.  We conclude that ($N$ times) the identity on
$H^0(X,\Omega_X^i)$ coincides with the zero map, hence
$H^0(X,\Omega_X^i)=0$ for all $i>0$.  Instead using the cycle class
map of Gros in logarithmic de Rham cohomology, Totaro shows that even
in characteristic $p$, if $X$ has an (integral) decomposition of the
diagonal, then $H^0(X,\Omega_X^i)=0$ for all $i > 0$.

Applying this to the transcendental part of the cohomology
$$
H^2(X,\QQ_\ell)/\im(\NS(S)\tensor\QQ_\ell \to \Het^2(X,\QQ_\ell))
$$ 
of a surface $X$, Bloch shows that a rational decomposition of the
diagonal implies the vanishing of the transcendental part of the
cohomology.  This gave a new proof, via the Hodge-theoretic fact that
$p_g(X)=0$ is equivalent to $b_2(X)-\rho(X)=0$, of Mumford's theorem
that if
$p_g(X) > 0$ then the kernel of the degree map 
$\deg : \CH_0(X) \to \ZZ$ is not representable.

Applying this to Berthelot's theory of rigid cohomology, Esnault shows
that if $X$ is defined over a field of characteristic $p$ and has a
rational decomposition of the diagonal, then the Frobenius slope
$[0,1)$ part of the rigid cohomology of $H^i(X)$ is trivial for all $i
>0$.  If $X$ is defined over a finite field $\FF_q$, the Lefschetz
trace formula implies that $\# X(\FF_q) \equiv 1 \pmod q$, in
particular, $X(\FF_q) \neq \varnothing$.

Over $\CC$, an integral decomposition of the diagonal does not imply
$H^0(X, \omega_X^{\tensor n})=0$ for all $n > 1$.  Otherwise, a smooth
projective surface $X$ over $\CC$ with integral decomposition of the
diagonal would, aside from satisfying $p_g(X) = h^0(X,\omega_X)=0$ and
$q=h^1(X,\OO_X)=h^0(X,\Omega^1_X)=0$, would additionally satisfy
$P_2(X) = h^0(X,\omega_X^{\tensor 2})=0$, hence would be rational by
Castelnovo's criterion.  However, there do exist (nonrational) complex
surfaces $X$ of general type (e.g., Barlow surfaces) admitting an
integral decomposition of the diagonal that emerge in the context of
Bloch's conjecture on 0-cycles on surfaces, see
\S\ref{subsec:surfaces} for a more detailed discussion.

\section{Cubic threefolds and special cubic fourfolds}\label{sect:cubics}

Cubic hypersurfaces of dimension 3 and 4 are some of the most
important motivating objects in birational geometry since the last
half of the 20th century.  An irreducible cubic hypersurface is
rational as soon as it has a rational singular point, unless possibly
when it is a cone over a cubic of lower dimension, see~\cite[Chapter
1, Section 5, Example 1.28]{corti-kollar-smith}.  Working over an
algebraically closed field of characteristic not 3, we recall that in
dimension 1, a cubic hypersurface is not rational if and only if it is
smooth, in which case it is a curve of genus 1. In dimension 2, smooth
cubic hypersurfaces are rational, and they are realized
geometrically as the blow-up of six points in general position on
$\PP^2$.  In dimension 3, the fact that every smooth cubic
hypersurface is not rational is a celebrated theorem of Clemens and
Griffith~\cite{clemens_griffiths}. In dimension 4, some families of
smooth cubics hypersurfaces are known to be rational, while the very
general one is expected to be nonrational, even though not a single
one is currently provably nonrational.

Cubic hypersurfaces seem to occupy a space in the birational
classification of varieties that is ``very close'' to rational
varieties, in that their familiar cohomological and birational
invariants are similar to those of projective space.  Proving their
nonrationality seems to require the development of much finer
techniques.  The nonrationality of the cubic threefold was indeed one
of the first counterexamples to the L\"uroth problem (see Problem~\ref{luroth-problem}) in characteristic zero, and the proof of its
nonrationality required a deep study of algebraic cycles and the
intermediate Jacobian.  The study of the (non)rationality of cubic
fourfolds has already attracted Hodge and moduli-theoretic techniques,
and is undoubtedly one of the most famous open question in algebraic
geometry.

The aim of this section is to introduce ``classical'' constructions
arising in the study of cubic hypersurfaces. Intermediate Jacobians
will be presented in the first part.  In the second part, we recall
the Hodge theoretic approach to moduli spaces of cubic fourfolds,
and the known examples.  We work here exclusively over $\CC$.

\subsection{Intermediate Jacobians and cubic threefolds}\label{sect:intjacobians}
Let us recall from~\cite[Ch. 12]{voisin:Hodge} the definition of the intermediate Jacobians of a smooth
complex variety $X$ of dimension $n$. Consider the Betti cohomology group $H^i(X,\CC)$ together with the Hodge filtration
$F^p H^i(X,\CC)$. If $i=2j-1$ is odd, the $j$-th filtered module yields:
$$F^j H^{2j-1}(X,\CC) = \bigoplus_{p+q=2j-1,\,\, p \geq j} H^{p,q}(X).$$
The Hodge structure on Betti cohomology then gives that $H^{2j-1}(X,\CC)$ is the sum of $F^j H^{2j-1}(X,\CC)$ and its conjugate, so that
$$H^{2j-1}(X,\ZZ)/{\mathrm{Tors}} \longrightarrow H^{2j-1}(X,\CC)/ F^j H^{2j-1}(X,\CC)$$
is an injective map (via de Rham cohomology). We define the ($2j-1$)-st intermediate Jacobian $J^{2j-1}(X)$ as the quotient of the $\CC$-vector space
$F^j H^{2j-1}(X,\CC)$ by this lattice. The Jacobian is in general a complex torus, and not an Abelian
variety.

If $X$ is a threefold with $H^1(X,\CC)=0$, then the only nontrivial Jacobian is $J^3(X)$.
Indeed, by Poincar\'e duality $H^1(X,\CC)=H^5(X,\CC)=0$, so that $J^1(X)=J^5(X)=0$. We denote then $J(X):=J^3(X)$.
Moreover, assume $X$ is a Fano, or in general a threefold with $H^1(X,\CC)=H^{3,0}(X)=0$.
The key idea of Clemens and Griffiths~\cite{clemens_griffiths} is to show that in this case 
the complex torus $J(X)$ is an abelian variety endowed with a canonical principal polarization. Let us briefly sketch a proof of that fact,
loosely following the presentation in~\cite[3.1]{voisin:exposition}.
The cup product gives a unimodular intersection pairing $\sod{-,-}$ on $H^3(X,\ZZ)/\mathrm{Tors}$.
Moreover, consider any nontrivial $(2,1)$-cohomology classes $\alpha,\beta \in H^{2,1}(X)$. Recall we assume that $H^{1,0}(X)=0$, so that 
\begin{equation}
\sod{\alpha,\beta} = 0, \,\,\,\,\,\,\,\,\,\,\,\,\, -\sqrt{-1}\sod{\alpha,\overline{\alpha}} > 0,
\end{equation}
since the cup product is Hermitian and skew symmetric and respects the Hodge decomposition (see~\cite[7.2.1]{voisin:Hodge} for more details).
It follows that $\sod{-,-}$ can be identified with the first Chern class of an ample line bundle $L$ on $J(X)$, via the identification
$\bigwedge^2 H^1(J(X),\ZZ) \simeq H^2(J(X),\ZZ)$ (see~\cite[Ch. I, 3]{mumford-abelian} for more details), and the line
bundle $L$ is well-defined up to translation. In particular $J(X)$ is an Abelian variety.
Moreover, since the cup product is unimodular, $H^0(J(X),L)$ is one dimensional. Hence $L$ is a Theta divisor
for $J(X)$, which is then principally polarized. A famous result proved by Clemens and Griffiths~\cite{clemens_griffiths}
shows that one can extract a birational invariant from this Abelian variety.

\begin{theorem}[Clemens--Griffiths \cite{clemens_griffiths}]
\label{thm:J-is-bir-inv}
If a complex threefold $X$ is rational, then there exist smooth
projective curves $\{C_i\}_{i=1}^r$ and an isomorphism of principally
polarized Abelian varieties
$$J(X) = \bigoplus_{i=1}^r J(C_i).$$
Moreover, if $X$ is a complex threefold with
$H^{3,0}(X)=H^{1,0}(X)=0$, there is a well-defined principally
polarized Abelian subvariety $A_X \subset J(X)$ which is a birational
invariant: if $X' \dashrightarrow X$ is a birational map, then $A_{X'}
\simeq A_X$ as principally polarized Abelian varieties.
\end{theorem}
\begin{proof}[Sketch of proof]
It is enough to define $A_X$ and prove the second statement, which is stronger. Indeed, it is easy to see that 
$J(\PP^3)=0$, so that $A_{\PP^3}=0$. The splitting of the intermediate Jacobian of a rational threefold
will then be evident by the definition of $A_X$.

Clemens and Griffiths show that the category of principally polarized Abelian varieties is semisimple, that is
any injective morphism is split (see \cite[\ 3]{clemens_griffiths}). We work exclusively in this category.
Hence we define $A_X$ as follows: any injective map $J(C) \to J(X)$ for
$C$ a smooth curve gives a splitting $J(X)=A\oplus J(C)$. There is hence a finite number of
curves $\{C_i\}_{i=1}^r$ with $J(C_i) \neq 0$ (i.e., $g(C_i) >0$) and a splitting $J(X)=A_X \oplus \bigoplus_{i=1}^r J(C_i)$,
such that there is no nontrivial morphism $J(C) \to A_X$ for any smooth projective curve $C$.
By semisimplicity of the category of principally polarized abelian varieties, we get that $A_X$ is
well defined.

If we consider $\rho: Y \to X$ a birational morphism, we can show that $\rho^*: J(X) \to J(Y)$ is an
injective map. Then $J(Y)=J(X) \oplus A$ for some
Abelian variety $A$. If moreover $\rho$ is the blow-up along of a
point, then $J(Y) \simeq J(X)$. If $\rho$ is the blow-up along a smooth curve $C$, then $J(Y) = J(X) \oplus J(C)$.

Consider the birational map $X' \dashrightarrow X$. By Hironaka's resolution of singularities, there is a smooth projective
$X_1$ with birational morphisms $\rho_1: X_1 \to X'$ and $\pi_1: X_1 \to X$, such that $\pi_1$ is a composition of
a finite number of smooth blow-ups. We denote $\{C_i\}_{i=1}^s$ the curves blown-up by $\pi_1$.
Similarly, there are $\rho_2: X_2 \to X$ and $\pi_2: X_2 \to X'$ birational maps with
$\pi_2$ a composition of a finite number of blow-ups. We denote $\{D_i\}_{i=1}^t$ the curves blown-up by $\pi_2$.
It follows that looking at the decompositions of $J(X_1)$ and $J(X_2)$ respectively, we have:
$$\begin{array}{c}
J(X) \subset J(X') \oplus J(D_1) \oplus \ldots \oplus J(D_t) \\
J(X') \subset J(X) \oplus J(C_1) \oplus \ldots \oplus J(C_s)
\end{array}$$
and we conclude, by semisimplicity of the category of principally polarized abelian varieties, that we must have $A_X = A_{X'}$.
Indeed the first equation gives $A_X \subset A_{X'}$, and the second one gives $A_{X'} \subset A_X$.
\end{proof}

The first statement of Theorem~\ref{thm:J-is-bir-inv} provides the
\linedef{Clemens--Griffiths nonrationality criterion}, namely that if
the intermediate Jacobian of a smooth projective threefold $X$ with
$h^1 = h^{3,0} = 0$ does not factor (in the category of principally
polarized abelian varieties) into a product of Jacobians of curves,
then $X$ is not rational.
The first application of this criterion is the proof of the nonrationality of
a smooth cubic threefold~\cite{clemens_griffiths}.

\begin{theorem}[Clemens--Griffiths]\label{thm:cubic3fd-notrat}
Let $X$ be a smooth cubic threefold.  The principally polarized abelian
variety $J(X)$ is not split by Jacobians of curves. In particular,
$X$ is not rational.
\end{theorem}

We will not give here a proof of Theorem~\ref{thm:cubic3fd-notrat}, but
just mention that it relies on the careful study of singularities the Theta-divisor of $J(X)$,
which is a five-dimensional Abelian variety. Just to mention the huge amount of
interesting mathematics appearing in this context, we notice that this question is also related to
the \linedef{Schottky problem}, that is the study of the moduli of Jacobians inside the
moduli space of principally polarized Abelian varieties.

Clemens--Griffiths nonrationality criterion applies to any threefold
with $H^{1,0}(X)=H^{3,0}(X)=0$, and has, for example, allowed
Beauville~\cite{beauvilleprym} and Shokurov~\cite{shokuprym} to
completely classify rational conic bundles over minimal surfaces.
We recall that a conic bundle is \linedef{standard} if the fiber over
any irreducible curve is an irreducible surface (this is equivalent
to relative minimality).

\begin{theorem}[Beauville, Shokurov]\label{thm:Jac-of-cbs}
Let $X \to S$ be a relatively minimal conic bundle, with $X$ smooth,
over a smooth minimal rational surface $S$ with discriminant divisor
$C \subset S$ having at most isolated nodal singularities. Then $X$ is
rational if and only if $J(X)$ is split by Jacobians of curves, and
this happens only in five cases (besides projective bundles):
\begin{itemize}
 \item $S$ is a plane, and $C$ is a cubic, or a quartic, or
 a quintic and the discriminant double cover $\widetilde{C} \to C$ is given by an even theta-characteristic in the latter case.
 \item $S$ is a Hirzebruch surface and the fibration $S \to \PP^1 $ induces either
 a hyperelliptic or a trigonal structure $C \to \PP^1$ on the discriminant divisor.
\end{itemize}
\end{theorem}

The proof of Theorem~\ref{thm:Jac-of-cbs} relies on the isomorphism $J(X) \simeq \mathrm{Prym}(\widetilde{C}/C)$
as principally polarized Abelian varieties~\cite{beauvilleprym} and on the study of Prym varieties.
Notice that Theorem~\ref{thm:Jac-of-cbs} recovers Theorem~\ref{thm:cubic3fd-notrat} since the
blow-up of a smooth cubic threefold $X$ along any line $l \subset X$ gives a relatively
minimal conic bundle $\widetilde{X} \to \PP^2$ whose discriminant divisor $C$ is a smooth
quintic and $\widetilde{C} \to C$ is given by an odd theta-characteristic.
As recalled $J(X) \simeq \mathrm{Prym}(\widetilde{C}/C)$, so one can fairly say then that cubic
threefolds are (birationally) the non-rational conic bundles with the smallest intermediate Jacobian.

\subsection{Intermediate Jacobians and the Zariski problem}
\label{subsec:BCTSSD}

Another important problem where the method of the intermediate
Jacobian has been successful is in constructing the first
counterexamples to Problem~\ref{zariski-problem}, posed by Zariski in 1949, see
\cite{segre:Zariski_problem}.
Indeed, using Prym variety considerations, Beauville,
Colliot-Th\'el\`ene, Sansuc, and Swinnerton-Dyer~\cite{bctssd} used
the intermediate Jacobian to construct the first example of a
nonrational but stably rational variety, a fibration in Ch\^atelet
surfaces $V \to \PP^1$ with affine model
$$
y^2-\delta(t)z^2 = P(x,t)
$$
where $P(x,t) = x^3 + p(t)x + q(t)$ is an irreducible polynomial in
$\CC[x,t]$ such that its discriminant $\delta(t) = 4p(t)^3+27q(t)^2$
has degree $\geq 5$.  They proved, using the intermediate Jacobian,
that $V$ is not rational, yet that $V \times \PP^3$ is rational.
Shepherd-Barron~\cite{shepherd-barron:stably_rational} used a slightly
different construction to prove that $V \times \PP^2$ is rational.  It
is unknown whether $V\times \PP^1$ is rational.

The key point is that the Clemens--Griffiths criterion for
irrationality of a threefold using the intermediate Jacobian is not a
stable birational invariant. Indeed, it strictly applies to threefolds.

\subsection{(Special) cubic fourfolds, Hodge theory and Fano schemes of lines}\label{subs:cubic-4folds}

We turn our attention to smooth cubic fourfolds and the Hodge
structure on their middle cohomology.  The ideas we present in this
section go back to Beauville--Donagi~\cite{beauville-donagi} and
Hassett~\cite{hassett:special-cubics,hassett:rational_cubic}.  Let $X
\subset \PP^5$ be a smooth cubic fourfold. We denote by $h \in
H^2(X,\ZZ)$ the Betti cycle class of a hyperplane section of $X$. In
particular, $h^4=3$ and
$$H^2(X,\ZZ) = \ZZ[h], \qquad H^6(X,\ZZ)= \ZZ[h^3/3],$$
by the Lefschetz hyperplane theorem and Poincar\'e duality. One can,
moreover, calculate the Hodge numbers: the (upper half) Hodge diamond of
$X$ has the following shape:
$$\begin{array}{lccccccccc}
H^8(X,\CC)\phantom{(X,\CC)} & & & & & 1 & & & & \\
H^7(X,\CC) & & & & 0 & & 0 & & & \\
H^6(X,\CC) & & & 0 & & 1 & & 0 & &\\
H^5(X,\CC) & & 0 & & 0 & & 0 & & 0 & \\
H^4(X,\CC) & 0 & & 1 & & 21 & & 1 & & 0.
\end{array}$$
We then focus on the cohomology lattice $H^4(X,\ZZ)$, endowed with the intersection pairing $\sod{-,-}$,
and we denote by $H^4_{\pr}(X,\ZZ)$ the primitive cohomology sublattice. In particular, we have that
$h^2 \in H^4(X,\ZZ)$ and that $H^4_{\pr}(X,\ZZ)=\sod{h^2}^\perp$.

Rational, algebraic, and homological equivalence all coincide for
cycles of codimension 2 on any smooth projective rationally connected
variety $X$ over $\CC$ satisfying $H^3(X,\ZZ/l\ZZ)=0$ for some prime
$l$, cf.\ \cite[Prop.~5.1]{colliot:descenteCH2}.  Hence for a smooth
cubic fourfold $X$, the Betti cycle class map $ \CH^2(X) \to
H^4(X,\ZZ) $ is injective.  The image of the cycle class map is
contained in the subgroup of Hodge classes $H^{2,2}(X) \cap
H^4(X,\ZZ)$.  In particular, $\CH^2(X)$, with its intersection
product, is a sublattice of $H^{2,2}(X) \cap H^4(X,\ZZ)$, which is
positive definite by the Riemann bilinear relations.  In fact, the
cycle class map induces an isomorphism $\CH^2(X) = H^{2,2}(X) \cap
H^4(X,\ZZ)$ by the integral Hodge conjecture for cycles of codimension
2 on smooth cubic fourfolds proved by
Voisin~\cite[Thm.~18]{voisin:aspects}, building on \cite{murre} and
\cite{zucker}.

To study the cohomology lattice, we consider the Fano variety of lines
$F(X)$, defined to be the subvariety $F(X) \subset \mathrm{Gr}(2,6)$
parameterizing the lines contained in $X$.  Then $F(X)$ is a smooth
fourfold. Despite its name, $F(X)$ is an irreducible holomorphic
symplectic (IHS) variety, as shown by Beauville and
Donagi~\cite[Prop. 2]{beauville-donagi}.

The cohomology of $F(X)$ and $X$ are related by an Abel--Jacobi map,
as follows. Denote by $Z \subset X \times F(X)$ the universal line
over $X$, and consider the diagram:
$$\xymatrix{
& Z \ar[dr]^q \ar[dl]_p & \\
X & & F(X), }$$
where $p$ and $q$ denote the restrictions to $Z$ of the natural
projections from $X \times F(X)$ to $X$ and $F(X)$ respectively.  The
Abel--Jacobi map $\alpha: H^4(X,\ZZ) \to H^2(F(X),\ZZ)$ is defined as
$\alpha= q_* p^*$.  Since $F(X)$ is an IHS variety, $H^2(X,\ZZ)$ is
endowed with a bilinear form, which we will denote by
$\sod{-,-}_{BB}$, the Beauville--Bogomolov form. Moreover, since $F(X)
\subset \mathrm{Gr}(2,6)$, we can restrict the class on
$\mathrm{Gr}(2,6)$ defining the Pl\"ucker embedding to a class $g \in
H^2(F(X),\ZZ)$.  Define $H^2_0(F(X),\ZZ) = \sod{g}^\perp \subset
H^2(F(X),\ZZ)$ to be the orthogonal complement of $g$ with respect to
the Beauville--Bogomolov form.  One checks that
$\alpha(h^2)=g$. Moreover, Beauville and Donagi establish an
isomorphism of Hodge structures~\cite{beauville-donagi}.

\begin{theorem}[Beauville--Donagi~\cite{beauville-donagi}]
The Abel--Jacobi map $\alpha: H^4_{\pr}(X,\ZZ) \to H^2_0(F(X),\ZZ)$ satisfies: 
$$\sod{\alpha(x),\alpha(y)}_{BB}=-\sod{x,y}.$$
In other words, $\alpha$ induces an isomorphism of Hodge structures:
$$H^4_{\pr}(X,\CC) \simeq H^2_0(F(X),\CC)(-1).$$
\end{theorem}

For a smooth projective surface $S$ and a positive integer $n$, we
write $S^{[n]} = \Hilb^{n}(S)$ for the Hilbert scheme of length $n$
subscheme on $S$, which is a smooth projective variety. Beauville and
Donagi describe the deformation class of $F(X)$.

\begin{theorem}[Beauville--Donagi]\label{thm:FX-deform}
The Fano variety of lines $F(X)$ is an irreducible holomorphic
symplectic variety deformation equivalent to $S^{[2]}$, where $S$ is a
degree 14 K3 surface.
\end{theorem}

One possible interpretation of the results of Beauville and Donagi is
that the variety $F(X)$ acts as a Hodge-theoretic analogue for the
intermediate Jacobian of a cubic threefold.

The proof of Theorem~\ref{thm:FX-deform} proceeds via a deformation
argument to the case where $X$ is a \linedef{Pfaffian cubic fourfold},
as follows.  Let $V$ be a $6$-dimensional complex vector space and
consider $\mathrm{Gr}(2,V) \subset \PP(\bigwedge^2 V)$ via the
Pl\"ucker embedding.  The variety $\mathrm{Pf}(4,\bigwedge^2 V^*)
\subset \PP(\bigwedge^2 V^*)$ is defined as (the projectivization of)
the set of degenerate skew-symmetric forms on $V$, which is isomorphic
to the set of skew symmetric $6\times 6$ matrices with rank bounded
above by 4. It is a (nonsmooth) cubic hypersurface of $\PP(\bigwedge^2
V^*)$ defined by the vanishing of the Pfaffian. Let $L \subset
\PP(\bigwedge^2 V)$ be a linear subspace of dimension 8, and denote by
$L^* \subset \PP(\bigwedge^2 V)$ its orthogonal subspace, which has
dimension $5$. If we take $L$ general enough, then $X = L^* \cap
\mathrm{Pf}(4,\bigwedge^2 V^*)$ is a smooth cubic fourfold in
$L^*=\PP^5$ and $S = L \cap \mathrm{Gr}(2,V)$ is a smooth K3 surface
in $L=\PP^8$ with a degree 14 polarization $l$. Cubic fourfolds
arising from this construction are called \linedef{Pfaffian} with
associated K3 surface $S$.  Then Beauville and Donagi prove
Theorem~\ref{thm:FX-deform} directly for pfaffian cubic
fourfolds\footnote{The fact that any Pfaffian cubic not containing a
plane has the properties required by Beauville and Donagi's proof was
proved recently by Bolognesi and Russo
\cite{bolognesi_russo_stagliano}.}.

\begin{theorem}[Beauville--Donagi]\label{thm:FX-pfaffian}
Let $X$ be a Pfaffian cubic fourfold with associated K3 surface $S$,
not containing a plane. Then $X$ is rational and $F(X)$ is isomorphic
to $S^{[2]}$.
\end{theorem}

Theorem~\ref{thm:FX-deform} is then obtained by a deformation argument
from Theorem~\ref{thm:FX-pfaffian}. The proofs of the two facts
stated in Theorem~\ref{thm:FX-pfaffian} both rely on the explicit
geometric construction of $X$ and $S$, and do not, on the face of it,
seem to be related. However, this result hints at a deep relationship
between the Fano variety of lines, K3 surfaces, and the birational
geometry of cubic fourfolds.

Hassett's work \cite{hassett:special-cubics} is based on the study
of the Hodge structure and the integral cohomology lattice of a smooth
cubic fourfold $X$. A key observation of Beauville and Donagi is that
being Pfaffian implies the existence of a rational normal quartic
scroll inside $X$, in fact a two dimensional family of such scrolls
parameterized by $S$.  In fact, cubic fourfolds containing rational
normal quartic scrolls, and their rationality, were already considered
by Fano~\cite{fano}.  Hassett's key idea is to consider the class of
such a ruled surface in $H^4(X,\ZZ)$ and the lattice-theoretic
properties that one can deduce from its existence.

Consider the integral cohomology lattice $H^4(X,\ZZ)$ and its
sublattice $H^4_{\pr}(X,\ZZ)$. Recall that $F(X)$ is a deformation of
$S^{[2]}$ and that $H^2(S^{[2]},\ZZ)= H^2(S,\ZZ) \oplus \ZZ[\delta]$,
with $\sod{\delta,\delta}_{BB}=-2$, is an orthogonal decomposition.
In particular,
$$
H^2(F(X),\ZZ) \simeq U^{\oplus 3} \oplus E_8^{\oplus 2} \oplus (-2),
$$
where $U$ is the hyperbolic lattice, $E_8$ is the lattice associated
to the Dynkin diagram of type $E_8$, and $(-2)$ is the rank one
primitive sublattice generated by $\delta$.  This allows one to
calculate the lattice $H^4_{\pr}(X,\ZZ)$ via the Abel--Jacobi
map.

\begin{prop}[Hassett~\cite{hassett:special-cubics}]\label{prop:prim-lattice-cubic4}
The integral primitive cohomology lattice of a cubic fourfold is
$$
H^4_{\pr}(X,\ZZ) \simeq B \oplus U^{\oplus 2} \oplus E_8^{\oplus 2},
$$
where $B$ is a rank 2 lattice with intersection matrix:
$$\left(\begin{array}{cc}
         2 & 1 \\
         1 & 2
        \end{array}\right).$$
In particular, $H^4_{\pr}(X,\ZZ)$ has signature $(20,2)$.        
\end{prop}

It follows from Proposition~\ref{prop:prim-lattice-cubic4} that,
though $H^4_{\pr}(X,\ZZ)$ has the same rank as a (Tate twist of a) K3
lattice, their signatures differ, since the latter has signature
$(19,3)$. However, one should be tempted to wonder whether, or under
which conditions, it is possible to find a K3 surface $S$ and
isomorphic sublattices of signature $(19,2)$ of $H^4_{\pr}(X,\ZZ)$ and
of $H^2(S,\ZZ)$.  On the surface side, there is a very natural (and
geometrically relevant) candidate: if $l$ is a polarization on $S$,
then the primitive cohomology $H^2_{\pr}(S,\ZZ) =\sod{l}^\perp$ could
be a candidate to consider.

For example, let $X$ be a Pfaffian cubic fourfold and and $S$ an
associated K3 surface with its polarization $l$ of degree 14.  As
recalled, $X$ contains a homology class of rational normal quartic
scrolls parameterized by $S$. Let $T \in H^4(X,\ZZ)$ be the cohomology
class of this $2$-cycle. In particular, $T$ is not homologous to
$h^2$, hence we have a rank 2 primitive sublattice $K$, generated by
$T$ and $h^2$, of $H^4(X,\ZZ)$.  As $T.T = c_2(N_{T/X})=10$, we have
that the intersection matrix of $K$ is
$$
\begin{pmatrix}
          3 & 4 \\
          4 & 10
\end{pmatrix},
$$
hence has determinant $14$, equal to the degree of the polarized K3
surface $S$ associated to the Pfaffian construction of $X$.  The key
remark of Hassett is that $K^\perp \subset H^4(X,\ZZ)$ and
$l^{\perp}=H^2_{\pr}(X,\ZZ) \subset H^2(X,\ZZ)$ are isomorphic
lattices (up to a Tate twist) in this case. This motivates the
following definition.

\begin{definition}[Hassett]
A cubic fourfold $X$ is \linedef{special} if it contains an algebraic
$2$-cycle $T$, not homologous to $h^2$, i.e., if the rank of
$\CH^2(X)$ is at least 2.

Given an abstract rank 2 positive definite lattice $K$ with a
distinguished element $h^2$ of self-intersection 3, a
\linedef{labeling} of a special cubic fourfold is the choice of a
primitive embedding $K \hookrightarrow \CH^2(X)$ identifying the
distinguished element with the double hyperplane section $h^2$. The
\linedef{discriminant} of a labeled special cubic fourfold $(X,K)$ is
defined to be the determinant of the intersection matrix of $K$; it is
a positive integer.  Note that a cubic fourfold could have labelings
of different discriminants.

Let $(X,K)$ be a labeled special cubic fourfold. A polarized K3
surface $(S,l)$ is \linedef{associated} to $(X,K)$ if there is an
isomorphism of lattices $K^\perp \simeq H^2_{\pr}(S,\ZZ)(-1)$.
\end{definition}

\begin{example}[Hassett]\label{ex:c14}
If $X$ is a Pfaffian cubic fourfold with associated polarized K3
surface $(S,l)$ of degree 14, then $X$ is special, has a labeling $K$
of discriminant $14$ defined by the class of the rational normal
quartic scrolls parameterized by $S$, and $(S,l)$ is associated to
$(X,K)$.

On the other hand, Bolognesi and Russo
\cite[Thm. 0.2]{bolognesi_russo_stagliano} have shown that any special
cubic fourfold of discriminant $14$ not containing a plane is
Pfaffian.  It is also known that any special cubic fourfold of
discriminant 14 that is not Pfaffian must contain two disjoint planes,
see \cite{auel:pfaffian}.
\end{example}

\begin{example}[Hassett~\cite{hassett:special-cubics},~\cite{hassett:rational_cubic}]\label{ex:cubicwithplane}
Let $X$ be a smooth cubic fourfold containing a plane $P \subset
\PP^5$.  Such $X$ is special, as $P$ is not homologous to $h^2$.
Since $P.P= c_2(N_{P/X})=3$, we have that the sublattice $K$ generated
by $h^2$ and $P$ has intersection matrix
$$
\begin{pmatrix}
3 & 1 \\
1 & 3 
\end{pmatrix},
$$
so defines a labeling of discriminant $8$. In general, there is no K3
surface associated to this labeled cubic fourfold $(X,K)$ (see
Theorem~\ref{thm:hassett-all}).

Consider the projection $\PP^5 \dashrightarrow \PP^2$ from the plane
$P$. Restricting this projection to $X$ gives rise to a rational map
$X \dashrightarrow \PP^2$ which can be resolved, by blowing up $P$,
into a quadric surface bundle $\pi:\widetilde{X} \to \PP^2$,
degenerating along a (generically smooth) sextic curve $C \subset
\PP^2$.  The double cover $S \to \PP^2$ branched along $C$ is a K3
surface with a polarization of degree 2, which plays a r\^ole in the
Hodge theory of $X$, but is not associated to $(X,K)$.
\end{example}

Using the period map and the Torelli Theorem (see~\cite{voisin}) for
cubic fourfolds, one can construct a 20-dimensional (coarse) algebraic
moduli space $\cC$ of smooth cubic fourfolds, as explained in
\cite[2.2]{hassett:special-cubics}. Using this algebraic structure,
Hassett shows that the very general cubic fourfold is not special, and
that the locus of special cubic fourfolds of fixed discriminant is a
divisor of $\cC$, which might be empty depending on the value of the
discriminant.  Hassett also finds further restrictions on the
discriminant for special cubic fourfolds having associated K3 surfaces
$S$ and for which $F(X)$ is isomorphic to $S^{[2]}$.

\begin{theorem}[Hassett~\cite{hassett:special-cubics}]\label{thm:hassett-all}
\label{thm:hassett:special}
Special cubic fourfolds of discriminant $d$ form a nonempty
irreducible divisor $\cC_d \subset \cC$ if and only if $d > 0$ and $d
\equiv 0,2 \bmod 6$.

Special cubic fourfolds of discriminant $d > 6$ have associated K3
surfaces if and only if $d$ is not divisible by $4$, $9$, or any odd
prime $p \equiv 2$ modulo $3$.

Assume that $d = 2(n^2 + n + 1)$ where $n \geq 2$ is an integer, and
let $X$ be a generic special cubic fourfold of discriminant $d$, in
which case $X$ has an associated K3 surface $S$.  Then there is an
isomorphism $F(X) \simeq S^{[2]}$.
\end{theorem}

\begin{remark}
The condition on $d>6$ ensures that $X$ is smooth. For completeness,
the very low-discriminant cases are known: a cubic fourfold of
discriminant $2$ is determinantal (and hence is singular along a
Veronese surface), see~\cite[4.4]{hassett:special-cubics}; a cubic
fourfold of discriminant $6$ has a single ordinary double point,
see~\cite[4.2]{hassett:special-cubics}. The loci $\cC_6$ and $\cC_2$
do not lie in the moduli space $\cC$, but rather in its boundary (see
\cite{laza:period_cubic_fourfold} and
\cite{looijenga:period_cubic_fourfold}).
\end{remark}

The last statement in Theorem~\ref{thm:hassett:special} can be made
more precise, once one weakens it by asking that $F(X)$ is not
isomorphic but just birational to $S^{[2]}$. The numerical necessary
and sufficient condition for this was established by
Addington~\cite{addington-twoconj}.

\begin{theorem}[Addington]\label{thm:add-two-cond}
Let $X$ be a special cubic fourfold of discriminant $d$, with associated K3 surface $S$. Then
$F(X)$ is birational to $S^{[2]}$ is and only if $d$ is of the form:
$$d = \frac{2n^2+2n+2}{a^2},$$
for some $n$ and $a$ in $\ZZ$.
\end{theorem}

As noticed by Addington~\cite{addington-twoconj}, having an associated
K3 surface does not necessarily imply that $F(X)$ is birational to
$S^{[2]}$. The numerical condition from the second statement of
Theorem~\ref{thm:hassett-all} is indeed strictly stronger than the
numerical condition from Theorem~\ref{thm:add-two-cond}. The smallest
value of $d$ for which a special cubic of discriminant $d$ has an
associated K3 surface $S$ but $F(X)$ is not birational to $S^{[2]}$ is $74$.

Let us recall the known examples of rational cubic fourfolds, in order
to consider a Hodge-theoretic expectation about rationality.

\begin{example}\label{ex:known-rat-cub}
Let $X$ be a cubic fourfold. If either
\begin{itemize}
 \item[2,6)] $X$ is singular, e.g. $X \in \cC_6$ has a single node or $X \in \cC_2$ is determinantal; or
 \item[8)] $X$ contains a plane $P$, so that $X \in \cC_8$, and the associated quadric surface fibration $\widetilde{X} \to \PP^2$
 (see Example~\ref{ex:cubicwithplane}) admits a multisection of odd degree~\cite{hassett:rational_cubic}; or
 \item[14)] $X$ is Pfaffian, so that $X \in \cC_{14}$~\cite{beauville-donagi};
\end{itemize}
then $X$ is rational.\footnote{A new class of rational cubic fourfolds
$X$ has very recently been constructed by Addington, Hassett,
Tschinkel, and V\'arilly-Alvarado~\cite{AHTVA:dP6}, these are in
$\cC_{18}$ and are birational to a fibration $\widetilde{X} \to \PP^2$
in sextic del Pezzo surfaces admitting a multisection of degree prime
to $3$.}

In particular, all cubics in $\cC_2$ or $\cC_6$, and the general cubic
in $\cC_{14}$ are rational\footnote{In fact, every cubic fourfold in
$\cC_{14}$ is rational, this was proved independently by
\cite{bolognesi_russo_stagliano} and \cite{auel:pfaffian}.}. The
cubics in $\cC_8$ satisfying condition 8) form a countable union of divisors in $\cC_8$~\cite{hassett:rational_cubic}.
\end{example}

Let $X$ be a cubic containing a plane, and $\widetilde{X} \to \PP^2$
the associated quadric fibration. Having an odd section for
$\widetilde{X}$ is a sufficient, but not necessary condition for
rationality. Indeed, there exist Pfaffian cubics in $\cC_8$ such that
$\widetilde{X} \to \PP^2$ doesn't have any odd section. Such cubics
are then rational, they lie in the intersection $\cC_8 \cap \cC_{14}$
and were constructed in~\cite{ABBV}.

\smallskip

If one imagines that $H^4_{\pr}(X,\ZZ)$, with its Hodge structure,
plays the r\^ole that the intermediate Jacobian plays for cubic
threefolds, then one would naturally expect that having no associated
K3 surface should be an obstruction to rationality.  For more on this
perspective, see the recent survey~\cite[\S
3]{hassett-cubicsurvey}\footnote{A sample result showing the interplay
between Hodge theory and rationality is provided by Kulikov, who has
shown that Hodge-indecomposability of the transcendental cohomology
would be a sufficient condition of nonrationality for $X$,
see~\cite{kulikov}. However, such indecomposability was recently shown
not to hold in~\cite{auel-bohning-bothmer}.}.  On the other hand,
there is no known example of a nonrational cubic fourfold, and few
general families of rational ones. We should then be very cautious to
wonder whether having an associated K3 surface is a sufficient
criterion of rationality. Hassett has recently asked about the
existence of other examples of rational cubic fourfolds~\cite[Question
16]{hassett-cubicsurvey}.

On the other hand, as we will see in Section~\ref{sect:high-dim-Db},
Kuznetsov's conjecture~\cite{kuz:4fold} is equivalent, at least for a
generic cubic fourfold, to the statement that the rationality of $X$
is equivalent to the existence of an associated K3, as shown by
Addington and Thomas \cite{at12}.  As we will see later,
decompositions of the derived category of a cubic fourfold increase
the amount of evidence motivating the expectation that having an
associated K3 surface should be a necessary condition for
rationality. Then one should read Kuznetsov conjecture and Hassett's
question~\cite[Question 16]{hassett-cubicsurvey} as the two most
``rational'' or ``nonrational'' expectations for cubic fourfolds.

Let us end this section by recalling Galkin--Shinder's
construction~\cite{galkin-shinder-cubic}, relying on motivic measures,
which aims to describe a criterion of nonrationality. This
construction would have given indeed a nonrationality criterion under
the hypothesis that the class of the affine line $\mathbb{L}$ in
$K_0(\mathrm{Var}(\CC))$ is not a zero-divisor (see
Chapter~\ref{sect:high-dim-Db} for details on this Grothendieck
group). Unfortunately, after Galkin--Shinder's paper appeared,
Borisov~\cite{borisov-zerodivisor} proved that $\mathbb{L}$ is indeed
a zero-divisor. However, we recall Galkin--Shinder's statement:

\medskip

Assume that the class of the affine line $\mathbb{L}$ is not a zero
divisor in the Grothendieck ring $K_0(\mathrm{Var}(\CC))$. If a cubic fourfold $X$ is rational, then $F(X)$ is
birational to $S^{[2]}$, where $S$ is a K3 surface.

\medskip

Though based on a false assumption\footnote{In
fact, less is required, only that $\mathbb{L}$ does not annihilate any
sum of varieties of dimension at most 2, a condition which is still
unknown.}, the previous statement, together
with Theorem \ref{thm:add-two-cond}, would say that having an
associated K3 is not a sufficient condition to rationality, the first
examples being cubic with discriminant $74$ or $78$
(see~\cite{addington-twoconj}).  As a conclusion, we must admit that
we are probably facing one of the most intriguing problems of
birational geometry: not only proving that the general cubic is not
rational, but also classifying the rational ones seems to need much
more work and finer invariants.

\section{Rationality and $0$-cycles}

One of the fundamental ingredients in the recent breakthrough in the
stable rationality problem was to explicitly tie together the
decomposition of the diagonal and the universal triviality of $\CH_0$.
Such a link was certainly established in the work of Bloch and
Srinivas.  In this section, we want to explain this relationship and
show how it is useful.

\subsection{Diagonals and $0$-cycles}

We begin with the fact that $\CH_0$ is a birational invariant of
smooth proper irreducible varieties, proved by Colliot-Th\'el\`ene and
Coray~\cite[Prop.~6.3]{colliot-thelene_coray} using resolution of
singularities and in general by
Fulton~\cite[Ex.~16.1.11]{fulton:intersection_theory} using the theory
of correspondences.

\begin{lemma}
\label{lem:CH_0-birat}
Let $X$ and $Y$ be smooth proper varieties over a field $k$.  If $X$
and $Y$ are $k$-birationally equivalent then $\CH_0(X) \isom
\CH_0(Y)$.
\end{lemma}
\begin{proof}
Let $f : Y \dashrightarrow X$ be a birational map and $\alpha \in
\CH^n(Y\times X)$ the closure of the graph of $f$, considered as a
correspondence from $Y$ to $X$.  Let $\alpha' \in \CH^n(X\times Y)$ be
the transpose correspondence.  To verify that $\alpha_*$ and
$\alpha'_*$ define inverse bijections, we check that $\alpha' \circ
\alpha$ is the sum of the identity (diagonal) correspondence and other
correspondences whose projections to $Y$ are contained in proper
subvarieties.  By the moving lemma for $0$-cycles, we can move any
element in $\CH_0(Y)$, up to rational equivalence, away from any of
these subvarieties, to where $(\alpha' \circ \alpha)_* = \alpha_*'
\circ \alpha_*$ is the identity map.
\end{proof}

If $X$ is proper over $k$, then there is a well-defined degree map
$\CH_0(X) \to \ZZ$.  We say that $\CH_0(X)$ is \linedef{universally
trivial} if $\deg : \CH_0(X_F) \to \ZZ$ is an isomorphism for every
field extension $F/k$.  This notion was first considered by
Merkurjev~\cite[Thm.~2.11]{merkurjev:unramified_cycle_modules}. Let $N$ be a positive integer.  We say that
$\CH_0(X)$ is \linedef{universally $N$-torsion} if $\deg : \CH_0(X_F)
\to \ZZ$ is surjective and has kernel killed by $N$ for every field
extension $F/k$.

Note that $\deg:\CH_0(\PP_k^n)\simeq\ZZ$ over any field $k$, so that
$\CH_0(\PP^n)$ is universally trivial.  By Lemma~\ref{lem:CH_0-birat},
if a smooth proper variety $X$ is $k$-rational then $\CH_0(X)$ is
universally trivial.  In fact, the same conclusion holds if $X$ is
retract $k$-rational, in particular, stably $k$-rational, which can be
proved using~\cite[Cor.~RC.12]{KMApp}, see also~\cite{CTPirutka}.

To check the triviality of $\CH_0(X_F)$ over every field extension
$F/k$ seems like quite a burden.  However, it suffices to check it
over the function field by the following theorem, proved in
\cite[Lemma~1.3]{ACTP}.

\begin{theorem}
\label{thm:univ_triv}
Let $X$ be a geometrically irreducible smooth proper variety over a
field $k$.  Then the following are equivalent:
\begin{enumerate}
\item The group $\CH_0(X)$ is universally trivial.

\item The variety $X$ has a $0$-cycle of degree 1 and the degree map
$\deg : \CH_0(X_{k(X)}) \to \ZZ$ is an isomorphism.

\item The variety $X$ has an (integral) decomposition of the diagonal.
\end{enumerate}
\end{theorem}
\begin{proof}
If $\CH_0(X)$ is universally trivial then $\CH_0(X_{k(X)})=\ZZ$ and
$X$ has a $0$-cycle of degree 1, by definition.  Let us prove that if
$X$ has a $0$-cycle $P$ of degree 1 and $\CH_0(X_{k(X)})=\ZZ$ then $X$
has a decomposition of the diagonal.  Write $n=\dim(X)$.  Let $\xi \in
X_{k(X)}$ be the $k(X)$-rational point which is the image of the
``diagonal morphism'' $\Spec k(X) \to X \times_k \Spec k(X)$.  By
hypothesis, we have $\xi = P_{k(X)}$ in $\CH_0(X_{k(X)})$.  The
closures of $P_{k(X)}$ and $\xi$ in $X \times X$ are $P \times X$ and
the diagonal $\Delta_X$, respectively.  By the closure in $X\times X$
of a $0$-cycle on $X_{k(X)}$, we mean the sum, taken with
multiplicity, of the closures of each closed point in the support of
the $0$-cycle on $X_{k(X)}$.  Hence the class of $\Delta_X - P \times
X$ is in the kernel of the map $\CH^n(X \times X) \to
\CH^n(X_{k(X)})$.  Since $\CH^n(X_{k(X)})$ is the inductive limit of
$\CH^n(X \times_k U)$ over all nonempty open subvarieties $U$ of $X$,
we have that $\Delta_X - P \times X$ vanishes in some $\CH^n(X \times
U)$.  We thus have a decomposition of the diagonal
$$
\Delta_X = P \times X + Z
$$
in $\CH^n(X \times X)$, where $Z$ is a cycle with support in $X \times
V$ for some closed subvariety $X \smallsetminus U = V \subsetneq X$.

Now we prove that if $X$ has a decomposition of the diagonal, then
$\CH_0(X)$ is universally trivial.  This argument is similar in spirit
to the proof of \ref{prop:bloch} presented
in~\S\ref{subsec:decomp_acting}.  The action of correspondences
(from~\S\ref{subsec:correspondences}) on $0$-cycles has the following
properties: $[\Delta_X]_*$ is the identity map and $[P \times X]_*(z)
= \deg(z) P$ for any $x \in \CH_0(X)$.  By the easy moving lemma for
0-cycles on a smooth variety recalled at the end of
\S\ref{subsec:intersections}, for a closed subvariety $V \subsetneq
X$, every 0-cycle on $X$ is rationally equivalent to one with support
away from $V$.  This implies that $[Z]_* = 0$ for any $n$-cycle with
support on $X \times V$ for a proper closed subvariety $V \subset X$.
Thus a decomposition of the diagonal $\Delta_X = P \times X + Z$ as in
\eqref{eq:decomp_diagonal} implies that the identity map restricted to
the kernel of the degree map $\deg : \CH_0(X) \to \ZZ$ is zero.  For
any field extension $F/k$, we have the base-change $\Delta_{X_F} = P_F
\times X_F + Z_F$ of the decomposition of the diagonal
\eqref{eq:decomp_diagonal}, hence the same argument as above shows
that $\CH_0(X_F)=\ZZ$.  We conclude that $\CH_0(X)$ is universally
trivial.
\end{proof}

This result is useful because often statements about $\CH_0$ are easier
to prove than statements about $\CH_n$.  There is also a version with
universal $N$-torsion.

\begin{theorem}
\label{thm:univ_triv_N}
Let $X$ be a geometrically irreducible smooth proper variety over a
field $k$.  Then the following are equivalent:
\begin{enumerate}
\item The group $\CH_0(X)$ is universally $N$-torsion.

\item The variety $X$ has a $0$-cycle of degree 1 and the degree map
$\deg : \CH_0(X_{k(X)}) \to \ZZ$ has kernel killed by $N$.

\item The variety $X$ has a rational decomposition of the diagonal of
the form $N\Delta_X = N(P \times X) + Z$ for a $0$-cycle $P$ of degree
1 on $X$.
\end{enumerate}
\end{theorem}

Now we mention a result of Merkurjev that helped to inspire the whole
theory.  Recall, from \ref{subsec:unramified}, the definition of the
group of unramified elements $M_\ur(X)$ of a cycle module $M$ and that
$M_\ur(X)$ is trivial means that the natural map $M(k) \to M_\ur(X)$
is an isomorphism.

\begin{theorem}[{Merkurjev~\cite[Thm.~2.11]{merkurjev:unramified_cycle_modules}}]
\label{thm:Merk}
Let $X$ be a smooth proper variety over a field $k$.  Then the
following are equivalent:
\begin{enumerate}
\item $\CH_0(X)$ is universally trivial.
\item $M_\ur(X)$ is universally trivial for any cycle module $M$.
\end{enumerate}
\end{theorem}

There is also an analogous version of Merkurjev's result for universal
$N$-torsion.  The triviality of unramified elements in cycle modules
is quite useful.

\begin{corollary}
If $X$ is a proper smooth retract rational (i.e., stably rational)
variety, then $M_\ur(X)$ is universally trivial for all cycle modules
$M$, e.g., all unramified cohomology is universally trivial.  In
particular, $\Het^1(X,\mu)$ and $\Br(X)$ are universally trivial and,
if $k=\CC$, then the integral Hodge conjecture for codimension 2
cycles holds for $X$.
\end{corollary}  

\begin{example}
In the spirit of Mumford's theorem on 2-forms on surfaces, if $X$ is
an algebraic surface with $p_g(X) > 0$ (more generally, $\rho(X) <
b_2(X)$), then $\CH_0(X)$ is not universally trivial and $X$ does not
have a decomposition of the diagonal.  Here we use the fact, which we
recall from \S\ref{subsec:purity}, that $\Br(X) \isom
(\QQ/\ZZ)^{b_2-\rho} \oplus H$, for $H$ a finite group.
\end{example}

\subsection{Rationally connected varieties}

A smooth projective variety $X$ over a field $k$ is called
\linedef{rationally connected} if for every algebraically closed field
extension $K/k$, any two $K$-points of $X$ can be connected by the
image of a $K$-morphism $\PP^1_K \to X_K$.

For example, smooth geometrically unirational varieties are rationally
connected.  It is a theorem of Campana~\cite{campana:Fano} and
Koll\'ar--Miyaoka--Mori~\cite{kollar_miyaoka_mori:Fano} that any
smooth projective Fano variety over a field of characteristic zero is
rationally connected.  

If $X$ is rationally connected, then $\CH_0(X_K)=\ZZ$ for any
algebraically closed field extension $K/k$.  While a standard argument
then proves that the kernel of $\deg : \CH_0(X_F) \to \ZZ$ is torsion
for every field extension $F/k$, the following more precise result is
known.

\begin{prop}
\label{prop:rat_conn_N}
Let $X$ be a smooth proper connected variety over a
field $k$.  Assume that $X$ is rationally connected, or more generally, that $\CH_0(X_K)=\ZZ$
for all algebraically closed extensions $K/k$.  
\begin{enumerate}
\item ({Bloch--Srinivas~\cite[Prop.~1]{bloch_srinivas}}) Then $X$ has
a rational decomposition of the diagonal.

\item ({Colliot-Th\'el\`ene~\cite[Prop.~11]{colliot:finitude}})  
Then there exists an
integer $N>0$ such that $\CH_0(X)$ is universally $N$-torsion.
\end{enumerate}
Of course, both of these are equivalent by Theorem~\ref{thm:univ_triv_N}.
\end{prop}

In fact, over $\CC$, something more general can be proved.

\begin{lemma}
\label{lem:absolute_gives_N}
Let $X$ be a smooth proper connected variety over an algebraically
closed field $k$ of infinite transcendence degree over its prime field
(e.g., $k=\CC$).  If $\CH_0(X)=\ZZ$ then there exists an integer $N>0$
such that $\CH_0(X)$ is universally $N$-torsion.
\end{lemma}
\begin{proof}
The variety $X$ is defined over an algebraically closed subfield $L
\subset k$, with $L$ algebraic over a field finitely generated over
its prime field.  That is, there exists a variety $X_0$ over $L$ with
$X \isom X_0 \times_L k$.  Let $\eta$ be the generic point of $X_0$.
Let $P$ be an $L$-point of $X_0$.  One may embed the function field
$F=L(X_0)$ into $k$, by the transcendence degree hypothesis on $k$.
Let $K$ be the algebraic closure of $F$ inside $k$.  By
Lemma~\ref{lem:inject} (below) and the hypothesis that $\CH_0(X)=\ZZ$,
we have that $\CH_0(X_0 \times_L F)=\ZZ$.  This implies that there is
a finite extension $E/F$ of fields such that $\eta_E - P_E =0$ in
$\CH_0(X_0 \times_L E)$.  Taking the corestriction (i.e., pushforward)
to $F$, one finds that $N(\eta_F-P_F)=0$ in $\CH_0(X_0\times_L F)$,
hence in $\CH_0(X)$ as well.  As in the proof of
Theorem~\ref{thm:univ_triv}, we conclude that $\CH_0(X)$ is
universally $N$-torsion.
\end{proof}

\begin{lemma}
\label{lem:inject}
Let $X$ be a smooth projective connected variety over $k$.  If $K/k$
is an extension of fields, then the natural map $\CH_0(X) \to
\CH_0(X_K)$ is torsion.  If $k$ is algebraically closed, then $\CH_0(X)
\to \CH_0(X_K)$ is injective.
\end{lemma}
\begin{proof}
Let $z$ be a 0-cycle on $X$ that becomes rationally equivalent to zero
on $X_K$.  Then there exists a subextension $L$ of $K/k$ that is
finitely generated over $k$, such that $z$ becomes rationally
equivalent to zero on $X_L$.  In fact, we can find a finitely
generated $k$-algebra $A$ with fraction field $L$ such that $z$ maps
to zero under $\CH_0(X) \to \CH_0(X \times_k U)$ where $U = \Spec\,
A$.  When $k$ is algebraically closed, there exists a $k$-point of
$U$, defining a section of $\CH_0(X) \to \CH_0(X \times_k U)$, showing
that $z$ is zero in $\CH_0(X)$. In general, we can find a rational
point of $U$ over a finite extension $k'/k$, so that $z_{k'}$ is zero
in $\CH_0(X_{k'})$, from which we conclude that a multiple of $z$ is
zero in $\CH_0(X)$ by taking corestriction.
\end{proof}

There exist rationally connected varieties $X$ over an algebraically
closed field of characteristic zero with $\CH_0(X)$ not universally
trivial.  Indeed, let $X$ be a unirational threefold with
$\Hur^2(X,\QQ/\ZZ(1)) \isom \Br(X) \neq 0$, see e.g.,
\cite{artin_mumford}.  Then by Theorem~\ref{thm:univ_triv_N}, $\CH_0(X)$ is
not universally trivial.

However, such examples do not disprove the natural universal
generalization of the result of Campana~\cite{campana:Fano} and
Koll\'ar--Miyaoka--Mori~\cite{kollar_miyaoka_mori:Fano}, and this was
posed as a question in \cite[\S1]{ACTP}.

\begin{question}
Does there exist a smooth Fano variety $X$ over an algebraically
closed field of characteristic $0$ with $\CH_0(X)$ not universally
trivial?
\end{question} 

After this question was posed, Voisin~\cite{voisin:threefolds}
constructed the first examples of (smooth) Fano varieties over $\CC$
with $\CH_0(X)$ not universally trivial, see
\S\ref{subsec:degeneration} for more details.

\subsection{Surfaces}
\label{subsec:surfaces}

We briefly recall Bloch's conjecture for a complex surface.  Let $X$
be a smooth projective variety.  The Albanese morphism $\alb_X : X \to
\Alb(X)$ is universal for morphisms from $X$ to an abelian variety.
It extends to the Albanese map
$$
\alb_X : A_0(X) \to \Alb(X)
$$
where $A_0(X)$ denotes the kernel of the degree map $\CH_0(X) \to
\ZZ$.  The Albanese map is surjective on geometric points.  In
characteristic zero, $\dim(\Alb(X)) = q(X) = h^1(X,\OO_X)$.  Recall
that $p_g(X) = h^0(X,\Omega^n_X)$ where $n=\dim(X)$.

\begin{conjecture}[Bloch's conjecture]
Let $X$ be a smooth projective surface over $\CC$.  If $p_g(X) =0$
then the Albanese map $\alb_X : A_0(X) \to \Alb(X)$ is injective.  In
particular, if $p_g(X)=q(X)=0$, then $A_0(X)=0$, i.e., $\CH_0(X)=\ZZ$.
\end{conjecture}

In fact, Bloch's conjecture is proved for all surfaces that are not of
general type by Bloch, Kas, and Lieberman~\cite{BKL}.

Of course, rational surfaces satisfy $p_g=q=0$ and have $A_0(X)=0$.
There do exists nonrational surfaces with $p_g=q=0$ and for which
$A_0(X)=0$.  Enriques surfaces were the first examples, extensively
studied in~\cite{enriques},~\cite[p.~294]{enriques:memorie} with some
examples considered earlier in~\cite{reye}, see also
\cite{castelnuovo}.  An Enriques surface has Kodaira dimension 0.  We
remark that for an Enriques surface $X$, we have that
$\Hur^1(X,\ZZ/2\ZZ)=H^1_{\et}(X,\ZZ/2\ZZ)=\ZZ/2\ZZ$ as well as
$\Hur^2(X,\QQ/\ZZ(1))=\Br(X)=\ZZ/2\ZZ$. Hence $\CH_{0}(X)$
is not universally trivial and $X$ does not have a decomposition of
the diagonal by Theorem~\ref{thm:Merk}.

The first surfaces of general type with $p_g=q=0$ were constructed
in~\cite{campedelli} and \cite{godeaux}.  Simply connected surfaces
$X$ of general type for which $p_g=0$ were constructed by
Barlow~\cite{barlow}, who also proved that $\CH_0(X)=\ZZ$ for some of
them.  See also the recent work on Bloch's conjecture by
Voisin~\cite{voisin:bloch}.

We want to explore the universal analogue of Bloch's conjecture, i.e.,
to what extent does $p_g=q=0$ imply universal triviality of $\CH_0(X)$.

\medskip

The following result was stated without detailed proof as the last
remark of~\cite{bloch_srinivas}.  The first proof appeared in
\cite[Prop.~1.19]{ACTP} using results of
\cite{colliot-thelene_raskind:second_Chow_group} and a different proof
appear later in \cite[Cor.~2.2]{voisin:cubics}.

\begin{prop}
\label{surfacenotorsion}
Let $X$ be a smooth proper connected surface over $\CC$.  Suppose that
all groups $\HB^i(X,\ZZ)$ are torsionfree and that $\CH_0(X)=\ZZ$.
Then $\CH_0(X)$ is universally trivial and admits a decomposition of
the diagonal.
\end{prop}
\begin{proof}
By Lemma~\ref{lem:absolute_gives_N}, we have that $\CH_0(X)$ is
universally $N$-torsion.  Hence by Lemma~\ref{prop:bloch}, we have
that $H^i(X,\OO_X)=0$ for all $i \geq 1$.  Hence $p_g(X) = q(X) = 0$,
and thus $b_3(X) = b_1(X) = 2q(X)=0$, so that $\HB^i(X,\ZZ)$ is
consists of classes of algebraic cycles.  

The torsion-free hypothesis on cohomology allows one to use the work
of Colliot-Th\'el\`ene and
Raskind~\cite[Thm.~3.10(d)]{colliot-thelene_raskind:second_Chow_group}
on the cohomology of the Milnor $K$-theory sheaf, to conclude that
$\CH_0(X)$ is universally trivial.

The torsion-free hypothesis on cohomology allows
Voisin~\cite[Cor.~2.2]{voisin:cubics} to argue using the integral
K\"unneth decomposition of the diagonal (see
Remark~\ref{rem:Kunneth}), that $X$ admits a decomposition of the
diagonal.
\end{proof}

We remark that if $X$ is a smooth proper connected surface over $\CC$
with torsionfree N\'eron--Severi group $\mathrm{NS}(X)$, then all Betti
cohomology groups are torsionfree, hence
Proposition~\ref{surfacenotorsion} applies.  Indeed, the torsion in
$\HB^1(X,\ZZ)$ is clearly trivial and is dual to the torsion in
$\HB^3(X,\ZZ)$, while the torsion in $\NS(X)$ is isomorphic to the
torsion in $\HB^2(X,\ZZ)$.

\begin{remark}
\label{rem:Kunneth}
If $\HB^i(X,\ZZ)$ is torsionfree for all $0 \leq i \leq n$, then
there is an integral K\"unneth decomposition
$$
\HB^n(X \times X,\ZZ) = \bigoplus_{i+j=n} \HB^i(X,\ZZ)\tensor\HB^j(X,\ZZ).
$$
This follows from the degeneration of the K\"unneth spectral sequence
with coefficients in $\ZZ$.
\end{remark}

We point out that the surfaces $X$ of general type with $p_g=q=0$ and
$\CH_0(X)=0$ mentioned above, e.g., those construction by Barlow,
satisfy $\Pic(X) ={\rm NS}(X)$ is torsionfree, hence
Proposition~\ref{surfacenotorsion} applies.  While the group
$\CH_{0}(X)$ is universally trivial, these surfaces are far from being
rational, since they are of general type.

The interested reader can find how to adapt
Proposition~\ref{surfacenotorsion} over an algebraically closed field
of infinite transcendence degree over its prime field.

Finally, we mention that Proposition~\ref{surfacenotorsion} has been
generalized by Kahn~\cite{kahn_colliot:torsion_order_surface}, and
independently by Colliot-Th\'el\`ene using
\cite{colliot-thelene_raskind:second_Chow_group}, to a determination
of the minimal $N$ for which $\CH_0(X)$ is universally $N$-torsion,
which turns out to be the exponent of $\NS(X)$.  In general, the
minimal $N \geq 1$ for which $\CH_0(X)$ universally $N$-torsion is a
stable birational invariant of smooth proper varieties; its properties
are explored in \cite{chatz-levine}, where it is called the
\linedef{torsion order} of $X$.

\section{Categorical representability and rationality, the case of surfaces}\label{sect:Db-surfaces}

This section consists of two main parts. In the first part, we define
the notion of categorical representability and begin to classify (or
at least, give criteria to discriminate) categories which are
representable in low dimension. The second part is devoted to the
applications in the case of surfaces.

\subsection{Categorical representability}
\label{subsec:cat_rep}

Using semiorthogonal decompositions, one can define a notion of
\linedef{categorical representability} for a triangulated category. In
the case of smooth projective varieties, this is inspired by the
classical notions of representability of cycles, see
\cite{bolognesi_bernardara:representability}.

\begin{definition}
\label{def-rep-for-cat}
A $k$-linear triangulated category $\cat{T}$ is \linedef{representable
in dimension $m$} if it admits a semiorthogonal decomposition
$$
\cat{T} = \langle \cat{A}_1, \ldots, \cat{A}_r \rangle,
$$
and for each $i=1,\ldots,r$ there exists a smooth projective connected
$k$-variety $Y_i$ with $\dim Y_i \leq m$, such that $\cat{A}_i$ is
equivalent to an admissible subcategory of $\Db(Y_i)$.

We use the following notation
$$\Rep\cat{T} := \min \{ m \in \NN \, \vert \, \cat{T} \text{ is representable in dimension } m\},$$
whenever such a finite $m$ exists.
\end{definition}

\begin{definition}
\label{def-cat-rep}
Let $X$ be a smooth projective $k$-variety. We say that $X$
is \linedef{categorically representable} in dimension $m$ (or
equivalently in codimension $\dim(X)-m$) if $\Db(X)$ is representable in
dimension $m$.

We will use the following notations:
$$
\Repcat(X) := \Rep\Db(X) \,\,\,\,\, \coRepcat(X):= \dim(X)-\Rep\Db(X),
$$
and notice that they are both integer numbers.
\end{definition}

We notice that, by definition, if $\Rep\cat{T}=n$, then $\cat{T}$ is representable in any dimension $m \geq n$.
\begin{lemma}\label{lem:minimal-represntability}
Let $\cat{T}$ be a $k$-linear triangulated category. If $\cat{T}$ is representable in
dimension $n$, then $\cat{T}$ is representable in dimension $m$ for any $m \geq n$.
\end{lemma}

\begin{remark}\label{rmk:warning-doubledeco}
{\bf Warning!} Suppose that $\cat{T}$ is representable in dimension $n$ via a
semiorthogonal decomposition $\cat{T} = \sod{\cat{A}_1,\ldots,\cat{A}_r}$,
and let $\cat{T}= \sod{\cat{B}_1 \ldots,\cat{B}_s}$ be another semiorthogonal
decomposition (that is, $\cat{B}_i$ is not admissible $\cat{A}_j$ and $\cat{A}_j$ is not
admissible in $\cat{B}_i$ for any
$i$ and $j$).
As recalled in Proposition~\ref{prop:no-JH} the Jordan--H\"older property for semiorthogonal
decompositions does not hold in general. It follows that one does not know in general
whether the $\cat{B}_i$ are also representable in dimension $n$, and counterexamples are known:
Bondal-Kuznetsov's counterexample~\cite{kuznet:JH} is given by a threefold $X$ with a full
exceptional sequence $\sod{E_1,\ldots,E_6}$, and another exceptional object $F$ whose complement cannot be generated
by exceptional objects.
\end{remark}

Let us record a simple corollary of Theorem~\ref{thm:blow-ups}.

\begin{lemma}
\label{lem:blow-up_catrep2}
Let $X \to Y$ be the blow-up of a smooth projective $k$-variety along
a smooth center. Then $\coRepcat(X) \geq \mathrm{max}\{\coRepcat(Y),2\}$.
In particular, if $\coRepcat(Y) \geq 2$, then $\coRepcat(X) \geq 2$.
\end{lemma}
\begin{proof}
This is a consequence of Theorem~\ref{thm:blow-ups}.
Denoting by $Z \subset Y$ the center of the blow-up 
we have $\Repcat(X) \leq \max\{\Repcat(Y),\Repcat(Z)\}$ and the
statement follows since $\dim(X)=\dim(Y)$ and by the fact that $Z$ has codimension
at least 2 in $Y$.
\end{proof}

Inspired by the proof of Theorem~\ref{thm:J-is-bir-inv}, one can consider
a birational map $X \dashrightarrow X'$ and its resolutions: by Hironaka's resolution of singularities,
there is a smooth projective
$X_1$ with birational morphisms $\rho_1: X_1 \to X'$ and $\pi_1: X_1 \to X$, such that $\pi_1$ is a composition of
a finite number of smooth blow-ups.
Similarly, there are $\rho_2: X_2 \to X$ and $\pi_2: X_2 \to X'$ birational morphisms with
$\pi_2$ a composition of a finite number of blow-ups.
By Lemma~\ref{lem:birat-map-vs-ffemb} we have that
$\Db(X')$ is admissible in $\Db(X_1)$, and $\Db(X)$ is admissible
in $\Db(X_2)$. Lemma~\ref{lem:blow-up_catrep2} gives bounds
for $\coRepcat(X_1)$ and $\coRepcat(X_2)$ in terms of $\coRepcat(X)$ and
$\coRepcat(X')$ respectively.

Based on these considerations, Kuznetsov~\cite{kuz:rationality-report} argues that if one could properly define 
an admissible subcategory $\cat{GK}_X$ of $\Db(X)$, maximal (with respect to the inclusion
ordering) with respect to the property $\Rep \cat{GK}_X \geq \dim(X)-1$, then such
a category would be a birational invariant, which we would call \linedef{the Griffiths--Kuznetsov component}
of $X$. In particular, since $\Repcat(\PP^n)=0$, we would have that the Griffiths--Kuznetsov
component of a rational variety is trivial.

Even if the Griffiths--Kuznetsov component is not well-defined, we
have that if $X$ is rational, then $\Db(X)$ is admissible in a
category with $\Rep \leq \dim(X)-2$.  As we recalled in
Remark~\ref{rmk:warning-doubledeco}, there is no known reason to
deduce that $\coRepcat(X) \geq 2$.  However, in the small dimensional
cases, we have a stronger understanding of these phenomena.  We will
come back to this question, giving more detailed arguments for
threefolds and examples for fourfolds, in \S\ref{sect:high-dim-Db}.

\subsection*{Representability in dimension 0}

\begin{prop}
\label{lem:0-dim=etale-algebra}
Let $\cat{T}$ be a $k$-linear triangulated category.
$\Rep\cat{T}=0$ if and only if there exists a semiorthogonal
decomposition
$$
\cat{T} = \langle \cat{A}_1, \ldots, \cat{A}_r \rangle,
$$
such that for each $i$, there is a $k$-linear equivalence $\cat{A}_i
\simeq \Db(K_i/k)$ for an \'etale $k$-algebra $K_i$.
\end{prop}

An additive category $\cat{T}$ is \linedef{indecomposable} if for any
product decomposition $\cat{T} \equi \cat{T}_1 \times \cat{T}_2$ into
additive categories, we have that $\cat{T} \equi \cat{T}_1$ or
$\cat{T} \equi \cat{T}_2$.  Equivalently, $\cat{T}$ has no nontrivial
completely orthogonal decomposition.  Remark that if $X$ is a
$k$-scheme then $\Db(X)$ is indecomposable if and only if $X$ is
connected (see \cite[Ex.~3.2]{bridg-equiv-and-FM}).
More is known if $X$ is the spectrum
of a field or a product of fields~\cite{auel-berna-surf}.

\begin{lemma}
\label{lem:indecomposable}
Let $K$ be a $k$-algebra.
\begin{enumerate}
\item\label{lem:indecomposable_1} If $K$ is a field and $\cat{A}$ is a
nonzero admissible $k$-linear triangulated subcategory of $\Db(k,K)$,
then $\cat{A} = \Db(k,K)$.

\item\label{lem:indecomposable_2} If $K \isom K_1 \times \dotsm \times
K_n$ is a product of field extensions of $k$ and $\cat{A}$ is a
nonzero admissible indecomposable $k$-linear triangulated subcategory
of $\Db(k,K)$, then $\cat{A} \equi \Db(k,K_i)$ for some $i=1, \dotsc, n$.

\item\label{lem:indecomposable_3} If $K \isom K_1 \times \dotsm \times
K_n$ is a product of field extensions of $k$ and $\cat{A}$ is a
nonzero admissible $k$-linear triangulated subcategory
of $\Db(k,K)$, then $\cat{A} \equi \prod_{j \in I} \Db(k,K_{j})$ for
some subset $I \subset \{1, \dotsc, n\}$.
\end{enumerate}
\end{lemma}

\begin{proof}[Proof of Proposition~\ref{lem:0-dim=etale-algebra}, (see also~\cite{auel-berna-surf})]
The smooth $k$-varieties of dimension 0 are precisely
the spectra of \'etale $k$-algebras.  Hence the semiorthogonal
decomposition condition is certainly sufficient to get $\Rep\cat{T}=0$.  On the other hand,
if $\Rep\cat{T}=0$ , we have
such a semiorthogonal decomposition with each $\cat{A}_i$ an admissible subcategory of the derived category of
an \'etale $k$-algebra.  By
Lemma~\ref{lem:indecomposable}\ref{lem:indecomposable_3}, we have that
$\cat{A}_i$ is thus itself such a category.
\end{proof}
We have the following corollary of Proposition~\ref{lem:0-dim=etale-algebra}.
\begin{lemma}\label{lem:K0-of-zerodimensional}
Let $\cat{T}$ be a $k$-linear triangulated category.
If $\Rep\cat{T}=0$, then $K_0(\cat{T})$ is a free $\ZZ$-module of finite rank.
In particular, if $X$ is smooth and projective and $\Repcat(X)=0$, we have that
$\CH^1(X)$ is torsion-free of finite rank.
\end{lemma}
\begin{proof}
The only non-trivial statement is the last one, which is shown in~\cite[Lemma 2.2]{galkin-katzarkov-mellit-shinder}
using the topological filtration on $K_0(X)$.
\end{proof}

\begin{remark}
Lemma~\ref{lem:K0-of-zerodimensional} gives useful criterion:
if $K_0(\cat{T})$ has torsion elements or if it is not of finite rank, then $\Rep\cat{T} >0$.

In the cases where $\cat{T}=\Db(X)$
for a smooth projective $X$, if $\Repcat(X)=0$, there are much more consequences that can be obtained
using non-commutative motives, in particular in the case where $k \subset \CC$ is
algebraically closed: for example, the even deRham cohomology and all the Jacobians are trivial~\cite{marcolli-tabuada-exc,berna-tabuada-jacobians},
the (rational) Chow motive is of Lefschetz type~\cite{marcolli-tabuada-exc}.
\end{remark}

\subsection*{Representability in dimension 1}

\begin{prop}\label{prop:cat-rep-dim1}
Let $\cat{T}$ be a $k$-linear triangulated category. $\Rep\cat{T} \leq 1$ if and only if $\cat{T}$ admits
a semiorthogonal decomposition whose components belong to the following list:

\begin{enumerate}
 \item categories representable in dimension 0, or
 \item categories of the form $\Db(k,\alpha)$, for $\alpha$ in $\Br(k)$ the Brauer class of a conic, or
 \item categories equivalent to $\Db(C)$ for some smooth $k$-curve $C$.
\end{enumerate}

\end{prop}

The main tool in the proof of the previous statement is the indecomposability result for curves
due to Okawa that we recalled in Theorem~\ref{thm:okawacurves}.

\begin{proof}[Proof of Proposition~\ref{prop:cat-rep-dim1}]
We already classified those $\cat{A}$ such that $\Rep\cat{A}=0$. 
We are hence looking for categories $\cat{A}$ with $\Rep\cat{A}=1$.
Of course, if $\cat{A}=\Db(C)$ for some curve, then $\Rep\cat{A}\leq 1$ and
we are done.

Using Theorem~\ref{thm:okawacurves}, if $\cat{A}$ is a nontrivial triangulated category with a full
and faithful functor $\phi:\cat{A}\to \Db(C)$ with right and left adjoints, then either
$\phi$ is an equivalence, or $g(C)=0$.
In the latter case, let $\alpha$ be the
class of $C$ in $\Br(k)$, which is trivial if and only if $C=\PP^1$. By~\cite{bernardara:brauer_severi}, there is a semiorthogonal decomposition
$\Db(C) = \sod{\Db(k),\Db(k,\alpha)}$. It is not difficult to see that this is the only possible semiorthogonal decomposition up to mutations, 
so we get the proof.
\end{proof}

Let us sketch a criterion of representability in dimension $1$, based
on the Grothendieck group, in the case where $k=\CC$.  Given a
$\ZZ$-module $M$ and an integer number $n>0$, we will denote by $M_n
\subset M$ the kernel of the multiplication by $n$ map: $M_n :=
\mathrm{ker}(M \stackrel{\times n}{\longrightarrow} M)$.  Such $M_n$
has a natural structure of $\ZZ/n\ZZ$-module.  We notice that, if $X$
is a smooth (connected) projective variety of dimension $\leq 1$, the
modules $K_0(X)_n$ are well-known.  Indeed, either $X$ is a point, or
$X$ is $\PP^1$, or $X$ is a curve of positive genus $g$. In the first
two cases, $K_0(X)$ is free of finite rank, hence $K_0(X)_n=0$ for any
$n$. In the latter case, $K_0(X) \simeq \ZZ \oplus \Pic(X) = \ZZ
\oplus \Pic^0(X) \oplus \ZZ$ by the Grothendieck--Riemann--Roch
Theorem, and the fact that the 1st Chern class is integral. Then
$K_0(X)_n= \Pic^0(X)_n$. Since $\Pic^0(X)$ is a complex torus of
dimension $g$, we have that $\Pic^0(X)_n = (\ZZ/n\ZZ)^{2g}$ (see,
e.g.,~\cite[\S I.1, (3)]{mumford-abelian}).

\begin{lemma}\label{lem:K0-of-1dimensional}
Suppose that $\cat{T}$ is $\CC$-linear and $\Rep\cat{T} \leq 1$. Then, for any integer $n$, $K_0(\cat{T})_n$ is 
a free $\ZZ/n\ZZ$-module of finite even rank.
\end{lemma}
\begin{proof}
Proposition~\ref{prop:cat-rep-dim1} gives us all the possible components of a semiorthogonal decomposition
of $\cat{T}$. If $\cat{A}$ is one of such components, it follows that $K_0(\cat{A})_n$ is either trivial
or $(\ZZ/n\ZZ)^{2g}$ if $\cat{A}=\Db(C)$ and $g=g(C)$. 
\end{proof}

Finally, let us just record a very simple remark, as a corollary of Theorem~\ref{thm:okawacurves}.

\begin{corollary}
A smooth projective curve $C$ is $k$-rational if and only if $\Repcat(C)=0$.
\end{corollary}

\subsection*{A glimpse on representability in dimension 2}
It is more difficult to classify categories which are representable in
dimension 2.  Of course, for any surface $S$, $\cat{T}=\Db(S)$
satisfies $\Rep\cat{T} \leq 2$, but it is a quite challenging question
to understand which categories can occur as proper admissible
subcategories of surfaces.  Using Kawatani--Okawa's results recalled
in Theorem~\ref{thm:okawasurfaces}, we can fairly suppose that we
should consider surfaces $S$ either ruled or with $p_g=q=0$, at least
in the case where $k=\CC$. In the ruled case, say $S \to C$, we have
$\Repcat(S)=\Repcat(C) \leq 1$ with strict inequality holding only for
$C=\PP^1$. Hence we will focus on surfaces with $p_g=q=0$.

Notice that, for a surface $S$, any line bundle is $k$-exceptional if
and only if $S$ satisfies $p_g=q=0$. This is a simple calculation
using that line bundles are invertible and the definition of $p_g$ and
$q$.  It is natural to study then exceptional collections on such
surfaces and describe the corresponding semiorthogonal
decompositions. These decompositions are conjecturally related to
rationality criteria for surfaces, and we will treat them extensively
in \S\ref{subs:cat-rep-vr-rat-surfaces}. We just notice that an
exceptional collection on a variety $X$ gives a free subgroup of
finite rank of $K_0(X)$. Surfaces of general type with $p_g=q=0$ and
torsion elements in $K_0(S)$ are known, and hence they cannot have a
full exceptional sequence. This remark gives rise to the definition of
phantom and quasi-phantom categories.

\begin{definition}
Let $X$ be a smooth projective variety. An admissible subcategory $\cat{T} \subset \Db(X)$
is called a \linedef{quasi-phantom} if its Hochschild homology $HH_*(\cat{T})=0$ vanishes
and $K_0(\cat{T})$ is a finite abelian group. A quasi-phantom $\cat{T}$ is a \linedef{phantom} if
$K_0(\cat{T})=0$. A phantom is a \linedef{universal phantom} if, for any smooth and projective
variety $Y$, the admissible subcategory $\cat{T} \boxtimes \Db(Y)$ of $\Db(X \times Y)$ is 
a phantom.
\end{definition}

Notice that if $\cat{T}$ is a quasi-phantom (or a phantom, or a universal phantom), then $\Rep\cat{T}>1$. 

On the other hand, if $k$ is general, then more complicated phenomena can arise, already by considering
descent of full exceptional sequences from $\overline{k}$ to $k$.

\begin{example}[Categories representable in dimension 2]\label{ex:catrep2}
Here is a list of categories $\cat{T}$ such that $\Rep\cat{T}=2$ and $\cat{T}$ is not equivalent to $\Db(S)$ for any smooth and projective $S$.

\smallskip

\noindent{\bf Phantoms, $k=\CC$.} If $S$ is a determinantal Barlow surface, $\cat{T}$ is the orthogonal
complement to an exceptional collection of length 11~\cite{BGKS}. If $S$ is a Dolgachev surface
of type $X_9(2,3)$, then $\cat{T}$ is the orthogonal complement to an exceptional collection
of length 12~\cite{cho-lee} (we refer to~\cite{dolgachev-p=q=0}
for the notations on Dolgachev surfaces).

\noindent{\bf Quasi-phantoms, $k=\CC$.} 
Since the first example of the classical Godeaux surface~\cite{boeh-graf-sos-Godeaux}, 
there are now many examples of quasi-phantoms as orthogonal complements of an exceptional
sequence of line bundles of maximal length on surfaces of general type, see~\cite{galkin.shinder-beauville},
\cite{galkin-katzarkov-mellit-shinder}, \cite{keum}, \cite{lee-surfaces-isogenous}, \cite{alexeev-orlov},
\cite{coughlan}.

\smallskip

\noindent{\bf Not quasi-phantoms, $k=\CC$.} If $S$ is an Enriques surface~\cite{ingalls-kuznetsov} and $\cat{T}$
is the orthogonal complement to an exceptional collection of length $10$, or if $S$ is a classical
Godeaux surface~\cite{boeh-graf-sos-JH} and $\cat{T}$ is the orthogonal complement to an exceptional
collection of length $9$. Both categories have $K_0(\cat{T})\otimes \QQ \simeq \QQ^2$, but $K_0(\cat{T})$ not
free, and do not admit any exceptional object~\cite{boeh-graf-sos-JH,vial-exceptional}.

\smallskip

\noindent{\bf Not quasi-phantoms, general $k$.} If $\alpha$ in $\Br(k)$ is the class of a Brauer-Severi surface,
or the class of a Brauer-Severi threefold with an involution surface, then $\cat{T}=\Db(k,\alpha)$.
If $\kc_0$ is the Clifford algebra of an involution surface, then $\cat{T}=\Db(k,\kc_0)$. If
$Q$ and $B$ are the simple algebras (quadratic over a degree 3 extension of $k$ and cubic over a degree
2 extension of $k$ respectively) related to a minimal degree 6 del Pezzo surface, then $\cat{T} =\Db(k,Q)$
and $\cat{T}=\Db(k,B)$. For all these examples, $K_0(\cat{T})$ is free of finite rank, see~\cite{auel-berna-surf}.
If $S$ is a minimal del Pezzo surface of degree $d < 5$, then $\cat{T}$ is the semiorthogonal complement
of $\sod{\ko_S}$, and is not of the form $\Db(K,\alpha)$ for $K/k$ \'etale and $\alpha$ in $\Br(K)$, see~\cite{auel-berna-surf}.

\end{example}

The result of Kawatani--Okawa, cf.\ Theorem~\ref{thm:okawasurfaces},
suggests that, for $k=\CC$, categories representable in dimension
exactly $2$, and not equivalent to any $\Db(S)$, should occurr only in
the case where $p_g=q=0$. In this case, Bloch conjecture would imply
that $K_0(S)\otimes\QQ$ is a finite vector space, so it is natural to
raise the following Conjecture.

\begin{conjecture}
Suppose that $\Rep\cat{T}=2$, and that $\cat{T}$ is not
equivalent to $\Db(S)$ for any surface $S$. Then $K_0(\cat{T})_\QQ$ is a finite-dimensional vector space.
\end{conjecture}

\subsection{Rationality questions for surfaces}\label{subs:cat-rep-vr-rat-surfaces}

We turn our attention to the possibility of characterizing rational surfaces via
categorical representability. A folklore conjecture by D.Orlov states that a complex surface
with a full exceptional collection is rational. We provide here a version for any field $k$
in terms of categorical representability.

\begin{conjecture}[Orlov]\label{conj:orlov-surfaces}
A surface $S$ is $k$-rational if and only if $\Repcat(S)=0$.
\end{conjecture}

If $k$ is algebraically closed, then being representable in dimension zero is equivalent
to having a full exceptional collection. Combining
Theorem~\ref{thm:blow-ups} and Proposition~\ref{prop:deco-projective} it is easy to check that
a rational surface has a full exceptional collection, since it is a blow-up
along smooth points of either a projective space or a Hirzebruch surface.

If $k$ is not algebraically closed, it is easy to construct rational
surfaces without a full exceptional collection, for example a
$k$-rational quadric surface of Picard rank 1, or by blowing up a
closed point of degree $>1$ on a $k$-rational surface.  In fact, the
``only if'' part of Conjecture~\ref{conj:orlov-surfaces} remains true,
thanks to the results of~\cite{auel-berna-surf} for del Pezzo
surfaces. 

\begin{theorem}
Let $S$ be any smooth $k$-rational surface. Then $\Repcat(S)=0$.
\end{theorem}

The converse is more difficult. Let us first recall the following
result based on a base change formula by Orlov~\cite{orlovequivabel}.

\begin{lemma}[\cite{auel-berna-bolo}, Lemma 2.9]\label{lem:orlov-base-change}
Let $X$ be a smooth projective variety over $k$, and $K$ a finite
extension of $k$.  Suppose that $\cat{A}_1, \ldots, \cat{A}_n$ are
admissible subcategories of $\Db(X)$ such that $\Db(X_K) = \langle
\cat{A}_{1 K}, \ldots, \cat{A}_{n K} \rangle$. Then $\Db(X) = \langle
\cat{A}_1, \ldots, \cat{A}_n \rangle$. 
\end{lemma}

Using Lemma~\ref{lem:orlov-base-change} and the classification from
Proposition~\ref{lem:0-dim=etale-algebra}, we deduce that, in the case
$k \subset \CC$, it is enough to check
Conjecture~\ref{conj:orlov-surfaces} for geometrically rational
surfaces and for complex surfaces with $p_g=q=0$.  As remarked
earlier, these are the only surfaces where a line bundle is
$k$-exceptional.

\subsection*{Geometrically rational surfaces}
The first case is handled in~\cite{auel-berna-surf} for del Pezzo surfaces and in
\cite{vial-exceptional} for geometrically rational surfaces with a numerically
$k$-exceptional collection of maximal length.

\begin{theorem}\label{thm:from-repre-to-ratio-for-georatsurf}
Let $S$ be a geometrically rational surface. If $S$ is
\begin{itemize}
 \item either is a blow-up of a del Pezzo and $\Repcat(S)=0$, 
 \item or has a (numerically) $k$-exceptional collection of maximal length,
\end{itemize}
 then $S$ is $k$-rational. 
\end{theorem}

The results of~\cite{auel-berna-surf} also
provide a categorical birational invariant.

\begin{theorem}[\cite{auel-berna-surf}]
Let $S$ be a minimal del Pezzo surface of degree $d$.
If $d < 5$, then $\cat{A}_S$ is a birational invariant.
If $d \geq 5$, then the product of components $\cat{T}$
of $\cat{A}_S$ with $\Rep\cat{T}>0$ is a birational invariant.
In particular, there is a well-defined Griffiths--Kuznetsov component.
\end{theorem}

Let us conclude by showing that having a full $k$-exceptional
collection is a stronger property than having a decomposition of the
diagonal. The following result is a slight generalization of a result
of Vial~\cite{vial-exceptional}. As we will see, nonrational surfaces
satisfying the assumptions of Theorem~\ref{thm:vial-phantom} exist,
and are known to have a decomposition of a diagonal, as recalled in
Theorem~\ref{surfacenotorsion}. Here we give a direct proof, adapted
from Vial's one, providing an explicit decomposition of the diagonal
from the exceptional objects.

\begin{theorem}\label{thm:vial-phantom}
Let $S$ be a surface with $\chi(\ko_S)=1$ and a semiorthogonal decomposition
\begin{equation}\label{eq:decophantomvial}
\Db(S)= \sod{\cat{A}, E_1,\ldots, E_r},
\end{equation}
where $E_i$ are $k$-exceptional and $\cat{A}$ is a phantom. Then $S$
has a decomposition of the diagonal and the integral Chow motive of
$X$ is of Lefschetz type.
\end{theorem}
\begin{proof}
If $\cat{A}=0$, the second statement is a result of
Vial~\cite[Thm.~2.7]{vial-exceptional}, while the first statement is
shown in the course of Vial's proof.  We will detail the main
steps of the proof to show that having a nontrivial phantom (which Vial
does not address) does not affect the proof.

\smallskip

{\bf 1st Step.} First of all, the semiorthogonal decomposition implies
that $K_0(S)$ is free of finite rank.  Using the topological
filtration on $K_0$, we can show that the integral Chow ring
$\CH^*(S)$ is then also free of finite rank: $\CH^0(S)=\ZZ$,
$\CH^1(S)=\Pic(S)$ is free of finite rank by
Lemma~\ref{lem:K0-of-zerodimensional}, and $\CH^2(S)$ is free of
finite rank since it coincides with the second graded piece of the
topological filtration, which is a subgroup of $K_0(S)$, as $S$ has
dimension 2.  In particular, the group of $0$-cycles $A_0(S)$ of
degree 0, is free of finite rank.  However, $A_0(S)$ is always
divisible, cf.\ \cite[Lec.~1,~Lemma~1.3]{bloch:lectures}.  We conclude
that $A_0(S)=0$.  In particular, $S$ satisfies Bloch's conjecture.  In
this case, Sosna~\cite[Cor.~4.8]{sosna:phantom} (using results of
\cite{gorch-orlov}) remarked that any phantom category in $\Db(S)$ is
a universal phantom, so that in any base change $S_K$ of $S$, the
admissible subcategory $\cat{A}_K$ is a phantom in the base change of
the decomposition \eqref{eq:decophantomvial}.  This implies that
$K_0(S_K)$ is torsion free of finite rank and then by the same
argument, that $\CH^*(S_K)$ is free of finite rank.  In particular, we
have that $r=\rk(K_0(S_K)) = \rho+2$, where $\rho$ is the Picard rank
of $S_K$.

{\bf 2nd Step.} 
Given an exceptional collection of maximal length,
Vial~\cite[Prop.~2.3]{vial-exceptional} provides a $\ZZ$-basis $D_1,
\ldots, D_{\rho}$ of $\CH^1(S_K)=\Pic(S_K)$ with unimodular
intersection matrix $M$ and dual basis $D_1^\vee , \ldots,
D_\rho^\vee$.  This is accomplished using Chern classes and the
Riemann--Roch formula to compare $\chi$ with the intersection pairing.

{\bf 3rd Step.} As shown by Vial~\cite[Cor.~2.5]{vial-exceptional}, a
surface with $\chi(\ko_S)=1$ and a numerically $k$-exceptional
collection of maximal length has a zero-cycle of degree 1. Hence $S_K$
has always a zero-cycle $a$ of degree 1, then we set $\pi^0:=a \times
S_K$ and $\pi^4 := S_K \times a$ as idempotent correspondences in
$\CH^2(S_K \times S_K)$. Moreover, Vial defines the correspondences
$p_i:=D_i \times D_i^\vee$ in $\CH^2(S_K \times S_K)$, which are
idempotent since the intersection product is unimodular. It is not
difficult to see that all the above correspondences are mutually
orthogonal.  Set $\Gamma_K:= \Delta_{S_K} - \pi^0 - \pi^4 -
\sum_{i=1}^\rho p_i$.

{\bf 4th Step.} Since $\pi^i$ and $p_i$ are mutually orthogonal
idempotents, $\Gamma_K$ is idempotent.  Moreover, $\Gamma_K$ acts
trivially on $\CH^*(S_K)$, since we have $\CH^2(S_K) = \ZZ a$ and
$\CH^1(S_K)$ generated (over $\ZZ$) by the
$D_i$'s. By~\cite[Prop. 3.7]{shen-vial}, choosing $K$ to be an
algebraically closed extension (universal domain) of $k$, it follows
that $\Gamma$ is nilpotent. Since $\Gamma$ is also idempotent, we have
$\Gamma=0$ and the claim holds.
\end{proof}

\subsection*{Complex surfaces with $p_g=q=0$}
In this paragraph, $S$ has $p_g=q=0$ and $k=\CC$.
The study of such surfaces is very rich and already very challenging
in the case where $k = \CC$. On one hand, it is easy to see that if $S$ has a full
exceptional collection, then $K_0(S)$ is a free $\ZZ$-module of finite rank.
This automatically exclude surfaces with torsion line bundles, which would give rise
to a torsion class in $K_0(S)$. However, a full understanding of the r\^ole of (quasi)-phantoms,
categories representable in dimension 2, and Conjecture~\ref{conj:orlov-surfaces}, requires a full
understanding of any such surface, independently on the obstruction mentioned above. 

First of all, Vial classifies all such $S$ which have a numerically exceptional sequence of
maximal length. This is based on the study of the Picard lattice which can be deduced by
the numerically exceptional sequence via Riemann--Roch theorem (we refer to~\cite{dolgachev-p=q=0}
for the notations on Dolgachev surfaces).

\begin{theorem}[Vial~\cite{vial-exceptional}]\label{vial:num-exc-surf}
Let $S$ be as above. Then $S$ has a numerically exceptional collection of maximal length
if and only if it has a numerically exceptional collection of maximal length consisting of line bundles. Moreover,
this is the cases if and only if either:
\begin{itemize}
 \item $S$ is not minimal, or
 \item $S$ is rational, or
 \item $S$ is a Dolgachev surface of type $X_9(2,3)$, $X_9(2,4)$, $X_9(3,3)$ or $X_9(2,2,2)$, or
 \item $\kappa(S)=2$.
 \end{itemize}
\end{theorem}
\begin{remark}
There are cases where such a numerically exceptional collection is actually an exceptional
collection, namely rational surfaces~\cite{hille-perling} (in which case it is full), Dolgachev
surfaces of type $X_9(2,3)$~\cite{cho-lee} and many examples of
surfaces of general type, see~\cite{galkin.shinder-beauville}, \cite{BGKS}
\cite{galkin-katzarkov-mellit-shinder}, \cite{keum}, \cite{lee-surfaces-isogenous}, \cite{alexeev-orlov},
\cite{coughlan}.
\end{remark}

The classification of complex surfaces of general type with these invariants
is quite wild, and probably still incomplete, see~\cite{bauer-cata-pigna-survey} for a recent survey.
Proceeding examplewise, one considers many interesting objects to
study, but it is not a realistic way to
attack Conjecture~\ref{conj:orlov-surfaces}.
On the other hand, such surfaces often come in positive-dimensional families of isomorphism
classes, and have ample anticanonical bundle, so that $\Db(S)$ identifies the
isomorphism class of $S$ by the famous theorem of Bondal and Orlov~\cite{bondorlovreconstruct}.
We hence have a positive dimensional family of equivalence classes of (dg enhanced) triangulated categories,
while categories generated by  exceptional collections depend on a finite number of countable parameters.

If a dg enhanced triangulated category 
$\cat{T}$ is not trivial, then its Hochschild cohomology $HH^*(\cat{T})$ is also nontrivial\footnote{
Denote by $\kt$ a dg enhancement of $\cat{T}$. The Hochschild cohomology is naturally interpreted as the homology of the complex
$\cal{H}om({\bf 1}_{\kt},{\bf 1}_{\kt})$ computed in the dg category
$\cal{RH}om(\kt,\kt)$ where ${\bf 1}_\kt$ denotes the identity functor of
$\kt$, see~\cite[\S 5.4]{keller:ICM}. It follows that whenever $\kt$ is nontrivial, the class of the identity
is a nontrivial element of $HH^*(\kt)$.}. Moreover, the second Hochschild cohomology encodes
deformations of the dg enhanced category $\cat{T}$ (see~\cite[\S 5.4]{keller:ICM} for a survey).

Suppose then that we have a family of surfaces $S_t$ of general type depending on a continuous
parameter $t$, and that we can produce for any of these surfaces
an exceptional collection $\{ L^t_1, \ldots, L^t_n \}$ of maximal length consisting of
line bundles, and set $\cat{A}_t$ as its semiorthogonal complement.
Then only a countable number of such collections can be full, that is $\cat{A}_t=0$ only
for a discrete set of parameters. It would be natural to expect that the exceptional collection does not
vary, while the informations on the deformation of $S_t$ (and, henceforth, of
$\Db(S_t)$) along with $t$ have to be parameterized by $\cat{A}_t$.
This reasoning is supported by two examples, namely a family of Barlow surfaces~\cite{BGKS}
and a family of Dolgachev surfaces of type $X_9(2,3)$~\cite{cho-lee} and justifies
the following conjectural question.

\begin{question}\label{quest:constant-exceptional}
Let $S_t$ be a family of minimal non-rational surfaces with $p_g=q=0$
depending on a continuous parameter $t$ admitting a numerically
exceptional sequence $\cat{E}_t=\{ E_1^t, \ldots, E_n^t \}$ of maximal length for any $t$. Suppose that
$\cat{E}_t$ is exceptional, then:
\begin{itemize}
 \item[1)] Is $\sod{\cat{E}_t}$ constant? That is, is $\sod{\cat{E}_t}$ equivalent to
 $\sod{\cat{E}_{t'}}$ for any $t$ and $t'$?
 \item[2)] Is $\cat{A}_t = \sod{\cat{E}_t}^\perp$ nontrivial for all $t$?
\end{itemize}
\end{question}
Notice that a positive answer to 1) in Question~\ref{quest:constant-exceptional}
would imply that $\cat{A}_t$ is nontrivial for all but possibly one value of
$t$, and that $\cat{A}_t $ would parameterize the deformations of $\Db(S_t)$
which contain the deformations of $S_t$.

\section{$0$-cycles on cubics}

In \S\ref{subsec:surfaces}, we ended with an essentially complete
classification of smooth projective complex universally
$\CH_0$-trivial surfaces. In this section we will move from dimension
2 to higher dimension, and discuss the universal $\CH_0$-triviality of
complex cubic threefolds and fourfolds.  Throughout, our base field
will be the complex numbers.

\subsection{Cubic threefolds}

We first provide some background about minimal curve classes on
principally polarized abelian varieties. Let $C$ be a smooth
projective curve and $(J(C),\Theta)$ its jacobian with the principal
polarization arising from the theta divisor.  Choosing a rational
point on $C$, there is an embedding $C \hookrightarrow J(C)$ giving
rise to a class $[C] \in \CH_1(J(C))$.  This class is related to the
theta divisor by means of the \linedef{Poincar\'e formula}
$$
[C] = \frac{\Theta^{g-1}}{(g-1)!} \in H^{2g-2}(J(C),\ZZ),
$$
in particular, the class $\Theta^{g-1}/(g-1)!$ is represented by an
effective algebraic class in $\CH_1(J(C))$.  To some extent, the
validity of this formula gives a characterization of jacobians of
curves among principally polarized abelian varieties by the
following result of Matsusaka.

Let $(A,\vartheta)$ be an irreducible principally polarized abelian
variety of dimension $g$.  The class $\vartheta^{g-1}/(g-1)!$ is
always an integral Hodge class in $H^{2g-2}(A,\ZZ)$.  We will say that
this class is algebraic (resp.\ effective) if it is homologically
equivalent to an algebraic (resp.\ effective) class in $\CH_1(A)$, i.e., is
equal to the image of an algebraic (resp.\ effective) cycle in the image of the
class map $\CH_1(A) \to H^{2g-2}(A,\ZZ)$.

\begin{theorem}[Matsusaka~\cite{matsusaka}]
Let $(A,\vartheta)$ be an irreducible principally polarized abelian
variety of dimension $g$.  Then there exists a smooth projective curve
$C$ such that $(A,\vartheta) \isom (J(C),\Theta)$ as principally
polarized abelian varieties if and only if the class
$\vartheta^{g-1}/(g-1)!$ is effective.
\end{theorem}

If the principally polarized abelian variety $(A,\vartheta)$ is not
irreducible, then the result of Matsusaka gives a characterization of
when $(A,\vartheta)$ is a product of jacobians of curves.  In
\cite{clemens_griffiths}, this condition is equivalent to
$(A,\vartheta)$ being of ``level one.'' This characterization gives a
nice reformulation of the Clemens--Griffiths criterion for
nonrationality (see Theorem~\ref{thm:J-is-bir-inv}) of a smooth
projective threefold $X$ with $h^1=h^{3,0}=0$ and intermediate
jacobian $(J(X),\Theta)$:
\begin{quotation}
If $\Theta^{g-1}/(g-1)!$ is not effective, then $X$ is not rational.
\end{quotation}
One can even interpret the proof of Clemens and Griffiths as showing
that $\Theta^4/4!$ is not effective when $X$ is a cubic threefold.

We know that the Clemens--Griffith criterion for nonrationality can
fail to detect stable rationality, in particular, can fail to ensure
universal $\CH_0$-triviality, see \S\ref{subsec:BCTSSD}.  In
hindsight, a natural question is whether there is a strengthening of
the Clemens--Griffiths criterion for obstructing universally
$\CH_0$-nontriviality.  Voisin~\cite{voisin:cubics} provides such a
strengthened criterion.  In the case of cubic threefolds, where the
universal $\CH_0$-triviality is still an open question, her results
are particularly beautiful.

\begin{theorem}[Voisin~\cite{voisin:cubics}]
\label{thm:cubic_threefold_universal}
Let $X$ be a smooth cubic threefold with intermediate jacobian
$(J(X),\Theta)$.  Then $X$ is universally $\CH_0$-trivial if and only
if $\Theta^4/4! \in H^8(J(X),\ZZ)$ is algebraic.
\end{theorem}

\begin{remark}
The problem of whether the very general cubic threefold is not stably
rational is still open.  Voisin~\cite[Thm.~4.5]{voisin:cubics} proves
that cubic threefolds are universally $\CH_0$-trivial over a countable
union of closed subvarieties of codimension at most 3 in the moduli.
Colliot-Th\'el\`ene~\cite{colliot-thelene:partiallement_diagonale} has
provided a new proof of this fact.  Voisin also points out the
striking open problem that there is not even a single principally
polarized abelian variety $(A,\vartheta)$ of dimension $g \geq 4$
known for which $\vartheta^{g-1}/(g-1)!$ is not algebraic!
\end{remark}

To give an idea of the ingredients in the proof, we will first start
with some results of Voisin on the decomposition of the diagonal in
various cohomology theories.

\begin{prop}[{Voisin~\cite[Prop.~2.1]{voisin:cubics}}]
\label{prop:algebraic=>rational}
Let $X$ be a smooth projective variety over a field of characteristic
0.  If $X$ admits a decomposition of the diagonal modulo algebraic
equivalence, then it admits a decomposition of the diagonal.
\end{prop}
The proof uses a special case of the nilpotence conjecture, proved
independently by Voevodsky~\cite{voevodsky:nilpotence} and
Voisin~\cite{voisin:self_products}, stating that correspondences in
$\CH(X\times X)$ algebraically equivalent to 0 are nilpotent for
self-composition.

We point out that the general nilpotence conjecture, for cycles
homologically equivalent to 0, is known to imply, in particular,
Bloch's conjecture for surfaces with $p_g=0$, see
\cite[Rem.~3.31]{voisin:decomposition_diagonal_book}.

\begin{prop}
\label{prop:homol=>algebraic}
Let $X$ be a smooth cubic hypersurface such that
$H^{2i}(X,\ZZ)/\im(\CH^i(X) \to H^{2i}(X,\ZZ))$ has no 2-torsion for all
$i \geq 0$ (e.g., $X$ has odd dimension or dimension 4, or is very
general of any dimension).  If $X$ admits a decomposition of the
diagonal modulo homological equivalence, then it admits a
decomposition of the diagonal modulo algebraic equivalence.
\end{prop}

The proof uses the relationship between a decomposition of the
diagonal on $X$ and on $X \times X$ and the Hilbert scheme of length 2
subschemes $X^{[2]}$, as well as the fact that $X^{[2]}$ is birational
to the total space of a projective bundle over $X$,
cf.~\cite{galkin-shinder-cubic}.  There is also a purely topological
approach to this result due to Totaro~\cite{totaro:Hilb}.

Finally, Voisin provides a general necessary and sufficient condition
for the decomposition of the diagonal modulo homological equivalence
of a rationally connected threefold.

We first recall the Abel--Jacobi map for codimension 2 cycles on a
smooth projective threefold $X$ with intermediate jacobian $J(X)$.
The Griffiths Abel--Jacobi map
$$
\alpha_X : \CH^2(X)_{\homological} \to J(X)(\CC)
$$
is an isomorphism by the work of Bloch and
Srinivas~\cite{bloch_srinivas}, since $\CH_0(X)=\ZZ$.

\begin{definition}
The we say that $X$ admits a \linedef{universal codimension $2$ cycle}
if there exists $Z \in \CH^2(J(X) \times X)$ such that $Z_a=Z_{a\times
X}$ is homologous to $0$ for any $a \in J(X)$ and that the morphism
$\Phi_Z : J(X) \to J(X)$, induced by $a \mapsto \alpha_X(Z_a)$, is
the identity on $J(X)$.
\end{definition}

The existence of a universal codimension $2$ cycle is equivalent to
the tautological class in $\CH^2(X_{\overline{F}})$ being in the image
of the map $\CH^2(X_F) \to \CH^2(X_{\overline{F}})$, where
$F=\CC(J(X))$ is the function field of the intermediate Jacobian, see
\cite[\S5.2]{colliot:descenteCH2}.

\begin{theorem}[{Voisin~\cite[Thm.~4.1]{voisin:cubics}}]
\label{thm:universal}
Let $X$ be a rationally connected threefold and $(J(X),\Theta)$ its
intermediate Jacobian of dimension $g$. Then $X$ admits a
decomposition of the diagonal modulo homological equivalence if and
only if the following properties are satisfied:
\begin{enumerate}
\item $H^3(X,\ZZ)$ is torsionfree.

\item $X$ admits a universal codimension $2$ cycle.

\item $\Theta^{g-1}/(g-1)!$ is algebraic.
\end{enumerate}
\end{theorem}

\begin{remark}
In fact, Theorem~\ref{thm:universal} has the following
generalizations, see~\cite[Thm.~4.2,~Rem.~4.3]{voisin:cubics}.  If
$N\Delta_X$ admits a decomposition modulo homological equivalence then
$N^2 \Theta^{g-1}/(g-1)!$ is algebraic.  Furthermore, if $X$ admits a
unirational parameterization of degree $N$, then $N
\Theta^{g-1}/(g-1)!$ is effective.
\end{remark}

Finally, we outline the proof of
Theorem~\ref{thm:cubic_threefold_universal} due to Voisin.  Let $X$ be
a cubic threefold with intermediate jacobian $(J(X),\Theta)$.
Propositions~\ref{prop:algebraic=>rational} and
\ref{prop:homol=>algebraic} imply that a decomposition of the diagonal
on $X$ is equivalent to a decomposition of the diagonal modulo
homological equivalence.  Since $H^3(X,\ZZ)$ is torsionfree, by
Theorem~\ref{thm:universal}, it would then suffice to
show that $X$ admits a universal codimension $2$ cycle.  However, this
is not known.  Instead, Voisin uses results of Markushevich and
Tikhomirov~\cite{mark_tikh} on parameterizations of $J(X)$ with
rationally connected fibers, which implies, using results from
\cite{voisin:abel-jacobi}, that if $\Theta^{g-1}/(g-1)!$ is algebraic
then $X$ admits a universal codimension $2$ cycle.

We point out that Hassett and
Tschinkel~\cite{hassett_tschinkel:stable} have proved that for every
family of smooth Fano threefold not birational to a cubic threefold,
either every element in the family is rational or the very general
element is not universally $\CH_0$-trivial.  They use the degeneration
method outlined in \S\ref{subsec:degeneration}.

\subsection{Cubic fourfolds}

Cubic fourfolds are rationally connected.  Indeed, they are Fano
hypersurfaces, hence their rational connectivity is a consequence of
the powerful results of \cite{kollar_miyaoka_mori:Fano}.  A more
elementary reason is that they are unirational.  This fact that was
likely known to M.\ Noether (cf.\ \cite[App.~B]{clemens_griffiths}).

\begin{prop}
Let $X\subset \PP^{n+1}$ be a cubic hypersurface of dimension $n \geq
2$ containing a line.  Then $X$ admits a unirational parameterization
of degree 2.
\end{prop}
\begin{proof}
Blowing up a line $\ell \subset X$, we arrive at a conic bundle
$\text{Bl}_\ell X \to \PP^{n-1}$, for which the exceptional divisor
$E$ is a multisection of degree $2$.  Thus the fiber product
$\text{Bl}_\ell X \times_{\PP^{n-1}} E \to E$ is a conic bundle with a
section, hence a rational variety since $E$ is rational.  The map
$\text{Bl}_\ell X \times_{\PP^{n-1}} E \to X$ is generically finite of
degree $2$.  Thus $X$ admits a unirational parameterization of degree
$2$.  
\end{proof}

Of course every smooth cubic hypersurface of dimension at least 2 over
an algebraically closed field contains a line, as one can reduce, by
taking hyperplane sections, to the case of a cubic surface.  In
particular, if $X$ is a cubic fourfold, $\CH_0(X)$ is universally
$2$-torsion.  One way that $X$ could be universally $\CH_0$-trivial is
if $X$ admitted a unirational parameterization of odd degree.  Indeed,
if a variety $X$ admits unirational parameterizations of coprime
degrees, then $X$ is universally $\CH_0$-trivial.  In fact, the
following question is still open.

\begin{question}
Does there exist a nonrational variety with unirational
parameterizations of coprime degrees?
\end{question}

While the existence of cubic fourfolds with unirational
parameterizations of odd degree is currently limited, a beautiful
result of Voisin~\cite[Thm.~5.6]{voisin:cubics} states that in fact
many classes of special cubic fourfolds are universally
$\CH_0$-trivial.

We recall, from \S\ref{subs:cubic-4folds}, the divisors $\cC_d \subset
\cC$ of special cubic fourfolds of discriminant $d$ in the coarse
moduli space $\cC$ of cubic fourfolds.  These are Noether--Lefschetz
type divisors, which are nonempty for $d > 6$ and $d \equiv 0,2 \bmod
6$, see Theorem~\ref{thm:hassett:special}.  We recall that
Voisin~\cite[Thm.~18]{voisin:aspects} has shown the integral Hodge
conjecture for cubic fourfolds, i.e., that the cycle class gives rise
to an isomorphism $\CH^2(X) = H^4(X,\ZZ) \cap H^{2,2}(X)$, and
moreover, that every class in $\CH^2(X)$ can be represented by a
(possibly singular) rational surface.  For $X$ very general in the
moduli space, $\CH^2(X)$ is generated by the square of the hyperplane
class $h^2$, and the rank of $\CH^2(X)$ is $>1$ if and only if $X$
lies on one of the divisors $\cC_d$.  For small values of $d$, the
geometry of additional $2$-cycles $T \in \CH^2(X)$, for general $X
\in \cC_d$, is well understood.  For $d\leq 20$, this was understood
classically.  For $d \leq 38$, Nuer~\cite{nuer} provides explicit
smooth models of the rational surfaces that arise.  It is still an open question as
to whether $\CH^2(X)$ is always generated by classes of smooth
rational surfaces, cf.\ \cite[Question~14]{hassett-cubicsurvey}.
Nuer's approach for $d \leq 38$ provides a unirational
parameterization of $\cC_d$, while it is known that for $d \gg 0$, the
divisors $\cC_d$ become of general type, see \cite{TVA:C_d}.

We are interested in representations of cycle classes for $X \in
\cC_d$ because of the following result on the existence of unirational
parameterizations of odd degree of certain special cubic fourfolds
(see \cite[Cor.~35]{hassett-cubicsurvey}), which was initiated in
\cite[\S7.5]{hassett_tschinkel:rational_curves_homolomorphic_symplectic},
with corrections by Voisin (see \cite[Ex.~38]{hassett-cubicsurvey}).

\begin{prop}
Let $X \in \cC_d$ be a special cubic fourfold whose additional
$2$-cycle $T \subset X$ is a rational surface.  Assume $T$ has
isolated singularities and a smooth normalization.  If $d$ is not
divisible by $4$ then $X$ admits a unirational parameterization of odd
degree.
\end{prop}

In general, we do not know if the required rational surface $T \subset
X$ can always be choosen with isolated singularities and smooth
normalization.  The construction of Nuer~\cite{nuer} provides smooth
rational $T \subset X$ for $d \leq 38$\footnote{Very recently,
Lai~\cite{lai:42} verified that the required rational surface $T
\subset X$ is nodal for $d=42$.}.

However, not assuming the existence of unirational parameterizations
of coprime degree, Voisin has the following result.

\begin{theorem}[{Voisin~\cite[Thm.~5.6]{voisin:cubics}}]
If $4 \nmid d$ then any $X \in \cC_d$ is universally $\CH_0$-trivial.
\end{theorem}
\begin{proof}
We give a sketch of the proof. The additional class $T \in \CH^2(X)$,
such that the discriminant of the sublattice generated by $h^2$ and
$T$ is $d$, can be represented (after adding multiples of $h^2$) by a
smooth surface (which by abuse of terminology we denote by) $T \subset
X$.

First, Voisin proves that (at least under the hypothesis that $X$ is
very general in $\cC_d$), if there exists any closed subvariety $Y
\subsetneq X$ such that $\CH_0(Y) \to \CH_0(X)$ is universally
surjective, then $X$ is universally $\CH_0$-trivial.

Second, letting $T \subset X$ be a smooth surface, consider the
rational map $T \times T \dashrightarrow X$ defined by sending a pair
of points $(x,y)$ to the point residual to the line joining $x$ and
$y$.  Voisin proves that if this map is dominant of even degree not
divisible by 4, then $\CH_0(T) \to \CH_0(X)$ is universally
surjective.

Finally, a calculation with Chern classes shows that if $T \subset X$
is a smooth surface in general position then the rational map $T
\times T \dashrightarrow X$ is dominant of degree $\equiv d \bmod 4$.
Indeed, the degree is equal to twice the number of double points
acquired by $T$ after a generic projection from a point, so
computations must be made comparing the numerology of the double point
formula and the intersection product on $X$.
\end{proof}

We remark that the universal $\CH_0$-triviality is still open in one
of the most interesting classes of special cubic fourfolds, namely
that of cubic fourfolds containing a plane, i.e., $X \in \cC_8$.  One
nontrivial consequence of the universal $\CH_0$-triviality would be
the universal triviality of the unramified cohomology in degree 3.
For cubic fourfolds containing a plane, this was first proved
in~\cite{ACTP}.  For arbitrary cubic fourfolds, this was then proved
by Voisin~\cite[Ex.~3.2]{voisin:threefolds}, with a different proof
given by Colliot-Th\'el\`ene~\cite[Thm.~5.8]{colliot:descenteCH2}
(which still relies on Voisin's proof of the integral Hodge
conjecture).

\subsection{The degeneration method}
\label{subsec:degeneration}

The degeneration method, initiated by
Voisin~\cite[\S2]{voisin:threefolds} and developed by
Colliot-Th\'el\`ene and Pirutka~\cite{CTPirutka}, has emerged as a
powerful tool for obstructing universal $\CH_0$-triviality for various
families of varieties.  The idea is that universal $\CH_0$-triviality
specializes well in families whose central fiber is mildly singular.

Analogous results for the specialization of rationality in families of
threefolds was established by de Fernex and
Fusi~\cite{defernex_fusi}. Already,
Beauville~\cite[Lemma~5.6.1]{beauvilleprym} had proved an analogous
result for the specialization of the Clemens--Griffiths criterion for
nonrationality relying on the Satake compactification of the moduli
space of abelian varieties.  Also,
Koll\'ar~\cite{kollar:hypersurfaces} used a specialization method (to
characteristic $p$) for the existence of differential forms to prove
nonrationality of hypersurfaces of large degree, a result that was
generalized by Totaro~\cite{totaro:hypersurfaces} using the
degeneration method for universally $\CH_0$-triviality.  We will
outline the degeneration method and some of its applications.

First we define a condition on the resolution of singularities of a
singular variety.

\begin{definition}
Let $X_0$ be a proper geometrically integral variety over a field $k$.
We say that a proper birational morphism $f : \widetilde{X}_0 \to X_0$
with $\widetilde{X}_0$ smooth is a \linedef{universally
$\CH_0$-trivial resolution} if $f_* : \CH_0(\widetilde{X}_{0,F}) \to
\CH_0(X_F)$ is an isomorphism for all field extensions $F/k$, and, is
a \linedef{totally $\CH_0$-trivial resolution} if for every
scheme-theoretic point $x$ of $X_0$, the fiber $(\widetilde{X}_0)_x$
is a universally $\CH_0$-trivial variety over the residue field
$k(x)$.
\end{definition}

The notions of universally and totally $\CH_0$-trivial resolutions are
due to Colliot-Th\'el\`ene and Pirutka~\cite{CTPirutka} and define a
new class of singularities that should be classified in the spirit of
the minimal model program.  For example, in characteristic zero, one
might ask whether $X_0$ has rational singularities if it admits a
totally $\CH_0$-trivial resolution.  It is proved (see
\cite[Prop.~1.8]{CTPirutka}) that every totally $\CH_0$-trivial
resolution is universally $\CH_0$-trivial, but not conversely.

\begin{example}
Let $X \to \PP^2$ be a conic bundle of Artin--Mumford type.  Then $X$
has isolated ordinary double points, and the (universally
$\CH_0$-trivial) resolution has nontrivial Brauer group.
\end{example}

Let $X$ be a smooth proper geometrically integral variety over $k$.
The \linedef{degeneration method} proceeds as follows:
\begin{enumerate}
\item Fit $X$ into a proper flat family $\mathcal{X} \to B$ over a
scheme $B$ of finite type, and let $X_0$ be a possibly singular
fiber.  Assume for simplicity that the generic fiber is regular.

\item Prove that $X_0$ admits a universally $\CH_0$-trivial resolution
$f : \widetilde{X}_0 \to X_0$.

\item Prove that $\widetilde{X}_0$ is not universally $\CH_0$-trivial.
\end{enumerate}
The outcome is that the very general fiber of the family $\mathcal{X}
\to B$ (though perhaps not $X$ itself) will not be universally
$\CH_0$-trivial.

\medskip

Part \textit{(i)} is, to a large extent, informed by the possibility
of achieving part \textit{(iii)}.  To this end, one is mostly
concerned with finding good singular varieties $X_0$ whose resolutions
have nontrivial unramified cohomological invariants or differential
forms.  Then one hopes that \textit{(ii)} can be verified for these
singular varieties.  For example, conic bundles of Artin--Mumford type
have been used quite a lot.  Koll\'ar~\cite{kollar:hypersurfaces} has
constructed hypersurfaces in characteristic $p$ with nontrivial global
differential forms.

\begin{example}
\item In \cite{voisin:threefolds}, a quartic double solid with $\leq
7$ nodes is shown to degenerate to an Artin--Mumford example.

\item In \cite{CTPirutka}, a quartic threefold degenerates to a
singular quartic hypersurface model birational to an Artin--Mumford
example.

\item In \cite{hassett_kresch_tschinkel:conic}, a conic bundle over a
rational surface, whose discriminant curve degenerates to a union of
curves of positive genus, is shown to degenerate to an Artin--Mumford
example.

\item In \cite{hassett_tschinkel:stable}, smooth Fano threefolds in a
family whose general element is nonrational, are shown to degenerate
to an Artin--Mumford example.

\item In \cite{totaro:hypersurfaces}, hypersurfaces of large degree
were already shown by Koll\'ar to degenerate to singular
hypersurfaces in characteristic $p$ with nonzero global differential
forms.
\end{example}

We find it striking that all successful instances of the
equicharacteristic degeneration method for threefolds over $\CC$ use
singular central fibers that are birational to threefolds of
Artin--Mumford type.  However, over arbitrary fields, there are other
methods, see \cite{chatz-levine}, \cite{colliot:cubiques}.

\medskip

We shall say a few words about the proof of the degeneration method by
Colliot-Th\'el\`ene and Pirutka~\cite{CTPirutka}.  First there is a
purely local statement about schemes faithfully flat and proper over a
discrete valuation ring, to the extent that if the special fiber
admits a universally $\CH_0$-trivial resolution, then universal
$\CH_0$-triviality of the generic fiber implies the universal
$\CH_0$-triviality of the special fiber.  This purely local statement
uses the specialization homomorphism as developed in
\S\ref{subsec:specialization}.  To get the statement about the very
general fiber of a family over a base $B$, there is a ``standard''
argument using Chow schemes, see \cite[App.~B]{CTPirutka}.

\section{Categorical representability in higher dimension}\label{sect:high-dim-Db}
In this section, we turn to higher dimensional varieties, in particular to varieties
of dimension 3 and 4. Even though the constructions work over any field, and
considerations related to weak factorization hold over any field of characteristic zero
(see~\cite{abrmovichandco-weak}), we consider here only the case where $k=\CC$
(or $k$ algebraically closed of characteristic zero).
The categorical questioning is already very rich and deep
in these cases.

The aim of this section is to motivate, by examples and motivic
arguments, the following question:

\begin{question}\label{quest:repincodim2}
Is categorical representability in codimension $2$ a necessary condition for rationality?
That is, if $X$ is rational, do we have $\coRepcat(X) \geq 2$?
\end{question}

Let us first notice that, as soon as we consider varieties of
dimension at least 3, we easily find examples of non-rational
varieties $X$ with $\coRepcat(X) \geq 2$, the easiest example being a
projective bundle $X \to C$ of relative dimension at least 2 over a
curve $C$ with $g(C)>0$. We thus restrict our attention to Mori fiber
spaces $X \to Y$ over varieties of negative Kodaira dimension. In this
case Proposition~\ref{prop:decoMFS} gives a natural subcategory
$\cat{A}_{X/Y}$ as a complement of a finite number of copies of
$\Db(Y)$. We now present some evidence and motivic considerations in
order to argue that $\Rep\cat{A}_{X/Y}$ should witness obstruction to
rationality.

Let $\pi:X \to Y$ be a Mori fiber space of relative dimension $m$, and
let $n$ be the dimension of $X$.  First of all, notice that if
$\Rep\cat{A}_{X/Y} \leq d$, then $\Repcat(X) \leq \max\{\dim(Y), d\}$.
It then follows that to study $\coRepcat(X)$ we should focus on
$\cat{A}_{X/Y}$ and its representability\footnote{If $m>1$, this is
obvious. If $m=1$, then the rationality of $X$ implies the rationality
of $Y \times \PP^1$, and hence the stable rationality of $Y$.  For
surfaces over $\CC$, stable rationality implies rationality, so if $X$
is a 3-dimensional conic bundle over a rational surface $S$, then the
obstruction is exactly contained in $\cat{A}_{X/Y}$, see
\cite{bolo_berna:conic} or
Theorem~\ref{thm:classification-of-rational-threefolds}.  For
threefolds, note that if $Y$ is nonrational generic Fano threefold,
but not a cubic threefold, then $Y$ is not stably rational,
see~\cite{hassett_tschinkel:stable}.  It follows that if $X \to Y$ is
a conic bundle over a stably rational threefold, then the obstruction
is again contained in $\cat{A}_{X/Y}$, unless perhaps if $Y$ is a
cubic threefold.}.

\subsection{Motivic measures and a rational defect}

Let us quickly recall some general motivation for
Question~\ref{quest:repincodim2}.  Bondal, Larsen, and Lunts defined
the Grothendieck ring $PT(k)$ of (dg enhanced) triangulated $k$-linear
categories~\cite{bondal-larsen-lunts}, by considering the free
$\ZZ$-module generated by equivalence classes of such categories,
denoted by $I(-)$, and introducing a scissor-type relation:
$I(\cat{T})=I(\cat{A})+I(\cat{B})$ if there is a semiorthogonal
decomposition $\cat{T}=\sod{\cat{A},\cat{B}}$.  The product of this
ring is a convolution product, in such a way that the product of
$I(\Db(X))$ and $I(\Db(Y))$ coincides with $I(\Db(X \times
Y))$. See~\cite{bondal-larsen-lunts} for more details.  We notice that
the unit $\bf{e}$ of $PT(k)$ is the class of $\Db(k)$ and that if
$\cat{T}$ is generated by $r$ $k$-exceptional objects, then
$I(\cat{T})=r {\bf e}$.

One can consider the following subsets of $PT(k)$:
$$PT_d(k) := \sod{ I(\cat{T}) \in PT(k) \,\, \vert \,\,  \Rep\cat{T} \leq d }^{+},$$
where $\sod{-}^+$ is the smallest subset closed under summands. One can show that these subsets are indeed subgroups
providing a ring filtration of $PT(k)$.

Notice that, by definition, if $\Rep\cat{T} \leq d$, then $I(\cat{T})$
is in $PT_d(k)$, but the converse is not true in general, even for
$d=0$, as the following example shows.

\begin{example}
Recall from Remark~\ref{rmk:warning-doubledeco} that Kuznetsov has
constructed a complex threefold $X$ generated by exceptional objects
not satisfying the Jordan--H\"older property~\cite{kuznet:JH}. In
particular, this is based on the description, originally due to
Bondal, of a quiver $Q$ with three vertexes and relations, so that
there are exceptional objects $E_1, E_2, E_3, F$ in $\Db(Q)$ and
semiorthogonal decompositions:
$$\begin{array}{lr}
  \Db(Q)=\sod{E_1,E_2,E_3}, \,\,\, & \,\,\, \Db(Q)=\sod{\cat{T},F},
  \end{array}$$
such that $\cat{T}=F^\perp$ has no exceptional object. It follows that
$I(\cat{T})$ lies in $PT_0(k)$, since $I(\Db(Q))$ does, but $\Rep\cat{T} > 0$.
\end{example}

It would thus be very interesting to give conditions under which a
category $\cat{T}$, admissible in some category generated by
$k$-exceptional objects (whence $I(\cat{T}) \in PT_0(k)$), admits a
full $k$-exceptional collection.

One can consider the Grothendieck ring $K_0(\mathrm{Var}(k))$ of
$k$-varieties whose unit $1=[\Spec(k)]$ is the class of the point. If
weak factorization holds, then this can be seen as the $\ZZ$-module
generated by isomorphism classes of smooth proper varieties with the
relation $[X] - [Z] = [Y] - [E]$ whenever $Y \to X$ is the blow-up
along the smooth center $Z$ with exceptional divisor $E$, see
\cite{bittner}.  Larsen and Lunts have then shown that there is a
surjective ring morphism (a \linedef{motivic measure}) $\mu:
K_0(\mathrm{Var}(k)) \to \ZZ[SB]$ to the ring generated by stable
birational equivalence classes. That is, $\ZZ[SB]$ is the quotient of
the Grothendieck ring $K_0(\mathrm{Var}(k))$ by the stable birational
equivalence relation.

Moreover, $\mathrm{ker} \mu = \sod{\LL}$, the ideal generated by the class $\LL$ of the affine line.
It follows that (as remarked in~\cite{galkin-shinder-cubic}), if $X$ is rational of dimension $n$, then:
\begin{equation}\label{eq:mot-rat-defect}
[X] = [\PP^n] + \LL M_X
\end{equation}
in $K_0(\mathrm{Var}(k))$, where $M_X$ is a $\ZZ$-linear combination of classes of varieties of
dimension bounded above by $n-2$. Galkin and Shinder define then $([X]-[\PP^n])/\LL \in K_0(\mathrm{Var}(k))[\LL^{-1}]$ as the
\linedef{rational defect} of $X$~\cite{galkin-shinder-cubic}.
 
On the other hand, Bondal, Larsen, and Lunts~\cite{bondal-larsen-lunts} show that the assignment
$$\nu: K_0(\mathrm{Var}(k)) \to PT(k),\,\,\,\,\,\,\,\,\,\,\,\,\,\,\, [X] \mapsto I(\Db(X))$$
also defines a motivic measure. 
Moreover, since $I(\Db(\PP^1))=2 {\bf e}$ and
$[\PP^1]=1 + \LL$ in $K_0(\mathrm{Var}(k))$, we have that $\nu(\LL)={\bf e}$. It follows from
\eqref{eq:mot-rat-defect} that if $X$ is rational of dimension $n$, then $I(\Db(X))$ is in $PT_{n-2}(k)$.
We can state the following result, motivating Question~\ref{quest:repincodim2}.

\begin{prop}\label{prop:motivic-nv-obstruct}
If $X$ is a smooth and projective variety of dimension $n$ such that $I(\Db(X))$ is not in $PT_{n-2}(k)$, then
$X$ is not rational.
\end{prop}
\begin{definition}
If $X$ is a smooth and projective variety of dimension $n$, the class of the element $I(\Db(X))$ in the group $PT(k)/PT_{n-2}(k)$ is called
the \linedef{noncommutative motivic rational defect} of $X$.
\end{definition}

We end by commenting the fact that Proposition~\ref{prop:motivic-nv-obstruct} is a rather weak result.
Indeed, as remarked above, we have an implication $\Repcat(X) \leq i \Rightarrow I(\Db(X)) \in PT_i(k)$,
but the converse implication is in general not known, even for $i=0$.

However, Proposition~\ref{prop:motivic-nv-obstruct} indicates that in the case of Mori fiber spaces
the category $\cat{A}_{X/Y}$ should be the object to consider.

\begin{corollary}\label{cor:AXY-is-the-right-one}
Let $X \to Y$ be a Mori fiber space of relative dimension $m$, and let $n= \dim(X)$.
Assuming that either $Y$ is rational or $m>1$, we have that $I(\Db(X))$ is in $PT_{n-2}(k)$ if and only if
$I(\cat{A}_{X/Y})$ is in $PT_{n-2}(k)$.
\end{corollary}
\begin{proof}
The assumptions on $Y$ imply that $I(\Db(Y))$ is in $PT_{n-2}(k)$: first of all note that $\dim(Y) \leq n-1$.
Then, if $Y$ is rational, we have $I(\Db(Y))$ in $PT_{n-3}(k)$. On the other hand, if $\dim(Y) \leq n-2$, we have $I(\Db(Y))$ 
in $PT_{n-2}(k)$. Then we conclude using the
definition of $PT_i(k)$ and the relation $I(\Db(X))=m I(\Db(Y))+ I(\cat{A}_{X/Y})$.
\end{proof}

\begin{question}\label{quest:repcat-invariance-for-MFS}
Let $X$ be a smooth projective variety and $X \to Y$ and $X \to Z$ two Mori fiber space
structures. Is $\Rep(\cat{A}_{X/Y})=\Rep(\cat{A}_{X/Z})$?
\end{question}

Notice that we could extend our analysis to $X \dashrightarrow Y$, a rational map whose resolution
$\widetilde{X} \to Y$ is a Mori fiber space. By abuse of notation, we denote $\cat{A}_{X/Y,\rho}:=\cat{A}_{\widetilde{X}/Y}$
(even though $\cat{A}_{X/Y,\rho}$ is not necessarily a subcategory of $\Db(X)$).
For example, $X$ is a cubic threefold and $X \dashrightarrow \PP^2$ the projection along any
line in $X$, which is resolved into $\widetilde{X} \to \PP^2$, a conic bundle.

\begin{corollary}
Suppose that there is a rational map $\rho: X \dashrightarrow Y$ and a commutative diagram:
$$\xymatrix{
\widetilde{X} \ar[dr]^\pi \ar[d]_\epsilon & \\
X \ar@{-->}[r]^\rho & Y,
}$$
where $\pi: \widetilde{X} \to Y$ is a Mori fiber space of relative dimension $m$ and $\epsilon: \widetilde{X} \to X$ is a blow up
along a smooth center. Assuming that either $Y$ is rational or $m>1$, we have that $I(\Db(X))$ is in $PT_{n-2}(k)$ if and only if
$I(\cat{A}_{X/Y,\rho})$ is in $PT_{n-2}(k)$.
\end{corollary}
\begin{proof}
We notice that $\Db(\widetilde{X})$ has two decompositions, one given by the Mori fiber space map $\widetilde{X} \to Y$
and the other given by the blow-up of $X$ hence containing a copy of $X$ and a finite number of copies of the blown-up loci.
Corollary~\ref{cor:AXY-is-the-right-one} applies then again, once we write the two decompositions of $I(\Db(\widetilde{X}))$
\end{proof}

\begin{question}\label{quest:rational-MFS}
Let $X$ be a smooth projective Fano variety of dimension $n$, and $\rho:X \dashrightarrow Y$ and $\sigma: X \dashrightarrow Z$ be rational Mori fiber spaces
as above. Is $\Rep \cat{A}_{X/Y,\rho} > n-2$ if and only if $\Rep \cat{A}_{X/Z,\sigma} > n-2$?
\end{question}

\subsection{Threefolds}
In this subsection, we consider Questions~\ref{quest:repincodim2}, \ref{quest:repcat-invariance-for-MFS} and~\ref{quest:rational-MFS}
for threefolds. Let us first notice that we only consider Mori fiber spaces $X \to Y$ with $Y$ of negative Kodaira dimension. Moreover,
being interested in rationality, we can exclude the cases where $Y$ is a ruled surface over a curve of positive genus. It follows that we only consider Fano
threefolds of Picard rank one, relatively minimal del Pezzo fibrations over $\PP^1$ and relatively minimal (or \linedef{standard}) conic bundles over rational surfaces. 

On one hand, there are now many examples of known semiorthogonal
decompositions describing $\cat{A}_{X/Y}$ for such varieties,
especially for Fano and conic bundles. We recall the most of them in
Table~\ref{table:fano-sods}.

On the other hand, recall from \S\ref{sect:intjacobians}, that such a
threefold $X$ has a unique principally polarized intermediate Jacobian
$J(X)$, and that one can define the Griffiths component $A_X \subset
J(X)$ to be the maximal component not split by Jacobians of curves. We
consider a stronger assumption on $J(X)$, namely that it carries an
\linedef{incidence polarization}
(see~\cite[D\'ef. 3.2.3]{beauvilleprym}), defined as follows.  For any
algebraic variety $T$, and $z$ a cycle in $\CH^2_\QQ(T \times X)$, the
{\em incidence correspondence $I(z)$ associated to $z$} is the
equivalence class of the cycle $ r_\ast (p^\ast(z)\cdot q^\ast(z)) \in
\CH^1_\QQ(T\times T)$, where $p$, $q$ and $r$ are the projections from
$T \times T \times X$ away from the first, the second and the third
factor respectively.  Recall that $J(X)$ carries a principal
polarization, which can be seen as an element $\theta$ in $\CH^1(J(X)
\times J(X))$. Moreover, we assume that $J(X)$ represents the group
$A^2(X) \subset \CH^2(X)$ of (integral) algebraically trivial
codimension 2 cycles on $X$, that is there is a universal regular map
$G:A^2(X) \to J(X)$, such that for every regular map $g: A^2(X) \to B$
to an abelian variety $B$, there is a unique morphism of abelian
varieties $u \colon J(X) \to B$ such that $u \circ G = g$.  Finally,
if all these properties are satisfied, we say that the principal
polarization $\theta$ of $J(X)$ is an \emph{incidence polarization} if
for any algebraic map $f: T \to A^2(X)$ defined by a cycle $z \in
\CH^2(T\times X)$ the equality $(G \circ f)^\ast(\theta)= I(z)$ holds.

Having a principally polarized intermediate Jacobian carrying an
incidence polarization may seem a rather restrictive assumption, but
this is actually satisfied by most of the (general) Mori fiber spaces
over rational bases we are considering in this paragraph: in these
cases $J(X)$ is known to carry such a polarization, unless $X$ is a
Fano of index 2 and degree 1 or a del Pezzo fibration of degree 1
(see~\cite[Rmk. 3.8]{bolognesi_bernardara:representability} for more
details).

\begin{theorem}[\cite{bolognesi_bernardara:representability}]\label{thm:reconstr-jacobians}
Suppose that $X$ is a threefold with principally polarized intermediate Jacobian $J(X)$ carrying
an incidence polarization. Then, assuming $\coRepcat(X) \geq 2$, the Griffiths component
$A_X \subset J(X)$ is trivial.
\end{theorem}

\begin{proof}[Sketch of proof]
The first step in the proof is the classification of categories representable in
dimension at most 1, see Proposition~\ref{prop:cat-rep-dim1}. In the complex case, this means
that there exist smooth and projective curves $\{C_i\}_{i=1}^s$ and a semiorthogonal
decomposition 
$$\Db(X) = \sod{\Db(C_1),\ldots, \Db(C_s),E_1,\ldots, E_r},$$
where $E_i$ are exceptional. This implies, via Grothendieck--Riemann--Roch,
that
$$\CH^*(X)_\QQ \simeq \QQ^{\oplus r} \oplus \bigoplus_{i=1}^s \CH^*(C_i)_\QQ,$$
and that the correspondences giving the maps $\phi_i: \CH^i(C_i)_\QQ \to \CH^*(X)_\QQ$
are obtained via the kernel of the Fourier--Mukai full and faithful functors
$\Phi_i : \Db(C_i) \to \Db(X)$.

One can show that $\phi_i$ induces an isogeny $J(C_i) \to J(X)$ onto an Abelian subvariety
(i.e., its kernel is finite),
since it has to send algebraically trivial cycles to algebraically trivial cycles, and
since the adjoint to the Fourier--Mukai functor provides a retraction $\psi_i$ of $\phi_i$ up to torsion.
Finally, one can check, using the explicit description of the kernel of the adjoint
and the incidence property of the polarization of $J(X)$, that $\phi_i$ actually preserves the principal polarization.
\end{proof}

Notice that Theorem~\ref{thm:reconstr-jacobians} has a much stronger generalization
relying on the theory of noncommutative motives~\cite{berna-tabuada-jacobians} which
allows one to define the Jacobian of any admissible subcategory of the derived category
$\Db(X)$ of a smooth projective variety.
As recalled, Theorem~\ref{thm:reconstr-jacobians} applies to almost all general threefolds
$X$ under examination, except Fano of index two and degree one or del
Pezzo fibrations of degree 1 over $\PP^1$, whose polarization of the
intermediate Jacobian is knot known to be incidence.

Secondly, by a classification argument, and thanks to the work of
many authors (see Table~\ref{table:fano-sods}), we can state a converse statement for
Theorem~\ref{thm:reconstr-jacobians}.

\begin{theorem}\label{thm:classification-of-rational-threefolds}
Let $X$ be a complex threefold. Assume that $X$ is either:
\begin{itemize}
 \item a Fano threefold of Picard rank one, very general in its moduli space, not of index 2 and degree 1; or
 \item a del Pezzo fibration $X \to Y=\PP^1$ of degree 4, or
 \item a standard conic bundle $X \to Y$ over a minimal rational surface.
\end{itemize}
Then $X$ is rational if and only if $\coRepcat(X) \geq 2$. In particular,
this is the case if and only if $\Rep\cat{A}_{X/Y} \leq 1$. 
\end{theorem}
\begin{proof}
Suppose that such $X$ is not rational. Then it is known - see e.g.~\cite[Table 1]{beauville:luroth}
for the Fano cases (very general quartics and sextic double solids can be treated by the same degeneration arguments as in~\cite[\S5]{beauvilleprym}), \cite{alekseev:dP4} for the del Pezzo fibrations
and~\cite{beauvilleprym, shokuprym} for the conic bundles - that the rationality defect of $X$ is
detected by a nontrivial Griffiths invariant $A_X \subset J(X)$. These Jacobians carry an
incidence polarization, so that Theorem~\ref{thm:reconstr-jacobians} implies that $\coRepcat(X) \leq 1$
and also that $\Rep \cat{A}_{X} \geq 2$.

Conversely, we know a list of the rational general such varieties: they are $\PP^3$, quadric hypersurfaces,
Fano varieties of index 2 and degree 5 or 4 (the latter are intersections of two quadrics), Fano varieties of index 1 and degree
22, 18, 16 or 12; conic bundles over $\PP^2$ with discriminant divisor of degree $\leq 4$ or of degree
$5$ and even theta-characteristic; conic bundles over Hirzebruch surfaces with trigonal or hyperelliptic
discriminant divisor; del Pezzo fibrations of degree 4 over $\PP^1$ are birational to a conic bundle
over a Hirzebruch surface. For those varieties, we recall all the known semiorthogonal decompositions
and descriptions of $\cat{A}_{X/Y}$, with the corresponding references, in Table~\ref{table:fano-sods}.

Notice that $\PP^1$-bundles over rational surfaces, $\PP^2$-bundles and quadric fibrations over $\PP^1$ also
(trivially) fit the statement.
\end{proof}

Recall that there exist nonrational threefolds with trivial Griffiths invariant, and even with trivial
intermediate Jacobian. For example, if $X$ is 
the Artin and Mumford double solid~\cite{artin_mumford} recalled in Example~\ref{ex:artin_mumford},
the obstruction to rationality is not given by a nontrivial Griffiths component, but rather by a nontrivial
unramified class. In this case, $X$ is singular
but can be resolved by blowing-up its ten double points $\widetilde{X} \to X$. 

\begin{prop}\label{prop:AMDS}
Let $X$ be the Artin--Mumford quartic double solid and $\widetilde{X} \to X$
be the blow-up of its ten double points. Then $J(\widetilde{X})=0$ and
$\widetilde{X}$ is not rational. Moreover, $\Db(\widetilde{X})$ is a noncommutative
resolution of singularities of $X$ and $\coRepcat(\widetilde{X})=1$.
\end{prop}
\begin{proof}
The fact that $\widetilde{X}$ is nonrational and has trivial Jacobian
goes back to the original paper of Artin and
Mumford~\cite{artin_mumford}: indeed, $X$ is not rational, and its
cohomologies are explicitly calculated. In particular $h^{1,2}(X)=0$,
so that it is easy to get $h^{1,2}(\widetilde{X})=0$ which implies
$J(\widetilde{X})=0$. Finally, just recall that $\Db(\widetilde{X})$
is a noncommutative resolution of $\Db(X)$ since $\widetilde{X} \to X$
is a resolution of singularities.

We are going to prove that $\Repcat(\widetilde{X})=2$ by using an
explicit semiorthogonal decomposition: first by showing that
$\Repcat(\widetilde{X})\leq 2$, then by showing that the inequality
cannot hold strictly.

Hosono and Takagi~\cite{hosono-takagi} consider the Enriques surface
$S$ associated to $X$ (the so-called Reye congruence), and show that
there is a semiorthogonal decomposition
$$\Db(\widetilde{X})=\sod{\Db(S), E_1,\ldots, E_{12}},$$
where $E_i$ are exceptional objects.  This implies first that
$\Repcat(\widetilde{X}) \leq 2$.

We want to prove that the inequality cannot hold strictly. First of
all, since $J(\widetilde{X})=0$, we cannot have admissible
subcategories of $\Db(\widetilde{X})$ equivalent to $\Db(C)$ for some
positive genus curve $C$.  It follows by
Proposition~\ref{prop:cat-rep-dim1} (and $k=\CC$) that
$\Repcat(\widetilde{X}) \leq 2$ implies either $\Repcat(\widetilde{X})
=2$ or $\Repcat(\widetilde{X}) =0$.

Let us exclude the second case.  Notice that we have
$K_0(\widetilde{X})= \ZZ^{12} \oplus K_0(S)$.  Moreover, the
$2$-torsion subgroup $K_0(S)_2$ of $K_0(S)$ is nontrivial: we have
$K_0(S)_2=\ZZ/2\ZZ$.  Indeed, if $S$ is an Enriques surface, the Chern
character is integral and gives an isomorphism between $K_0(S)$ and
the singular cohomology of $S$ (similarly, one can argue by using the
Bloch conjecture, which is true for $S$, and the topological
filtration of the Grothendieck group of $S$). In particular, $K_0(S) =
\ZZ \oplus \Pic(S) \oplus \ZZ$ and $\Pic(S) \simeq \ZZ^{10} \oplus
\ZZ/2\ZZ$ (see, e.g.,~\cite[VIII Prop. 15.2]{bhpv}).  We conclude that
$$K_0(\widetilde{X}) \simeq \ZZ^{\oplus 22} \oplus \ZZ/2\ZZ.$$
But if $\Repcat(\widetilde{X}) =0$, then $K_0(\widetilde{X})$ is free of finite rank, so we conclude.\footnote{Notice that we can also use
Proposition~\ref{prop:cat-rep-dim1} since $K_0(X)_2$ is a one dimensional free $\ZZ/2\ZZ$-module.}.
\end{proof}

Finally notice that there could be other noncommutative resolution of
singularities $X$, such as a small resolution $X^+ \to X$; note that
$X^+$ is a non projective Moishezon manifold. Similar arguments, based
on the semiorthogonal decomposition described by Ingalls and
Kuznetsov~\cite{ingalls-kuznetsov}, show that $\Rep \Db(X^+) =2$ as
well. In that case, the component $\cat{T} \subset \Db(X^+)$ with
$\Rep \cat{T}=2$ is a crepant categorical resolution of singularities
of $\cat{A}_X$.  Notice also that, despite the fact that $X^+$ is not
an algebraic variety, it is however an algebraic space, see Artin
\cite{Artin-Moish}, so that the category $\Db(X^+)$ makes sense.

\begin{table}
\small{
\centering
\begin{tabular}{||c|c||c|c||}
\hline\hline
$Y$ & $X$ & explicit description of $\cat{A}_{X/Y}$ & Ref. \\
\hline \hline
\multicolumn{4}{||c||}{\multirow{2}{*}{Threefolds}}\\
\multicolumn{4}{||c||}{}\\
\hline \hline
\multirow{14}{*}{$pt$} & $\PP^3$ & $0$ & \cite{beilinson} \\ 
\cline{2-4}
 & quadric 3fold & $\sod{S}$, a spinor bundle & \cite{kapranovquadric} \\
\cline{2-4}
 & index $2$, degree $5$ & $\sod{U,V}$, two vector bundles & \cite{orlovV5} \\
 \cline{2-4}
 & int. of 2 quadrics & $\Db(C)$ with $g(C)=2$ and $J(X)=J(C)$ & \cite{bondal_orlov:semiorthogonal} \\
 \cline{2-4}
 & cubic 3fold & fractional Calabi-Yau & \cite{kuznetsov:v14} \\
 \cline{2-4}
 & nodal cubic 3fold* & $\sod{\Db(C),L}$, $L$ a l.bd., $g(C)=2$ and $J(X)=J(C)$  & \cite{bolognesi_bernardara:representability} \\
 \cline{2-4}
 & determinantal cubic 3fold* & four exceptional objects  & \cite{berna-bolo-faenzi} \\
 \cline{2-4} 
 & quartic double solid & noncommutative Enriques surface & \cite{kuznet-perry} \\
 \cline{2-4}
 & Artin-Mumford double solid* & complement of exc. collection on Enriques surface & \cite{ingalls-kuznetsov} \\
 \cline{2-4}
 & index $1$, degree $22$ & $\sod{U,V,W}$, three vector bundles & \cite{kuznetsov:v22}\\
 \cline{2-4}
 & index $1$, degree $18$ & $\sod{\Db(C),U}$, $U$ a v.b., $g(C)=2$ and $J(C)=J(X)$ & \cite{kuznetsov:hyp-sections} \\
 \cline{2-4}
 & index $1$, degree $16$ & $\sod{\Db(C),U}$, $U$ a v.b., $g(C)=3$ and $J(C)=J(X)$ & \cite{kuznetsov:hyp-sections} \\
 \cline{2-4}
 & index $1$, degree $14$ & $\sod{\cat{B},U}$, $U$ a v.b., $\cat{B}$ fractional Calabi-Yau & \cite{kuznetsov:v14} \\
 \cline{2-4}
 & index $1$, degree $12$ & $\sod{\Db(C),U}$, $U$ a v.b., $g(C)=7$ and $J(C)=J(X)$ & \cite{kuznetsov:v12} \\
 \cline{2-4}
 & Hyperell. Gushel--Mukai & noncommutative Enriques surface & \cite{kuznet-perry} \\
 \cline{2-4}
 & int. of 3 quadrics & $\Db(\PP^2,\kc_0)$, $\kc_0$ a Clifford algebra & \cite{kuznetquadrics} \\
 \cline{2-4}
 & quartic 3fold & fractional Calabi--Yau & \cite{kuznetsov:v14} \\
\hline 
\multirow{4}{*}{$\PP^1$} & $\PP^2$-bundle & $0$ & \cite{orlovprojbund} \\
\cline{2-4} & quadric bundle, simple deg. & $\Db(C)$, $C$ hyperelliptic and $J(X)=J(C)$ & \cite{kuznetquadrics} \\
\cline{2-4} & DP4 fibration & $\Db(S,\kc_0)$, $S$ a Hirz. surf, $\kc_0$ a Cliff. algebra & \cite{auel-berna-bolo} \\
\cline{2-4} & rational DP4 fibration & $\Db(C)$, with $J(C)=J(X)$& \cite{auel-berna-bolo} \\
\hline
\multirow{2}{*}{rat.surf.} & $\PP^1$-bundle & $0$ & \cite{orlovprojbund} \\
\cline{2-4} & conic bundle & $\Db(Y,\kc_0)$, $\kc_0$ a Clifford algebra & \cite{kuznetquadrics} \\
\hline
\multirow{2}{*}{$\PP^2$} & rational conic bundle & $\sod{\Db(C),V}$, $V$ v.b and $J(C)=J(X)$ if $g(\Gamma) > 1$ & \cite{bolo_berna:conic} \\
\cline{3-4} & with deg. curve $\Gamma$ & $\sod{L_1,L_2,L_3}$, $L_i$ line bd. if $g(\Gamma)=1$ & \cite{bolo_berna:conic} \\
\hline
\multirow{4}{*}{Hirz.} & & $\sod{\Db(C),V_1,V_2}$,  & \multirow{2}{*}{\cite{bolo_berna:conic}} \\
&rational conic bundle  & $V_i$ v.b and $J(C)=J(X)$ if $g(\Gamma)=3$ &\\
\cline{3-4} & with deg. curve $\Gamma$ & $\sod{\Db(C_1),\Db(C_2)}$,  & \multirow{2}{*}{\cite{bolo_berna:conic}} \\
 & & with $J(X)=J(C_1)\oplus J(C_2)$ if $g(\Gamma)=2$ &\\
\hline \hline
\multicolumn{4}{||c||}{\multirow{2}{*}{fourfolds}}\\
\multicolumn{4}{||c||}{}\\
\hline \hline
\multirow{2}{*}{any} & $\PP$-bundle & $0$ & \cite{orlovprojbund} \\ 
\cline{2-4} & quadric bundle & $\Db(Y,\kc_0)$, $\kc_0$ Cliff. algebra & \cite{kuznetquadrics} \\
\hline
\multirow{11}{*}{pt} & quadric 4fold & $\sod{S_1,S_2}$, spinor bundles & \cite{kapranovquadric} \\
\cline{2-4}
 & cubic fourfold & a noncommutative K3 & \cite{kuz:4fold} \\
\cline{2-4} 
 & Pfaffian cubic & $\Db(S)$, with $S$ a degree $14$ K3 & \cite{kuz:4fold} \\
\cline{2-4} 
 & \multirow{2}{*}{cubic with a plane} & $\Db(S,\alpha)$, with $S$ a degree $2$ K3 and $\alpha \in \mathrm{Br}(S)$ & \multirow{2}{*}{\cite{kuz:4fold}} \\
 & & $\alpha =0$ iff the associated quadric fibr. has a section & \\
 \cline{2-4}
 & nodal cubic* & $\Db(S)$, with $S$ a degree $6$ K3 & \cite{kuz:4fold} \\
 \cline{2-4}
 & determinantal cubic* & six exceptional objects & \cite{berna-bolo-faenzi} \\
 \cline{2-4}
 & general cubic in ${\mathcal C}_d$, & \multirow{3}{*}{$\Db(S)$, with $S$ a degree $d$ K3} & \multirow{3}{*}{\cite{at12}} \\
 &  $d$ not multiple of & &\\
 & $4$, $9$ or $p\equiv_3 2$ odd prime & & \\
 \cline{2-4}
 & deg $10$ ind $2$ in $G(2,5)$ & a noncommutative K3 & \cite{kuz:ICM2014} \\
 \hline
\multirow{2}{*}{$\PP^1$} & fibration in intersections & $\Db(T,\alpha)$, $T \to \PP^1$ hyperell. fib. $\alpha \in \mathrm{Br}(T)$, & \multirow{2}{*}{\cite{auel-berna-bolo}}\\
 & of 2 quadrics &  $\alpha=0$ if $S$ has a line/$\PP^1$ &\\
\hline \hline 
\end{tabular}}
\caption{Known descriptions of $\cat{A}_{X/Y}$, for $X \to Y$ a Mori fiber space, and $\dim(X)=3$ or $\dim(X)=4$ and $\kappa(Y) = -\infty$.
In the nonsmooth cases, indicated by *, the description refers to a categorical resolution of $\cat{A}_{X/Y}$} 
\label{table:fano-sods}
\end{table}

 \subsection{Cubic Fourfolds}

This section is completely devoted to complex cubic fourfolds. From now on,
$X$ denotes a smooth hypersurface of degree 3 of the complex projective space $\PP^5$.
As we have seen in \S\ref{subs:cubic-4folds}, one of the most useful tools to investigate
the geometry of $X$ is Hodge theory. In particular, we have seen how to construct
Noether--Lefschetz type divisors $\cC_d$, and to relate the numerical properties
of $d$ to the existence of some K3 surface $S$ whose geometry is intimately related to the one of $X$.

We present results showing how derived categories and semiorthogonal
decompositions, in the spirit of Question~\ref{quest:repincodim2} and even beyond,
provide a new language which superposes, and, hopefully, extends the Hodge theoretic approach.
We start with an observation by Kuznetsov that we recalled in Corollary~\ref{cor:cubic-4folds=nck3}:

\smallskip

The category $\cat{A}_X$ is a noncommutative K3 surface.

\smallskip

Kuznetsov states the following conjecture~\cite{kuz:4fold}, which we will try to motivate and explore in this section.

\begin{conjecture}[Kuznetsov]\label{conj:kuzn}
A cubic fourfold $X$ is rational if and only if there is a K3 surface $S$ and an equivalence
$\Db(S) \simeq \cat{A}_X$.
\end{conjecture}

We notice that Conjecture~\ref{conj:kuzn} is stronger than Question~\ref{quest:repincodim2}.
Moreover, Addington and Thomas have shown~\cite{at12} that the existence of the equivalence requested in Conjecture
\ref{conj:kuzn} is equivalent to asking $X$ to have an associated K3 surface (in the sense of \S\ref{subs:cubic-4folds}), at least if $X$ is general
on its Noether--Lefschetz divisor.

\begin{theorem}[Addington--Thomas~\cite{at12}]\label{thm:add-thomas}
Let $X$ be a smooth cubic fourfold. If there exists a K3 surface $S$ and an equivalence
$\Db(S) \simeq \cat{A}_X$, then $X$ is special and has an associated K3 surface.
Conversely, if $X$ is special and has an associated K3 surface, and if it is general on some
$\cC_d$, then there exists a K3 surface $S$ and an equivalence $\Db(S) \simeq \cat{A}_X$.
\end{theorem}

The above result can be stated in terms of $d$, since, as recalled in Theorem~\ref{thm:hassett-all},
being special with an associated K3 surface is equivalent to lie on $\cC_d$ for $d>6$ not
divisible by $4$, $9$, or any odd prime $p$ which is not $2$ modulo $3$.

Notice that Theorem~\ref{thm:add-thomas} tells that, in
Hodge-theoretic terms, Conjecture \ref{conj:kuzn} could be phrased as
``$X$ is rational if and only if it has an associated K3''. As we have
seen at the end of \S\ref{subs:cubic-4folds}, this seems to be the
most ``rational'' expectation for such
varieties. Conjecture~\ref{conj:kuzn} seems then, so far, out of
reasonable reach. But we want to give in this last part of this
section a quick idea of the very interesting interplay and new
informations that one can extract analyzing this new point of view.

First of all, let us briefly sketch how to prove Theorem~\ref{thm:add-thomas}. 
The main idea is to have a categorical way to reconstruct the Hodge lattice. This is done by
considering the topological $K$-theory $K_0(X)_{top}$, with the intersection pairing $\chi$, given by the Euler characteristic.
Using the semiorthogonal decomposition
$$\Db(X) = \sod{\cat{A}_X, \ko, \ko(1),\ko(2)},$$
it is possible to split off $K_0(\cat{A}_X)_{top} \subset K_0(X)_{top}$ as the $\chi$-semiorthogonal complement to the classes $[\ko(i)]$,
for $i=0,1,2$. If $\Db(S)\simeq \cat{A}_X$ is an equivalence, it is then of Fourier--Mukai type and gives
an isomorphism $K_0(S)_{top} \simeq K_0(\cat{A}_X)_{top}$ of $\ZZ$-modules respecting the Euler pairing.

The first observation is obtained using the fact that the Chern character is integral for $S$, so that $K_0(S)_{top} \otimes \CC$
with the pairing $\chi$ has a Hodge structure of weight $2$ induced by the so-called Mukai lattice structure on cohomology.
This is (up to identifying $K_0$ and the cohomology via the Chern character) the following Hodge structure $\widetilde{H}^{p,q}(S)$:
$$\begin{array}{rl}
\widetilde{H}^{2,0} (S) &= H^{2,0}(S)\\
\widetilde{H}^{1,1} (S) &= H^{2,2}(S) \oplus H^{1,1}(S) \oplus H^{0,0}(S)\\
\widetilde{H}^{0,2} (S) &= H^{0,2}(S).
  \end{array}$$
For the cubic fourfold $X$, Addington and Thomas define the Mukai lattice as follows:
$$\begin{array}{rl}
\widetilde{H}^{2,0} (X) &= H^{3,1}(X)\\
\widetilde{H}^{1,1} (X) &= H^{4,4} (X) \oplus H^{3,3}(X) \oplus H^{2,2}(X) \oplus H^{1,1}(X) \oplus H^{0,0}(X)\\
\widetilde{H}^{0,2} (X) &= H^{1,3}(X),
  \end{array}$$
and obtain the corresponding weight $2$ Hodge structure on $K_0(\cat{A}_X)_{top}$ (up to identifying $K_0$ and the cohomology via the Chern character).
It follows then that if $\Db(S) \simeq
\cat{A}_X$, there is a Hodge isometry $K_0(S)_{top} \simeq K_0(\cat{A}_X)_{top}$.

For the general $X$, the numerical properties of $K_0(\cat{A}_X)_{top}$ are explored, and the lattice $(K_0(\cat{A}_X)_{top}, \chi)$
is related to the Hodge lattice $H^4(X,\ZZ)$. A particularly interesting result states that $X$ has an associated
K3 surface if and only if the Mukai lattice on the numerical $K_0(\cat{A}_X)_{num}$\footnote{The numerical $K_0$ is obtained
by taking the quotient of the algebraic $K_0$ by the kernel of the Euler form.} contains a hyperbolic plane
\cite[Thm. 3.1]{at12}.
Moreover, one can characterize classes of skyscraper sheaves of points $[\ko_x]$ and $[\ko_y]$ inside $K_0(S)_{top}$, purely
using their behavior under the Euler pairing: the sublattice they generate is a hyperbolic lattice. Hence, it follows
that if $X$ has no associated K3, such classes do not exist and hence there cannot be any equivalence $\Db(S) \simeq \cat{A}_X$.
A similar result for cubic fourfolds containing a plane with two dimensional group of algebraic $2$-cycles was obtained
by Kuznetsov~\cite{kuz:4fold}.

On the other hand, consider a divisor $\cC_d$ where $d$ is such that
$X$ has an associated K3 surface.  Then consider the intersection
$\cC_d \cap \cC_8$ with the locus of cubics containing a plane.  As we
will see in Example~\ref{ex:DX-for-cub-with-a-plane} (see also
Table~\ref{table:fano-sods}), in this case there is a degree $2$ K3
surface $S$ with a Brauer class $\alpha$ and an equivalence $\cat{A}_X
\simeq \Db(S,\alpha)$. Moreover, one can find an $X$ on $\cC_d \cap
\cC_8$ such that $\alpha$ vanishes. This gives a K3 surface $S$ of
degree 2 and an equivalence $\cat{A}_X \simeq \Db(S)$ in this
particular case. Now, though one would like to deform this K3 surface,
it is not the right thing to do. Instead, one has to construct an
appropriate K3 surface $S'$ of degree $d$ as a moduli space of vector
bundles on $S$, so that $\Db(S') \simeq \Db(S) \simeq \cat{A}_X$.
Then the existence of a K3 surface $S'$ and of the equivalence
$\cat{A}_X \simeq \Db(S')$ for the general $X$ in $\cC_d$ is obtained
by a degeneration method.

\smallskip

We turn now to explicit cases supporting
Conjecture~\ref{conj:kuzn}. Indeed, in all the cases where a cubic
fourfold is known to be rational, there is an explicit realization of
the K3 surface $S$ and of the equivalence $\Db(S) \simeq \cat{A}_X$.

\begin{example}[Kuznetsov~\cite{kuz:4fold}]\label{ex:DX-for-cub-with-a-plane}
Let $X$ be a cubic fourfold containing a plane $P$. As explained in Example~\ref{ex:cubicwithplane},
the blow-up $\widetilde{X} \to X$ of the plane $P$ has a structure of quadric surface bundle
$\widetilde{X} \to \PP^2$.

Recall from Example~\ref{ex:semiortho-for-MFS} that in this case, we have
an identification $\cat{A}_{\widetilde{X}/\PP^2} \simeq \Db(\PP^2,\kc_0)$ (see~\cite{kuznetquadrics}).
On the other hand, there is a semiorthogonal decomposition induced by the blow-up 
$\widetilde{X} \to X$. Comparing the two decompositions via explicit mutations, gives an equivalence
$\cat{A}_X \simeq \Db(\PP^2,\kc_0)$.

Suppose that the degeneration divisor of the quadric fibration $\widetilde{X} \to \PP^2$ is
smooth, and consider the double cover $S \to \PP^2$ ramified along it. Kuznetsov shows that
$\kc_0$ lifts to a sheaf of Clifford algebras with Brauer class $\alpha$ in $\Br(S)$. It follows
that $\cat{A}_X \simeq \Db(S,\alpha)$.

On the other hand, the classical theory of quadratic forms says that the class $\alpha$ is trivial
if and only if $\widetilde{X} \to \PP^2$ has an odd section. As described in Example~\ref{ex:cubicwithplane}
this is a necessary condition for rationality.

Secondly, Kuznetsov shows that if $S$ has Picard rank one, then $\alpha$ is nontrivial and $\Db(S,\alpha)$
cannot be equivalent to any $\Db(S')$, for $S'$ a K3 surface (this can be seen as a special case of
Addington--Thomas result). In particular, one should expect that cubics with a plane with such an $S$
are not rational. On the other hand, notice that there exist rational cubics with a plane, such that $\alpha$ is not
trivial but $\Db(S,\alpha) \simeq \Db(S')$ for some other K3 surface $S'$. Example of such cubics are constructed
in~\cite{ABBV}.
\end{example}

\begin{example}[Kuznetsov~\cite{kuz:4fold}]\label{ex:DX-for-pfaff-cub}
Let $X$ be a Pfaffian cubic fourfold. As explained in Example~\ref{ex:c14}, there is a classical duality
construction providing a degree 14 K3 surface $S$ associated to $X$. In this case, the powerful theory
of Homological Projective Duality allows to show that $\cat{A}_X \simeq \Db(S)$.
\end{example}

\begin{example}[Kuznetsov~\cite{kuz:4fold}]\label{ex:DX-for-cub-with-a-node}
Let $X$ be a cubic fourfold with a single node $x$. The projection $\PP^5 \dashrightarrow \PP^4$
from the point $x$ gives a rational parametrization $X \dashrightarrow \PP^4$. The resolution
of the latter map is obtained by blowing-up $x$ and is realized as a blow-up $\widetilde{X} \to
\PP^4$ of a degree 6 K3 surface $S$, obtained as a complete intersection of a cubic and a quadric.
Then one can show that there is a crepant resolution of singularities $\widetilde{\cat{A}}_X$ and
an equivalence $\widetilde{\cat{A}}_X \simeq \Db(S)$.
\end{example}

\begin{example}[\cite{berna-bolo-faenzi}]\label{ex:DX-for-det-cub}
Let $X$ be a determinantal cubic fourfold. In this case, Homological Projective Duality
can be used to show that there is a crepant resolution of singularities $\widetilde{\cat{A}}_X$ 
which is generated by six exceptional objects. Roughly speaking, one should think of the latter
as a crepant resolution of a degeneration of the K3 surface from the previous case.
\end{example}

We can then rephrase Example~\ref{ex:known-rat-cub} listing all known rational cubic fourfolds in terms of Conjecture 
\ref{conj:kuzn}.

\begin{example}
Let $X$ be a cubic fourfold. If either
\begin{itemize}
 \item[2,6)] $X$ is singular, e.g. $X \in \cC_6$ has a single node or $X \in \cC_2$ is determinantal; or
 \item[8)] $X$ contains a plane $P$, so that $X \in \cC_8$, and the associated quadric surface fibration $\widetilde{X} \to \PP^2$
 (see Example~\ref{ex:cubicwithplane}) admits an odd section~\cite{hassett:rational_cubic}; or
 \item[14)] $X$ is Pfaffian, so that $X \in \cC_{14}$ \cite{beauville-donagi};
\end{itemize}
then $X$ is rational and (a crepant categorical resolution of) $\cat{A}_X$ is equivalent to (a crepant categorical resolution of) $\Db(S)$,
for some K3 surface $S$.
\end{example}

\subsection{Other fourfolds}
Let us quickly conclude with another example, described in~\cite{auel-berna-bolo}, of fourfolds whose rationality
or nonrationality is conjecturally related to categorical representability via 
an explicit semiorthogonal decomposition.

Let $X$ be a fourfold with a Mori Fiber Space structure $X \to \PP^1$ such that
the fibers are complete intersections of two quadrics. This means that there is
a projective bundle $\PP(E) \to \PP^1$ with $\rk(E)=5$, and two line bundle
valued non-degenerate quadratic forms $q_i : L_i \to S_2(E)$, such that $X \subset \PP(E)$
is the complete intersection of the two quadric fibrations $Q_i \to \PP^1$
given by the forms $q_i$.
Moreover, working over $\CC$, we know that $X \to \PP^1$ has a smooth section (see, e.g.
\cite[Lemma 1.9.3]{auel-berna-bolo} for a direct argument, or use~\cite{campana_peternell_pukhlikov} or~\cite{graber_harris_starr}).

Setting $F:=L_1 \oplus L_2$, one has that the linear span $q$ of $q_i$ and $q_2$
gives a quadric fibration $Q \to \PP(F)$ of relative dimension 4 over a Hirzebruch surface $\PP(F)$.
The smooth section of $X \to \PP^1$ gives a smooth section of $Q \to \PP(F)$, along which
we can perform reduction by hyperbolic splitting. This means that we can split off the form $q$
a hyperbolic lattice, whose complement gives a quadric fibration $Q' \to \PP(F)$ of dimension
two less than $Q$. That is, $Q' \to \PP(F)$ is a quadric surface fibration.

Homological Projective Duality and Morita equivalence of Clifford algebras under hyperbolic splitting
(see~\cite{auel-berna-bolo} for details) show that $\cat{A}_{X/\PP^1} \simeq \cat{A}_{Q'/\PP(F)}$.
The latter is known to be equivalent to $\Db(\PP(F),\kc_0)$, where $\kc_0$ is the sheaf
of even Clifford algebras of the quadric surface fibration $Q' \to \PP(F)$.
Finally, assuming the degeneration divisor of $Q' \to \PP(F)$ to be smooth, we have a smooth double
cover $S \to \PP(F)$ and a Brauer class $\alpha$ in $\Br(S)$, such that $\cat{A}_{X/\PP^1} \simeq \Db(S,\alpha)$.
Notice that the composition $S \to \PP(F) \to \PP^1$ endows $S$ with a fibration into genus 2 curves,
since there are 6 degenerate quadrics for each fiber of $\PP(F) \to \PP^1$.

The following conjecture is inspired by
Question~\ref{quest:repincodim2} and Kuznetsov's conjecture
\ref{conj:kuzn} for cubic fourfolds.

\begin{conjecture}[\cite{auel-berna-bolo}, Conj. 5.1.2]
Let $X \to \PP^1$ be a fibration in complete intersections of two
four-dimensional quadrics.
\begin{itemize}
\item {\bf Weak version.}  
The fourfold $X$ is rational if and only if $\coRepcat(X) \geq 2$.

\item {\bf Strong version.}  
The fourfold $X$ is rational if and only
if $\Rep\cat{A}_{X/\PP^1} \leq 2$.
\end{itemize}
\end{conjecture}

\end{document}